\documentclass[10pt,a4paper]{amsart}
\usepackage[shortlabels]{enumitem}
\usepackage{amsfonts}
\usepackage{amssymb}
\usepackage{amsthm}
\usepackage{amsmath}
\usepackage{amscd}
\usepackage{t1enc}
\usepackage{lmodern}
\usepackage[mathscr]{eucal}

\usepackage{graphicx}
\usepackage{graphics}
\usepackage{pict2e}
\usepackage{epic}
\numberwithin{equation}{section}
\usepackage[margin=2.9cm]{geometry}
\usepackage{epstopdf} 
\usepackage{bbm}
\usepackage{tikz-cd}
\usetikzlibrary{tikzmark,quotes}
\usepackage{cite}
\usepackage[linktocpage=true,colorlinks=true,pagebackref,hyperindex]{hyperref}
\usepackage{mathrsfs}
\usepackage{xcolor}
\usepackage{printlen}
\usepackage{eqparbox}
\usepackage[numbers]{natbib}

\setlist[enumerate]{leftmargin=2pc}
\setlist[itemize]{leftmargin=2pc}

\makeatletter
\def\l@subsection{\@tocline{2}{0pt}{2pc}{5pc}{}}
\makeatother

\theoremstyle{plain}
\newtheorem{thm}{Theorem}[section]
\newtheorem{lemma}[thm]{Lemma}
\newtheorem{cor}[thm]{Corollary}
\newtheorem{prop}[thm]{Proposition}

 \theoremstyle{definition}
\newtheorem{Def}[thm]{Definition}
\newtheorem{conj}[thm]{Conjecture}
\newtheorem{rmk}[thm]{Remark}
\newtheorem{?}[thm]{Problem}

\newtheorem{cons}[thm]{Construction}
\newtheorem{ass}[thm]{Assumption}

\newcommand{\zz}{\mathbb{Z}}

\newcommand{\qq}{\mathbb{Q}}

\newcommand{\id}{\text{id}}
\newcommand{\res}{\big|}

\newcommand{\Gdisp}{G\text{-}\textbf{\textup{Disp}}_\mu^W}

\begin{document}
\begin{abstract} %shortened abstract for IMRN
We develop a Tannakian framework for group-theoretic analogs of displays, originally introduced by B\"ultel and Pappas, and further studied by Lau. We use this framework to define Rapoport-Zink functors associated to triples $(G,\{\mu\},[b])$, where $G$ is a flat affine group scheme over $\zz_p$ and $\mu$ is a cocharacter of $G$ defined over a finite unramified extension of $\zz_p$. We prove these functors give a quotient stack presented by Witt vector loop groups, thereby showing our definition generalizes the group-theoretic definition of Rapoport-Zink spaces given by B\"ultel and Pappas.  As an application, we prove a special case of a conjecture of B\"ultel and Pappas by showing their definition coincides with that of Rapoport and Zink in the case of unramified EL-type local Shimura data.
\end{abstract}

\title{A Tannakian framework for $G$-displays  and Rapoport-Zink spaces}

\author{Patrick Daniels}
\thanks{Research partially supported by NSF DMS-1406787}
\address{Department of Mathematics, University of Maryland, College Park, MD 20742}

%Science \\ H-1117 Budapest
%\\ P\'{a}zm\'{a}ny P\'{e}ter s\'{e}t\'{a}ny 1/C \\ Hungary} 
%
\email{pdaniel1@math.umd.edu}
%
% \subjclass[2010]{Primary: 05C??. Secondary: 05C??}
%
% \keywords{sample paper} 

%\begin{abstract} 
%\end{abstract}

\maketitle
\tableofcontents

\
\section{Introduction}
Let $p$ be a fixed prime. The theory of displays, developed by Zink in \cite{Zink2002}, is a generalization of Dieudonn\'{e} theory for formal $p$-divisible groups. By results of Zink and Lau \cite{Lau2008}, formal $p$-divisible groups over $p$-adically complete and separated $\zz_p$-algebras are classified by their associated displays. When $G$ is a reductive group over $\zz_p$ and $\mu$ is a minuscule cocharacter of $G$, B\"{u}ltel and Pappas \cite{BP2017} defined group-theoretic analogs of displays, called $(G,\mu)$-displays, with the intention of using these objects to stand in for $p$-divisible groups with $G$-structure in a general definition of Rapoport-Zink spaces. When $G = \text{GL}_n$ and $\mu$ is the cocharacter $t \mapsto (1^{(d)}, t^{(n-d)})$, the category of $(G,\mu)$-displays over a $p$-nilpotent $\zz_p$-algebra $R$ is equivalent to the category of Zink displays over $R$. In this way the theory of $(G,\mu)$-displays naturally generalizes Zink's theory of displays.

In a different direction, Langer and Zink \cite{LZ2007} defined a category of higher displays, which contains Zink's displays as a full subcategory. Recently, Lau \cite{Lau2018} reformulated the framework for higher displays in such a way as to allow for a uniform treatment of a number of display-like objects, including Dieudonn\'e displays as in \cite{Zink2001}, $F$-zips as in \cite{MW2004}, and Frobenius gauges as in \cite{FJ2013}. Further, Lau used his framework to give a general definition of $G$-displays of type $\mu$ for an arbitrary smooth group scheme $G$ and cocharacter $\mu$. When $G$ is reductive and $\mu$ is minuscule, the category of $G$-displays of type $\mu$ coincides with the category of $(G,\mu)$-displays defined by B\"{u}ltel and Pappas. Hence Lau's work can be seen as a way to link Langer and Zink's generalization of the theory of displays with that of B\"{u}ltel and Pappas. 

In this paper we offer another way to relate these two theories by developing a Tannakian framework for $(G,\mu)$-displays. More precisely, we define a Tannakian $G$-display to be an exact tensor functor from the category of representations of $G$ on finite free $\zz_p$-modules to the category of higher displays. That such a definition is reasonable is suggested by the following general mantra. Let $G$ be a group, and let $\textbf{Cat}$ be an exact tensor category. Then an object in $\textbf{Cat}$ endowed with $G$-structure should manifest itself in two ways: as a torsor for $G$ (or for some closely related group) perhaps with some additional structures, and as an exact tensor functor from the category of representations of $G$ to $\textbf{Cat}$. The relation between these two interpretations is well-known when $\textbf{Cat}$ is the category of vector bundles over a scheme $S$, cf. \cite{Nori1976}. This principle has been notably applied in the case where $\textbf{Cat}$ is the category of isocrystals over a field $k$ in \cite{RR1996}, and where $\textbf{Cat}$ is the category of $F$-zips over a field $k$ in \cite{PWZ2015}. In our situation, Lau's theory offers the torsor-theoretic definition of $G$-displays, and we contribute a Tannakian version of the theory. 

When $G$ is a classical group, there is often a third interpretation: An object in \textbf{Cat} is said to be endowed with $G$-structure if it is equipped with some additional structures corresponding to the group $G$, such as a bilinear form or actions on the object by a semisimple algebra. The Tannakian framework for objects with $G$-structure is closely related to this third interpretation. Indeed, given an exact tensor functor $\mathscr{F}: \textbf{Rep}(G) \to \textbf{Cat}$, we can obtain an object in \textbf{Cat} with additional structures by evaluating $\mathscr{F}$ on the faithful representation which gives the embedding of $G$ into some $\text{GL}_n$. In the case where $G$ is an orthogonal group, Lau applies this principle to prove that $G$-displays correspond to displays equipped with a perfect symmetric bilinear form, cf. \cite[Proposition 5.5.2]{Lau2018}. This can be seen as a special case of the Tannakian framework we develop in this paper.

Since our definition of Tannakian $(G,\mu)$-displays extends the definition of $(G,\mu)$-displays in \cite{BP2017}, we can use it to define a natural generalization of the Rapoport-Zink functor defined there. Our Tannakian framework proves advantageous in this regard, because it brings the theory closer to Zink's original theory of displays, and therefore to the theory of $p$-divisible groups. In particular, we prove that B\"ultel and Pappas's Rapoport-Zink functor is representable by the classical Rapoport-Zink space in the case where the data of definition is of EL-type. This proves a conjecture of B\"ultel and Pappas in this special case.

Let us describe our results in more detail. Let $R$ be a $p$-adic ring, and denote by $W(R)$ the ring of $p$-typical Witt vectors for $R$, which is equipped with Frobenius $f$ and Verschiebung $v$. Then, following Lau \cite{Lau2018}, we define a graded variant of the Witt ring $W(R)$, which we denote by $W(R)^\oplus$ (cf. Definition \ref{def-wittframe}). This ring is equipped with homomorphisms $\sigma, \tau: W(R)^\oplus \to W(R)$ such that the triple $\underline{W}(R) = (W(R)^\oplus, \sigma, \tau)$ becomes a higher frame in the sense of \textit{loc. cit.} Pairs $(M,F)$ consisting of a finite projective graded $W(R)^\oplus$-module $M$ and a $\sigma$-linear bijection of $W(R)^\oplus$-modules $F: M \to M\otimes_{W(R)^\oplus,\tau} W(R)$ are called displays over $\underline{W}(R)$, cf. Definition \ref{def-display}. The categories of finite projective graded $W(R)^\oplus$-modules and of displays over $\underline{W}(R)$ are exact rigid tensor categories.

Let $G$ be a flat affine group scheme of finite type over $\zz_p$, let $k_0$ be a finite extension of $\mathbb{F}_p$, and let $\mu: \mathbb{G}_{m, W(k_0)} \to G_{W(k_0)}$ be a cocharacter of $G$ defined over $W(k_0)$. Lau associates to $G$ and $\mu$ a group, called the display group, as follows. Let $G = \text{Spec}$ $A$. Then $\mathbb{G}_{m,W(k_0)}$ acts on $A$ via conjugation and therefore determines a $\zz$-grading on $A$. The display group $L^+_\mu G(R)$ (denoted $G(W(R)^\oplus)_\mu$ in \cite{Lau2018}) is the subset of $G(W(R)^\oplus)$ consisting of homomorphisms $A \to W(R)^\oplus$ which preserve the respective gradings. Our first result is to give an interpretation of this group as the collection of tensor automorphisms of a certain fiber functor. This result can be seen as an analog of Tannakian duality in this setting.

Denote by $\textbf{PGrMod}(W(R)^\oplus)$ the category of finite projective graded $W(R)^\oplus$-modules. Associated to $G$ and $\mu$ we define a canonical graded fiber functor for every $p$-adic ring $R$
\begin{align*}
	\mathscr{C}_{\mu, R}: \textbf{Rep}_{\zz_p} (G) \to \textbf{PGrMod}(W(R)^\oplus)
\end{align*}
by assigning to $(V,\rho)$ the $W(R)^\oplus$-module $V_{W(k_0)} \otimes_{W(k_0)} W(R)^\oplus$, where $V_{W(k_0)}$ is endowed with a natural $\zz$-grading via the action of $\mathbb{G}_{m,W(k_0)}$. Define $\text{Aut}^\otimes(\mathscr{C}_{\mu,R})$ to be the collection of tensor automorphisms of this functor. As $R$ varies, this defines an fpqc sheaf in groups, denoted \underline{Aut}$^\otimes(\mathscr{C}_{\mu})$. For any $h \in L^+_\mu G(R)$, the collection $\{\rho(h)\}_{(V,\rho)}$, where $(V,\rho)$ varies over all representations of $G$ on finite free $\zz_p$-modules, defines an element of $\text{Aut}^\otimes(\mathscr{C}_{\mu,R})$. 
\begin{thm}\label{thm-introisom}
	The association $h \mapsto \{\rho\textup{(}h\textup{)}\}$ defines an isomorphism of fpqc sheaves of groups
	\begin{align*}
		L^+_\mu G \xrightarrow{\sim} \textup{\underline{Aut}}^\otimes\textup{(}\mathscr{C}_{\mu}\textup{)}.
	\end{align*}
\end{thm}
\iffalse In particular, $\textup{\underline{Aut}}^\otimes(\mathscr{C}_\mu)$ is representable by an affine group scheme over $W(k_0)$. A somewhat surprising corollary of this theorem is that, if $G$ is smooth over $\zz_p$, then $\textup{\underline{Aut}}^\otimes(\mathscr{C}_\mu)$ is flat and formally smooth over $W(k_0)$, cf. Lemma \ref{lem-prosmooth}. This is not apparent from the definition of $\textup{\underline{Aut}}^\otimes(\mathscr{C}_\mu)$.\fi

Let $\textbf{Disp}(\underline{W}(R))$ be the category of displays over the frame $\underline{W}(R)$. We define a Tannakian $G$-display to be an exact tensor functor
\begin{align*}
	\mathscr{D}: \textbf{Rep}_{\zz_p}(G) \to \textbf{Disp}(\underline{W}(R)).
\end{align*}
We say $\mathscr{D}$ is a Tannakian $(G,\mu)$-display if the underlying functor to the category of finite projective $W(R)^\oplus$-modules is fpqc-locally isomorphic to $\mathscr{C}_{\mu,R}$. Our second main result is the connection between this definition and that of Lau \cite{Lau2018}. Let us recall his definition of $G$-displays of type $\mu$.

Denote by $L^+G$ the Witt loop group, i.e. the functor on $p$-nilpotent $W(k_0)$-algebras $R \mapsto G(W(R))$. Both the display group $L^+_\mu G$ and the Witt loop group $L^+G$ are representable by \cite[Lemma 5.4.1]{Lau2018}. The ring homomorphisms $\sigma$ and $\tau$ induce group homomorphisms $\sigma, \tau: L^+_\mu G \to L^+G$, and the display group acts naturally upon the Witt loop group via
\begin{align}\label{eq-introaction}
	L^+G(R) \times L^+_\mu G (R) \mapsto L^+G(R), \ (g,h) \mapsto \tau(h)^{-1} \cdot g \cdot \sigma(h).
\end{align}
Lau defines the stack of $G$-displays of type $\mu$ to be the fpqc-quotient stack $[L^+G / L^+_\mu G]$ with respect to this action. Explicitly, a $G$-display of type $\mu$ over a $p$-nilpotent $W(k_0)$-algebra $R$ can be interpreted as a pair $(Q,\alpha)$, where $Q$ is an $L^+_\mu G$-torsor over $R$, and $\alpha: Q \to L^+G$ is a map which is $L^+_\mu G$-equivariant with respect to the action (\ref{eq-introaction}).

Now we can naturally associate a $G$-display of type $\mu$ to any Tannakian $(G,\mu)$-display $\mathscr{D}$ over a $p$-nilpotent $W(k)$-algebra $R$. Denote by $\upsilon_R$ the forgetful functor from the category of displays over $\underline{W}(R)$ to the category of finite projective graded $W(R)^\oplus$-modules. Then the fpqc sheaf in groups 
\begin{align*}
	Q_{\mathscr{D}}:= \underline{\text{Isom}}^\otimes(\mathscr{C}_{\mu,R}, \upsilon_R \circ \mathscr{D})
\end{align*}
consisting of isomorphisms of tensor functors between $\mathscr{C}_{\mu,R}$ and $\upsilon_R\circ \mathscr{D}$ is naturally an $L^+_\mu G$-torsor. For any representation $(V,\rho)$, write $\mathscr{D}(V,\rho) = (M(\rho),F(\rho))$ for the corresponding display over $\underline{W}(R)$. Given an isomorphism of tensor functors $\lambda: \mathscr{C}_{\mu,R'} \xrightarrow{\sim} \upsilon_{R'} \circ \mathscr{D}_{R'}$ defined over an $R$-algebra $R'$, the collection $\{F(\rho)\}_{(V,\rho)}$ determines an element $\alpha_\mathscr{D}(\lambda)$ of $L^+G(R')$. One checks that the assignment $\lambda \mapsto \alpha_{\mathscr{D}}(\lambda)$ is $L^+_\mu G$-equivariant, so the pair $(Q_\mathscr{D},\alpha_\mathscr{D})$ determines a $G$-display of type $\mu$. 

\begin{thm}\label{thm-introequiv}
	The association $\mathscr{D} \mapsto \textup{(}Q_{\mathscr{D}},\alpha_\mathscr{D}\textup{)}$ defines an equivalence of categories between Tannakian $(G,\mu)$-displays over $R$ and $G$-displays of type $\mu$ over $\underline{W}(R)$.
\end{thm}

This equivalence follows from Theorem \ref{thm-equiv} below. We remark that the analogous theorem should hold when one considers frames other than the Witt frame $\underline{W}(R)$. However, for applications to Rapoport-Zink spaces we are primarily interested in displays over the Witt frame, so we felt it sufficient to stick to this case throughout the paper. The corresponding framework for more general frames will be addressed in a future work.

When $G$ is reductive and $\mu$ is minuscule, the category of $G$-displays of type $\mu$ over $\underline{W}(R)$ is equivalent to the category of $(G,\mu)$-displays over $R$ by \cite[Remark 6.3.4]{Lau2018}. Hence the same holds for the category of Tannakian $(G,\mu)$-displays over $R$. 

Finally, let us discuss the connection with Rapoport-Zink spaces. If $R$ is a $p$-adic ring, an isodisplay over $R$ is a pair $(N, \varphi)$ consisting of a finitely generated projective $W(R)[1/p]$-module $N$ and an $f$-linear automorphism $\varphi$ of $N$. By generalizing a construction in \cite{Zink2002}, we can associate to any display over $\underline{W}(R)$ an isodisplay over $R$. This association defines an exact tensor functor $\textbf{Disp}(\underline{W}(R)) \to \textbf{Isodisp}(R)$, and by composing a Tannakian $(G,\mu)$-display with this functor we obtain a functor $\textbf{Rep}_{\zz_p}(G)\to \textbf{Isodisp}(R)$, which denote $\mathscr{D}[1/p]$. Such an object is called a $G$-isodisplay. A $G$-quasi-isogeny of Tannakian $(G,\mu)$-displays is an isomorphism of their corresponding $G$-isodisplays. 

Now let $k$ be an algebraic closure of $\mathbb{F}_p$, and let $(G,\{\mu\},[b])$ be a local Shimura datum, so $G$ is a smooth affine group scheme over $\zz_p$ whose generic fiber is reductive, $\{\mu\}$ is a geometric conjugacy class of cocharacters, and $[b]$ is a $\sigma$-conjugacy class of elements in $G(W(k)[1/p])$. The triple $(G,\{\mu\},[b])$ is required to satisfy certain axioms, see Definition \ref{def-Shimdatum}. To a choice of $(\mu,b)$ satisfying some conditions (cf. Definition \ref{def-framingpair}) we associate a Tannakian $(G,\mu)$-display $\mathscr{D}_0$ over $k$. Following \cite{BP2017}, we then define $\textup{RZ}_{G,\mu,b}$ as the functor on $p$-nilpotent $W(k)$-algebras which assigns to any $R$ the set of isomorphism classes of pairs $(\mathscr{D},\iota)$, where
\begin{itemize}
	\item $\mathscr{D}$ is a Tannakian $(G,\mu)$-display over $R$,
	\item $\iota: \mathscr{D}_{R/pR} \dashrightarrow (\mathscr{D}_0)_{R/pR}$ is a $G$-quasi-isogeny.
\end{itemize}

The functor RZ$_{G,\mu,b}$ can be interpreted as the functor of isomorphism classes associated to an fpqc stack $\textbf{RZ}_{G,\mu,b}$ on the site of $p$-nilpotent $W(k)$-algebras. This stack can, in turn, be expressed explicitly as a quotient stack in terms of Witt vector loop groups. The functor $R \mapsto G(W(R)[1/p])$ is representable by an group ind-scheme $LG$ over $\zz_p$ (cf. \cite{Kreidl2014}), so we can form the fiber product $L^+G \times_{m_\mu, c_b} LG$ whose points in a $W(k)$-algebra $R$ are pairs $(U,g)$ with $U \in L^+G(R)$ and $g \in LG(R)$ such that
\begin{align*}
	g^{-1}\cdot b \cdot f(g) = U \cdot \mu^\sigma(p).
\end{align*}
From this we form the quotient stack $[(L^+G \times_{m_\mu, c_b} LG)/L^+_\mu G]$ with respect to the following action of $L^+_\mu G$, which is well-defined by Lemma \ref{lem-RZaction}: 
\begin{align*}
	(U,g)\cdot h = (\tau(h)^{-1} \cdot U \cdot \sigma(h), g \cdot \tau(h)).
\end{align*}
The following, which is Theorem \ref{prop-rzquotient} below, is a generalization of \cite[\S 4.2.3]{BP2017}.
\begin{thm}
	There is an isomorphism of stacks 
	\begin{align*}
		\textup{\textbf{RZ}}_{G,\mu,b} \xrightarrow{\sim} [(L^+G \times_{m_\mu, c_b} LG) / L^+_\mu G].
	\end{align*}
\end{thm}

It is a consequence of this theorem that when $G$ is reductive and $\mu$ is minuscule, the Rapoport-Zink functor defined above coincides with the one defined in \cite{BP2017}, cf. Proposition \ref{cor-BPcompare}. In \textit{loc. cit.} it is conjectured that, under mild assumptions, the functor $\textup{RZ}_{G,\mu,b}$ is representable by a formal scheme which is formally smooth and locally formally of finite type over Spf $W(k)$. When $G = \text{GL}_n$ and $\mu$ is the cocharacter $t \mapsto (1^{(d)}, t^{(n-d)})$, the category of $(G,\mu)$-displays over a ring $R$ coincides the category of Zink displays over $R$, so in this case representability can be proved by explicit connection with the original functor defined by Rapoport and Zink \cite{RZ1996}. This is stated in \cite{BP2017}, and we provide details in \textsection \ref{sub-gln}. Further, B\"ultel and Pappas prove that, when $(G,\{\mu\},[b])$ is of Hodge type (again with some additional assumptions), the restriction of $\textup{RZ}_{G,\mu,b}$ to Noetherian $p$-nilpotent $W(k)$-algebras is representable by a formal scheme with the desired properties. 

The Tannakian framework we develop in this paper allows us to compare the functor $\textup{RZ}_{G,\mu,b}$ with that of Rapoport and Zink outside of the case $G = \text{GL}_n$. In particular, we consider the case of unramified EL-type local Shimura data. In particular, let $B$ be a semisimple $\qq_p$-algebra whose simple factors are all matrix algebras over unramified extensions of $\qq_p$, let $\mathcal{O}_B$ be a maximal order in $B$, let $\Lambda$ be a finite free $\zz_p$-module equipped with an action of $\mathcal{O}_B$, and define $G =  \text{GL}_{\mathcal{O}_B}(\Lambda)$. Such a tuple $\textbf{D} = (B, \mathcal{O}_B, \Lambda)$ is called an unramified integral EL-type datum. There is a natural embedding $\eta: G \hookrightarrow \text{GL}(\Lambda)$, and if $(\mathscr{D},\iota) \in \textup{RZ}_{G,\mu,b}(R)$, then evaluating $\mathscr{D}$ on the representation $(\Lambda, \eta)$ determines a Zink display equipped with an $\mathcal{O}_B$-action. In turn, by applying the functor $BT_R$ from nilpotent Zink displays to formal $p$-divisible groups (cf.  \cite{Zink2002} and Theorem \ref{thm-BT} below), we obtain a formal $p$-divisible group $X$ over $R$ with $\mathcal{O}_B$-action. A key result which allows the comparison of $\textup{RZ}_{G,\mu,b}$ with the classical EL-type RZ-space is the following lemma, which reinterprets the Kottwitz determinant condition on Lie$(X)$ (cf. \cite[3.23(a)]{RZ1996}) as a condition on the Zink display associated to $X$:

\begin{lemma}
	Let $\underline{M} = (M,F)$ be the Zink display associated to a formal $p$-divisible group $X$, so $BT_R(\underline{M}) = X$. Then \textup{Lie}$(X)$ satisfies the determinant condition with respect to $\textbf{\textup{D}}$ if and only if $M$ is fpqc-locally isomorphic to $\Lambda \otimes_{\zz_p} W(R)^\oplus$ as a graded $\mathcal{O}_B \otimes_{\zz_p} W(R)^\oplus$-module.
\end{lemma}

We remark that the condition on $M$ in the lemma is automatic if $M$ comes from evaluating a Tannakian $(G,\mu)$-display $\mathscr{D}$ on $(\Lambda, \eta)$. As a result of this lemma and the above discussion, we obtain a map from $\textup{RZ}_{G,\mu,b}$ to the EL-type Rapoport-Zink functor defined in \cite{RZ1996}, denoted $\textup{RZ}_{\textbf{D}}(X_0)$, where $X_0$ is the $p$-divisible group corresponding to the Tannakian $(G,\mu)$-display $\mathscr{D}_0$.

\begin{thm}
	If $(G,\{\mu\},[b])$ is of unramified EL-type, and $\eta(b)$ has no slopes equal to 0, the map $\textup{RZ}_{G,\mu,b} \to \textup{RZ}_{\textbf{D}}(X_0)$ is an isomorphism. In particular, the functor $\textup{RZ}_{G,\mu,b}$ is representable by a formal scheme which is formally smooth and locally formally of finite type over \textup{Spf} $W(k)$.
\end{thm}

This proves \cite[Conjecture 4.2.1]{BP2017} in the case of EL-type local Shimura data. We remark that a similar analysis should prove the conjecture in the case where the data is of PEL-type.
\subsection{Acknowledgments} I thank my advisor, Thomas Haines, for his constant support and encouragement, and for suggesting to me that a Tannakian framework for $(G,\mu)$-displays might be possible. \iffalse\textcolor{blue}{I thank Timo Richarz for making me aware of \cite[Lemma 2.21]{Richarz2019}, which was needed in the proof of Lemma \ref{lem-prosmooth}}.\fi I also thank Eike Lau and Michael Rapoport for helpful discussions.
\subsection{Notation}
\begin{itemize}[leftmargin=*]
\item Let $p$ be a prime. A ring or abelian group will be called $p$-adic if it is complete and separated with respect to the $p$-adic topology. 

\item If $f: A \to B$ is a ring homomorphism, and $M$ is an $A$-module, we write $M^f$ for $M \otimes_{A,f} B$. If the map $f$ is understood, we sometimes also write $M_B = M\otimes_A B$. If $N$ is a $B$-module, we say a map $\varphi: M \to N$ is $f$-linear if $\varphi(am) = f(a) \varphi(m)$ for all $a \in A, m \in M$. In other words, if we denote by $N_{[f]}$ the $A$-module obtained from $N$ via restriction of scalars along $f$, then $\varphi$ is an $A$-module homomorphism $M \to N_{[f]}$. In this case we write $\varphi^\sharp$ for the linearization $M^f \to N$ given by $m \otimes b \mapsto \varphi(m)b$. We say $\varphi$ is an $f$-linear bijection if $\varphi^\sharp$ is a $B$-module isomorphism. 

\item For a $\zz_p$-algebra $\mathcal{O}$, denote by \textbf{Nilp}$_{\mathcal{O}}$ the site consisting of the category of $\mathcal{O}$-algebras in which $p$ is nilpotent, endowed with the fpqc topology. We will refer to such an $\mathcal{O}$-algebra as a $p$-nilpotent $\mathcal{O}$-algebra.

\item If $\varphi: G \to H$ is a morphism of groups in a topos and $P$ is an $G$-torsor, then $P^\varphi$ is the pushforward of $P$ to $H$, which is the $G$-torsor defined as the quotient of $P \times H$ by the action $(p, h) \mapsto (pg^{-1}, gh)$. 

\item Let $\bigoplus S_n$ be a $\zz$-graded ring. For a ring homomorphism $\varphi: \bigoplus S_n \to R$, we write $\varphi_n$ for the restriction of $\varphi$ to $S_n$. 

\item For a group scheme $G$ defined over a ring $A$, we write $\textbf{Rep}_A(G)$ for the category of representations of $G$ on finite projective $A$-modules.

\item Suppose $\textbf{Cat}$ is a tensor category, and $\mathscr{F}_1$ and $\mathscr{F}_2$ are tensor functors $\textbf{Rep}_A(G) \to \textbf{Cat}$. Then if $\lambda: \mathscr{F}_1 \to \mathscr{F}_2$ is a tensor morphism, and $(V,\rho)$ is an object in $\textbf{Rep}_A(G)$, we write $\lambda_\rho$ for the induced morphism $\mathscr{F}_1(V,\rho) \to \mathscr{F}_2(V,\rho)$ in \textbf{Cat}.
\end{itemize}
\section{Preliminaries}
%We begin by recalling the definitions and basic properties of higher displays. The main reference for this section is \cite{Lau2018}. At the end of this section we review the connection to Zink's original theory of displays.
\subsection{Higher frames}

Let us recall the formalism of higher frames, following \cite{Lau2018}.
%While we are primarily only concerned with the Witt frame (cf. Definition \ref{def-wittframe}) in this paper,
\begin{Def}\label{def-frame}
	A \textit{pre-frame} $\underline{S} = (S,\sigma,\tau)$ consists of a $\zz$-graded ring 
	\begin{align*}
		S = \bigoplus_{n\in \zz} S_n
	\end{align*}
	along with ring homomorphisms $\sigma: S \to S_0$ and $\tau:S \to S_0$. A pre-frame is a \textit{frame} if the following conditions are satisfied:
	\begin{itemize}
		\item The endomorphism $\tau_0$ of $S_0$ is the identity, and $\tau_{-n}: S_{-n} \to S_0$ is bijective for all $n \ge 1$. 
		\item The endomorphism $\sigma_0$ of $S_0$ induces the $p$-power Frobenius $s \mapsto s^p$ on $S_0 / pS_0$, and if $t$ is the unique element in $S_{-1}$ such that $\tau_{-1}(t) = 1$, then $\sigma_{-1}(t) = p$.
		\item We have $p \in \textup{Rad}(S_0)$, the Jacobson radical of $S_0$.  
	\end{itemize}
	We say $\underline{S}$ is a frame for $R = S_0/tS_1$. A morphism of pre-frames $(S,\sigma,\tau) \to (S',\sigma',\tau')$ is a morphism of graded rings $\psi:S \to S'$ such that $\sigma' \circ \psi = \psi \circ \sigma$ and $\tau' \circ \psi = \psi \circ \tau$.
\end{Def}

	As in \cite{Lau2018}, we remark that a frame is equivalent to a triple $(\bigoplus_{n \ge 0} S_n, \sigma, \{t_n\}_{n\ge 0})$ consisting of an $\mathbb{Z}_{\ge 0}$-graded ring $S_{\ge 0} = \bigoplus_{n \ge 0} S_n$, a ring homomorphism $\sigma: S_{\ge 0} \to S_0$ and a collection of maps $t_n: S_{n+1} \to S_n$ for $n \ge 0$ such that
\begin{itemize}
	\item For every $n \ge 0$, the homomorphism $t_n: S_{n+1} \to S_{n}$ is $S_{\ge 0}$-linear.
	\item The endomorphism $\sigma_0: S_0 \to S_0$ induces the $p$-power Frobenius $s \mapsto s^p$ on $S_0/pS_0$, and $\sigma_n(t_n(a)) = p \sigma_{n+1}(a)$ for all $a \in S_{n+1}.$
	\item We have $p \in \text{Rad}(S_0).$
\end{itemize}
The equivalence is given as follows. Define $S_{\le 0} = S_0[t]$ where $t$ is an indeterminate with degree $-1$, and let $S = S_{\le 0} \oplus \bigoplus_{n > 0} S_n$. To give $S$ a ring structure it suffices to define multiplication by $t$ on $S_n$ for $n > 0$. For this we use the maps $t_n$, i.e. if $s \in S_n$, $n > 0$, define $t \cdot s := t_{n-1}(s)$. The homomorphism $\sigma$ extends uniquely to all of $S$ by defining $\sigma(t^n \cdot s_0) = p^n\sigma_0(s_0)$ for $s_0 \in S_0$. It remains only to define $\tau: S \to S_0$. The restriction of $\tau$ to $S_0$ is necessarily given by the identity. Since multiplication by $t$ is bijective on $S_0[t]$, we can define $\tau_n: S_{-n} \to S_0$ by multiplication by $t^{-n}$ for $n > 0$. Lastly, if $s \in S_n$ for $n > 0$, then 
\begin{align*}
	\tau(s) = (t_0 \circ \cdots \circ t_{n-1})(s) = t^n \cdot s.
\end{align*}

Let us take a further moment to recall some notations and definitions concerning Witt vectors. Attached to a ring $R$ is the ring $W(R)$ of $p$-typical Witt vectors $W(R)$. Elements of $W(R)$ are tuples $(\xi_0, \xi_1, \dots) \in R^{\mathbb{\zz}_{\ge 0}}$, and the ring structure is characterized as the unique one which is functorial in $R$ and for which the maps
\begin{align*}
	w_i: W(R) \to R, \ (\xi_i)_{i \in \zz_{\ge 0}} \mapsto \xi_0^{p^i} + p \xi_1^{p^{i-1}} + \cdots + p^i \xi_i
\end{align*}
are ring homomorphisms. Additionally, the ring $W(R)$ is equipped with Frobenius and Verschiebung maps $f, v: W(R) \to W(R)$. The Verschiebung is the additive map given by shifting: $v(\xi_0, \xi_1, \dots) = (0, \xi_0, \xi_1, \dots)$, and the Frobenius is a ring homomorphism characterized by its compatibility with the maps $w_i$:
\begin{align*}
	w_i(f(x)) = w_{i+1}(x). 
\end{align*}
We will denote by $I_R$ the kernel of $w_0: W(R) \to R$. Equivalently, $I_R = v(W(R))$.
 
The following is the frame of primary interest in this paper.
 
\begin{Def} \label{def-wittframe}
	Let $R$ be a $p$-adic ring. The \textit{Witt frame} for $R$ is defined as follows (cf. \cite[Example 2.1.3]{Lau2018}). By the above remarks, it suffices to define $S_{\ge 0}$, $\sigma: S_{\ge 0} \to S_0$, and $t_n: S_{n+1} \to S_n$ for every $n>0$. Let $S_0 = W(R)$, and define $S_n = I_R$ viewed as a $W(R)$-module for $n \ge 1$. If $n,m \ge 1$, then multiplication for $S_n \times S_m \to S_{n+m}$ is given by 
	\begin{align*}
		I_R \times I_R \to I_R, \ (v(a), v(b)) \mapsto v(ab).
	\end{align*}
	The homomorphism $t_0: S_1 \to S_0$ is the inclusion of the submodule $I_R \hookrightarrow W(R)$, and for $n \ge 1$, $t_n: S_{n+1} \to S_n$ is multiplication by $p$. The endomorphism $\sigma_0$ of $S_0 = W(R)$ is the Witt vector Frobenius, $f$. For every $n \ge 1$, define $\sigma_n(v(a)) = a$ for all $v(a) \in S_n = I_R$.
	
	We will denote the graded ring $S$ for this frame by $S = W(R)^\oplus$, and write $\underline{S} = \underline{W}(R)$ to denote the frame $(W(R)^\oplus, \sigma, \tau)$. If $R\to R'$ is a ring homomorphism, then the induced map $\underline{W}(R) \to \underline{W}(R')$ is a morphism of frames. By \cite[4.1]{Lau2018} the functor $\underline{W}$ which sends $R$ to $\underline{W}(R)$ is an fpqc sheaf of frames on \textbf{Nilp}$_{\zz_p}$. 
\end{Def}
\subsection{Graded modules}
%We review definitions and properties of graded modules over graded rings, which will comprise the underlying modules for displays.
%
Let $S$ be a $\zz$-graded ring, and denote by \textbf{\textbf{GrMod}}$(S)$ the category of graded $S$-modules. If $M$ and $N$ are objects in this category, denote the morphisms between $M$ and $N$ in \textbf{GrMod}$(S)$ by Hom$_S^0(M,N)$. Then Hom$_S^0(M,N)$ is the set of $S$-module homomorphisms which preserve the gradings of $M$ and $N$, i.e. the set of $\varphi\in \text{Hom}_S(M,N)$ such that $\varphi(M_i) \subseteq N_i$. 

Let \textbf{PGrMod}$(S)$ be the full subcategory of \textbf{GrMod}$(S)$ consisting of finite projective graded $S$-modules. By \cite[Lemma 3.0.1]{Lau2018}, this is equivalent to the full subcategory of projective objects in \textbf{GrMod}$(S)$ which are finitely generated.

For reference, let us review the exact tensor structure of $\textbf{PGrMod}(S)$. Note \textbf{GrMod}$(S)$ is an abelian category, so \textbf{PGrMod}$(S)$ inherits additivity and a notion of exactness: a sequence
%	\begin{align*}
%		0 \to M' \to M \to M'' \to 0
%	\end{align*}
of finite projective graded $S$-modules is exact if and only if it is exact in \textbf{GrMod}$(S)$. The category \textbf{GrMod}$(S)$ is also endowed with a tensor product: if $M= \bigoplus_i M_i$ and $N=\bigoplus_i N_i$ are graded $S$-modules, then $M\otimes_S N$ is a graded $S$-module with graded pieces
\begin{align*}
(M \otimes_S N)_\ell = \{\sum_i m_i \otimes n_i \in M \otimes N \mid \deg(m_i)+ \deg(n_i) = \ell\}.
\end{align*}
The ring $S$ viewed as a graded module over itself is the unit object in both \textbf{GrMod}$(S)$ and \textbf{PGrMod}$(S)$. Since the tensor product of two finite projective $S$-module is again a finite projective $S$-module, \textbf{PGrMod}$(S)$ is a tensor subcategory of \textbf{GrMod}$(S)$. The dual of $M$ in \textbf{GrMod}$(S)$ is the dual $S$-module $M^\vee = \text{Hom}_S(M,S)$ with grading $(M^\vee)_i = (M_{-i})^\vee$.
\begin{lemma} \label{lem-pmod_rigid}
	The category \textup{\textbf{PGrMod}}$(S)$ is an exact rigid tensor category.
\end{lemma}
\begin{proof}
	After the above remarks, it remains only to show that \textbf{PGrMod}$(S)$ is rigid. Since every object in \textbf{PGrMod}$(S)$ admits a dual, it is enough to show that $M^{\vee\vee} \cong M$ (cf. the footnote under \cite[Definition 1.7]{Deligne1982}).  But this is clear because projectivity is preserved under taking duals and because finite projective $S$-modules are reflexive.
\end{proof}

Suppose now $\underline{S} = (S,\sigma,\tau)$ is a frame for a ring $R$. Let
\begin{align}\label{eq-nu}
	\nu: S \to S_0/tS_1 = R
\end{align}
be the natural projection $S_0 \to R$ extended by zero on $S_n$ for $n \ne 0$ (this map is called $\rho$ in \cite{Lau2018}). By considering $R$ to be a graded ring concentrated in degree zero, we can view $\nu$ as a homomorphism of graded rings. Then for any finite projective graded $S$-module $M$, the base change $\overline{L} = M \otimes_{S,\nu} R$ along $\nu$ is a finite projective graded $R$-module. Write 
\begin{align*}
	\overline{L} = \bigoplus_{i \in \zz} \overline{L}_i
\end{align*}
for the decomposition of $\overline{L}$ into its graded pieces.

Recall the rank function associated to a finite projective $R$-module is the locally constant function on Spec $R$ defined by $\mathfrak{p} \mapsto \dim_{\kappa(\mathfrak{p})}(M \otimes_R \kappa(\mathfrak{p}))$, where $\kappa(\mathfrak{p}) = R_\mathfrak{p} / \mathfrak{p}R_{\mathfrak{p}}$ is the residue field of $\mathfrak{p}$. In particular, if Spec $R$ is connected, any finite projective $R$-module has constant rank.

%\begin{Def}
%	Let $M$ be a finite projective graded $S$-module, let $\overline{L} = M \otimes_{S,\nu} R$, and write $\overline{L} = \bigoplus_{i \in \zz} \overline{L}_i$	for the decomposition of $\overline{L}$ into its graded pieces. The \textit{graded rank} of $M$ is the collection of locally constant functions $(\text{rk}(\overline{L}_i))_{i \in \zz}$. 
%\end{Def}

%We will frequently want to focus our attention on finite projective graded $S$-modules with prescribed graded rank.

\begin{Def}\label{def-type}
	Let $I = (i_1, \dots, i_n) \in \zz^n$ be a collection of integers such that $i_1 \le i_2 \le \cdots \le i_n$. We say a finite projective graded $S$-module $M$ is \textit{of type $I$} if rk$(\overline{L}_k)$ is equal to the multiplicity of $k$ in $I$ for all $k$. 
\end{Def}

For example, for any collection $I = (i_r)$, the finite free graded $S$-module $\bigoplus_r S(-i_r)$ has type $I$. If Spec $R$ is connected, every finite projective graded $S$-module has a unique type. Note that our convention on the ordering of the components of $I$ differs from that of \cite{Lau2018}.

\begin{Def}\label{def-depth}
	Let $M$ be a finite projective graded $S$-module and write $M \otimes_{S,\nu} R = \bigoplus_i \overline{L}_i$.
	\begin{enumerate}[\textup{(}i\textup{)}]
		\item The \textit{depth} of $M$, denoted $d(M)$, is the minimal integer $i$ such that $\overline{L}_i \ne 0$.
		\item The \textit{altitude} of $M$, denoted $a(M)$, is the maximal integer $i$ such that $\overline{L}_i \ne 0$.
	\end{enumerate}
\end{Def}

If a finite projective graded $S$-module $M$ is of type $I=(i_1, \dots, i_n)$, then $d(M) = i_1$ and $a(M) = i_n$. 

\begin{Def}
	A \textit{normal decomposition} for a finite projective graded $S$-module $M$ is a finite projective graded $S_0$-module $L = \bigoplus_i L_i$ such that $L \subseteq M$ and $L \otimes_{S_0} S = M$ as graded $S$-modules.
\end{Def}

It follows from \cite[Lemma 3.1.4]{Lau2018} that every finite projective graded $S$-module has a normal decomposition if every finite projective $R$-module lifts to $S_0$. This holds in particular when $S_0$ is $p$-adic. It follows that every finite projective graded $W(R)^\oplus$-module has a normal decomposition.

\begin{lemma}\label{lem-type}
	Let $M$ be a finite projective graded $S$-module, and write $M \otimes_{S,\nu} R = \bigoplus_i \overline{L}_i$. If $L=\bigoplus_i L_i$ is a normal decomposition for $M$, then $L_i = 0$ if and only if $\overline{L}_i = 0$. In particular, if $M$ is of type $I=(i_1, \dots, i_n)$, then $L_i \ne 0$ if and only if the multiplicity of $i$ in $I$ is nonzero.
\end{lemma}
\begin{proof}
	We have an isomorphism of graded $R$-modules
	\begin{align*}
	M \otimes_{S,\nu} R \cong L \otimes_{S_0} S \otimes_{S,\nu} R \cong L \otimes_{S_0} R,
	\end{align*}
	using that $S_0 \to S \xrightarrow{\nu} R$ is the natural quotient $S_0 \to R$. Hence for every $i$ we have an isomorphism of $R$-modules $L_i \otimes_{S_0} R \cong \overline{L}_i$. Clearly $\overline{L}_i = 0$ if $L_i = 0$. On the other hand, by the proof of \cite[Lemma 3.1.1]{Lau2018}, $tS_1 \subseteq \text{Rad}(S_0)$. Then $L_i \otimes_{S_0} R = 0$ implies $L_i = 0$ by Nakayama's lemma. 
\end{proof}

It follows from Lemma \ref{lem-type} that, if $M$ has normal decomposition $L = \bigoplus_i L_i$, then the depth (resp. altitude) of $M$ is equal to the minimal (resp. maximal) $i$ such that $L_i \ne 0$. For any finite projective graded $S$-module, we have a natural homomorphism of $S_0$-modules $\theta_n: M_n \to M^\tau$ given by the composition of $M_n \hookrightarrow M$ and $M\to M^\tau$. We remark that $\tau$ defines an isomorphism $S/(t-1)S \xrightarrow{\sim} S_0$; to see that $\ker \tau \subseteq (t-1)S$ it is enough to check $(\ker \tau) \cap S_{\le 0} \subseteq (t-1)S$, which is easy. Then
\begin{align*}
M^\tau \cong \varinjlim M_i,
\end{align*}
where the colimit is taken along $t: M_n \to M_{n-1}$. 

\begin{lemma}\label{lem-theta}
	Let $M$ be a finite projective graded $S$-module with normal decomposition $L$. Then $\theta_n: M_n \to M^\tau$ is an isomorphism of $S_0$-modules for all $n \le d(M)$. 
\end{lemma}
\begin{proof}
	Let $n \le d(M)$, and let $L = \bigoplus_i L_i$ be a normal decomposition for $M$. Observe  
	\begin{align*}
	M_n = \bigoplus_i L_i \otimes_{S_0} S_{n-i},
	\end{align*}
	and under $M^\tau \cong L$, the map $\theta_n$ is given by
	\begin{align*}
	\bigoplus_i L_i \otimes_{S_0} S_{n-i} \to \bigoplus_i L_i, \ (\ell_i \otimes s_{n-i})_i \mapsto (\tau(s_{n-i})\ell_i)_i. 
	\end{align*}
	By Lemma \ref{lem-type}, $L_i = 0$ for all $i < n$, so $S_{n-i}$ occurs in the above decomposition only when $n-i \le 0$. In this case we have an isomorphism of $S_0$-modules $S_{n-i} \cong S_0 \cdot t^{i-n}$, where $t \in S_{-1}$ is the unique element such that $\tau(t) = 1$. Then any $\ell_i \otimes s_{n-i} \in L_i \otimes_{S_0} S_{n-i}$ can be written as $\ell_i' \otimes t^{i-n}$, and $\theta_n$ is given by $(\ell_i' \otimes t^{i-n}) \mapsto \ell_i'.$ From this description the conclusion is clear. 
\end{proof}

%\begin{rmk}\label{rmk-freedepth}
% 	If $M = \bigoplus_{i=1}^r S(m_i)$ is a finite free graded $S$-module and $R \ne 0$, then $d(M) = \min_i\{-m_i\}$. Indeed, the proof of the previous lemma shows $d(M) \ge \min_i\{-m_i\}$, but if $n = \min_i\{-m_i\}+1$, then for some $i$ we have $n+m_i = 1$, and $\tau: S_1 \to S_0$ is not surjective because $R = S_0 / tS_1$ is nonzero. Hence $d(M) < \min_i\{m_i\}+1$, i.e. $d(M) = \min_i\{m_i\}$.
%\end{rmk}

%\begin{rmk}\label{rmk-freealt}
%	Similarly to Remark \ref{rmk-freedepth}, we see that if $M$ is a finite free graded $S$-module, then $a(M) = \max_i\{-m_i\}$
%\end{rmk}

Let us now focus on the case where $S = W(R)^\oplus$ for a $p$-adic ring $R$. Denote by \textbf{PGrMod}$^W$ the fibered category over \textbf{Nilp}$_{\zz_p}$ whose fiber over $R$ is the category \textbf{PGrMod}$(W(R)^\oplus)$. 
\begin{lemma}\label{lem-pmod_descent}
	The fibered category \textup{\textbf{PGrMod}}$^W$ is an fpqc stack over \textup{\textbf{Nilp}}$_{\zz_p}$.
\end{lemma}
\begin{proof}
	This is \cite[Lemma 4.3.2]{Lau2018}.
\end{proof}

%First we show that the type of a finite projective $W(R)^\oplus$-module can be determined fpqc-locally.
In the following we collect some properties of finite projective graded $W(R)^\oplus$-modules which can be checked fpqc-locally on Spec $R$.
\begin{lemma}\label{lem-localtype}
	Let $M$ be a finite projective $W(R)^\oplus$-module. Let $R \to R'$ be a faithfully flat homomorphism of $\zz_p$-algebras, and let $M'$ be the base change of $M$ to $W(R')^\oplus$. Then 
	\begin{enumerate}[\textup{(}i\textup{)}]
		\item $a(M) = a(M'),$ 
		\item $d(M) = d(M'),$ and
		\item $M$ is of type $I=(i_1, \dots, i_n)$ if and only if $M'$ is of type $I$.
	\end{enumerate}
\end{lemma}
\begin{proof}
	Denote by $\nu'$ the map $W(R')^\oplus \to R'$ defined as in (\ref{eq-nu}). Write $\overline{L} = \bigoplus_i \overline{L}_i$ for the base change of $M$ along $\nu$ and $\overline{L}' = \bigoplus_i \overline{L}_i'$ for the base change of $M'$ along $\nu'$.
%	The homomorphism $R \to R'$ induces an isomorphism of graded $R'$-modules
%	\begin{align*}
%		M \otimes_{W(R)^\oplus} W(R')^\oplus \otimes_{W(R')^\oplus, \nu'} R' \cong 
%			M \otimes_{W(R)^\oplus,\nu} R\otimes_R R'.
%	\end{align*}
%	The following diagram commutes:
%	\begin{center}
%		\begin{tikzcd}
%			W(R)^\oplus 
%				\arrow[r] \arrow[d, "\nu"]
%			& W(R')^\oplus
%				\arrow[d, "\nu'"]
%			\\ R
%				\arrow[r]
%			& R'
%		\end{tikzcd}
%	\end{center}
	By functoriality of the maps $\nu$ and $\nu'$ we have an isomorphism of graded $R'$-modules
	\begin{align*}
		\overline{L}\otimes_R R' \cong \overline{L}'
	\end{align*}
	and therefore an isomorphism of their graded pieces $\overline{L}_i \otimes_R R' \cong \overline{L}_i'$. Faithful flatness of $R \to R'$ implies that $\overline{L}_i = 0$ if and only if $\overline{L}_i' = 0$. This proves (i) and (ii). Part (iii) follows because the rank of a projective module is invariant under base change.
\end{proof}

The following is an analog of \cite[Lemma 30]{Zink2002}.
\begin{lemma}\label{lem-exactW}
	Let $M$ be a finite projective graded $W(R)^\oplus$-module, and let $R \to R'$ be a faithfully flat extension of $W(k_0)$-algebras. Then there is an exact sequence
	\begin{center}
		\begin{tikzcd}[column sep = small]
			0 
			\arrow[r]
			&M 	
			\arrow[r] 
			&M \otimes_{W(R)^\oplus} W(R')^\oplus
			\arrow[r, yshift=2pt] \arrow[r, yshift=-2pt] 
			&M \otimes_{W(R)^\oplus} W(R'\otimes_R R')^\oplus
			 \arrow[r, yshift = 4pt] \arrow[r] \arrow[r, yshift=-4pt]
			& \cdots
		\end{tikzcd}
	\end{center}
	where the arrows are induced by applying the functor $W$ to the usual exact sequence
	\begin{center}
		\begin{tikzcd}[column sep = small]
			0 
			\arrow[r]
			&R
			\arrow[r] 
			&R'
			\arrow[r, yshift=2pt] \arrow[r, yshift=-2pt] 
			&R'\otimes_R R'
			\arrow[r, yshift = 4pt] \arrow[r] \arrow[r, yshift=-4pt]
			& \cdots
		\end{tikzcd}
	\end{center}
\end{lemma}
\begin{proof}
	Since $M$ is a direct summand of a free $W(R)^\oplus$-module we can reduce to the case $M = W(R)^\oplus$. In that case the result follows because $R \mapsto W(R)$ and $R \mapsto I_R$ are fpqc sheaves.
\end{proof}
To close this section, we prove that exactness is a property of finite projective $W(R)^\oplus$-modules which is fpqc local in $R$.
\begin{lemma}\label{lem-sequence}
	Let $M$, $N$, and $P$ be finite projective $W(R)^\oplus$-modules equipped with $W(R)^\oplus$-module homomorphisms $N\to M \to P$. The following are equivalent:
	\begin{enumerate}[\textup{(}i\textup{)}]
		\item The sequence
		\begin{align*}
			0 \to N \to M \to P \to 0
		\end{align*}
		is exact.
		\item For some faithfully flat $\zz_p$-algebra homomorphism $R \to R'$ the sequence 
		\begin{align*}
			0 \to N_{W(R')^\oplus} \to M_{W(R')^\oplus} \to P_{W(R')^\oplus} \to 0
		\end{align*} 
		is exact.
		\item For every faithfully flat $\zz_p$-algebra homomorphism $R \to R'$ the sequence in (ii) is exact. 
	\end{enumerate}
\end{lemma}
\begin{proof}
	If the sequence in (i) is exact then it is split exact, so it will remain exact after tensoring by any extension. Then (i) implies (iii). Obviously (iii) implies (ii), so it remains only to show (ii) implies (i).
	
	Suppose the sequence in (ii) is exact for some faithfully flat extension $R \to R'$. Let us write $R''$ for $R' \otimes_R R'$, Consider the following commutative diagram:
	\begin{center}
		\begin{tikzcd}
			& 0 
				\arrow[d] 
			& 0 
				\arrow[d]
			& 0 
				\arrow[d]
			&
			\\ 0
				\arrow[r]
			& N 
				\arrow[r] \arrow[d]
			& M
				\arrow[r] \arrow[d]
			& P
				\arrow[r] \arrow[d]
			& 0
			\\ 0
				\arrow[r]
			& N_{W(R')^\oplus}
				\arrow[r] \arrow[d]
			& M_{W(R')^\oplus}
				\arrow[r] \arrow[d]
			& P_{W(R')^\oplus}
				\arrow[r] \arrow[d]
			& 0
			\\ 0 
				\arrow[r]
			& N_{W(R'')^\oplus}
				\arrow[r]
			& M_{W(R'')^\oplus}
				\arrow[r]
			& P_{W(R'')^\oplus}
				\arrow[r]
			& 0
		\end{tikzcd}
	\end{center}
	Here the bottom map in each column is induced by the difference of the two maps $W(R')^\oplus \to W(R'\otimes_R R')^\oplus$ induced by the two canonical ring homomorphisms $R' \to R'\otimes_R R'$. Then the columns are exact by Lemma \ref{lem-exactW}. The middle row is exact by assumption and the bottom row is exact because it is obtained by tensoring the middle row over $W(R')^\oplus \to W(R'\otimes_R R')^\oplus$. Injectivity of $N \to M$ and exactness at $M$ follow immediately from the diagram.
	
	It remains to check $M \to P$ is surjective. By Nakayama's Lemma, cf. \cite[Lemma 3.1.1]{Lau2018}, it is enough to show $M \otimes_{W(R)^\oplus} R \to P \otimes_{W(R)^\oplus} R$ is surjective. But $M\otimes_{W(R)^\oplus} R'\to P \otimes_{W(R)^\oplus} R'$ is surjective, so the result follows because $R \to R'$ is faithfully flat. 
\end{proof}

\begin{rmk}\label{rmk-canvases}
	In an earlier draft of this paper, before \cite{Lau2018} was announced, we defined a category of objects called canvases, which we used to develop the subsequent theory of Tannakian $(G,\mu)$-displays. The category of canvases over $R$ is equivalent to the category of finite projective graded $W(R)^\oplus$-modules. Let us briefly explain the equivalence. A \textit{precanvas} is a collection of $W(R)$-modules $\{P_i\}_{i\in\zz}$ along with $W(R)$-module homomorphisms $\iota_i: P_{i+1} \to P_i$ and $\alpha_i: I_R\otimes_{W(R)} P_i \to P_{i+1}$ such that $\iota_i \circ \alpha_i = \alpha_{i-1} \circ (\id_{I_R} \otimes \iota_{i-1}) = \text{mult}: I_R \otimes_{W(R)} P_i \to P_i$. Given a collection $\{L_i\}_{i\in\zz}$ of finite projective $W(R)$-modules there is a standard construction of a precanvas, essentially following the construction for displays given in \cite{LZ2007}. A \textit{canvas} is a precanvas isomorphic to one resulting from this construction. Given a finite projective graded $W(R)^\oplus$-module $M$, we define a precanvas over $R$ as follows: 
	\begin{itemize}
		\item $M_i$ is the $i$th graded piece of $M$, regarded as a $W(R)$-module;
		\item $\iota_i: M_{i+1} \to M_i$ is multiplication by $t \in W(R)^\oplus$;
		\item $\alpha_i: I_R \otimes_{W(R)} M_i \to M_{i+1}$ is the action of $W(R)^\oplus_1 = I_R$ on $M_i$.
	\end{itemize}
	One checks that this construction satisfies the desired compatibilities and defines a functor from \textbf{PGrMod}$(W(R)^\oplus)$ to the category of canvases over $R$, which is an equivalence of categories since every finite projective graded $W(R)^\oplus$-module has a normal decomposition. Defining basic constructions such as duals and tensor products using canvases requires some work in this category, as it is tedious to check all the necessary compatibilities. In the framework developed in \cite{Lau2018} this work is no longer necessary because in this case the constructions follow from the well-known constructions for graded modules. Therefore we have adopted the more streamlined approach using the graded ring $W(R)^\oplus$ in this paper.
\end{rmk}

\subsection{Displays}
In this section we review the definitions and elementary properties of displays over a frame. As in the previous sections, our main reference for this section is \cite{Lau2018}. Let $\underline{S}=(S,\sigma,\tau)$ be a frame. 

\begin{Def}
	A \textit{predisplay} $\underline{M} = (M, F)$ over $\underline{S}$ consists of a graded $S$-module $M$ and a $\sigma$-linear map $F: M \to M^\tau$. 
\end{Def} 

A morphism $(M,F) \to (M',F')$ of predisplays is a homomorphism of graded $S$-modules $M \to M'$ which is compatible with the maps $F$ and $F'$. Denote the resulting category of predisplays over $\underline{S}$ by \textbf{Predisp}$(\underline{S})$. This is an abelian category, because the same is true of \textbf{GrMod}$(S)$. 

%\begin{rmk} \label{rmk-lz}
%	The category of predisplays over a ring $R$ in \cite{LZ2007} differs slightly but is related to this category. Let us make the connection precise. In loc. cit. the category of predisplays consists of collections of $W(R)$-modules $\{P_i\}_{i\in \mathbb{Z}_{\ge 0}}$ along with $W(R)$-module homomorphisms $\iota_i: P_i \to P_{i-1}$, $\alpha_i: I_R \otimes_{W(R)} P_i \to P_{i+1}$, and $F_R$-linear homomorphisms $F_i: P_i \to P_0$, along with certain compatibilities between the various maps. There is a natural generalization of this category by allowing negative indices and by taking the $F_i$ to be $F_R$-linear homomorphisms $P_i \to \varinjlim P_i$, with the colimit taken along the $\iota_i$ maps. Let us denote the resulting category of Langer-Zink predisplays over $R$ by \textbf{Predisp}$_{LZ}(R)$. Define a functor
%	\begin{align*}
%		\text{\textbf{Predisp}}(\underline{W}(R)) \to \text{\textbf{Predisp}}_{LZ}(R)
%	\end{align*}
%	as follows: if $(M, F)$ is a predisplay over the frame $\underline{W}(R)$, then define a Langer-Zink predisplay $(M_\bullet, \iota_\bullet, \alpha_\bullet, F_\bullet)$ where $(M_\bullet,\iota_\bullet,\alpha_\bullet)$ is the precanvas associated to $M$ as in Remark \ref{rmk-canvases}, and $F_i$ is the restriction of $F$ to $M_i$. One checks that these satisfy the necessary compatibilities, and that this construction defines a fully faithful functor. Hence the category of predisplays here is a full subcategory of Langer-Zink predisplays. 
%\end{rmk}

\begin{Def}\label{def-display}
	A \textit{display} over $\underline{S}$ is a predisplay $\underline{M} = (M, F)$ over $\underline{S}$ such that $M$ is a finite projective graded $S$-module and $F: M \to M^\tau$ is a $\sigma$-linear bijection.
\end{Def}

Displays over $\underline{S}$ form a full subcategory of \textbf{Predisp}$(\underline{S})$ which we will denote by \textbf{Disp}$(\underline{S})$. Note a $\sigma$-linear bijection $M \to M^\tau$ is by definition a $\sigma$-linear homomorphism whose linearization $F^\sharp: M^\sigma \to M^\tau$ is an $S_0$-module isomorphism. In this way we see that endowing a finite projective $S$-module $M$ with the structure of a display is equivalent to giving an $S_0$-module isomorphism $M^\sigma \xrightarrow{\sim} M^\tau$.

\begin{Def} Suppose $\underline{S}$ is a frame for $R$ and let $\underline{M} = (M,F)$ be a predisplay (resp. display) over $\underline{S}$.
	\begin{enumerate}[(i)]
		\item $\underline{M}$ is \textit{effective} if $d(M) \ge 0$.
		\item $\underline{M}$ is an \textit{$n$-predisplay} (resp. $n$-display) if it is effective and $a(M) = n$.
	\end{enumerate}
\end{Def}

One can check that our definition of an effective predisplay agrees with that of \cite{Lau2018}. As in the previous section, we collect here some notions regarding the exact tensor category structure of $\textbf{Disp}(\underline{S})$. Morphisms of displays over $\underline{S}$ are, in particular, morphisms of the underlying finite projective graded $S$-modules, so exactness is inherited from that category (or from \textbf{GrMod}$(S)$). The tensor product of displays is the tensor product in the category of predisplays:
\begin{align*}
(M,F) \otimes (M',F') := (M\otimes_S M', F\otimes F').
\end{align*}
Since $M \otimes_{S} M'$ is a finite projective graded $S$-module, this tensor product preserves the category of displays. The unit object is $(S,\sigma)$.

The dual of a display $\underline{M} = (M,F)$ is the display $\underline{M}^\vee = (M^\vee, F^\vee)$, where $M^\vee$ is the dual of $M$ in \textbf{PGrMod}$(S)$, and $F^\vee$ corresponds to the dual of the inverse of $F^\sharp: M^\sigma \xrightarrow{\sim} M^\tau$. It is clear that $\underline{M}$ is reflexive with respect to this notion, so the following analog of Lemma \ref{lem-pmod_rigid} is immediate.

\begin{lemma}
	The category \textup{\textbf{Disp}}$(\underline{S})$ is an exact rigid tensor category.
\end{lemma}

Displays carry a good notion of bilinear form, which characterizes the tensor product of displays. 
\begin{Def}
	Let $\underline{M} = (M, F)$, $\underline{M}' = (M', F')$, and $\underline{M}'' = (M'', F'')$ be displays over $\underline{S}$. A \textit{bilinear form} $\beta: \underline{M} \times \underline{M}'  \to \underline{M}''$ is a bilinear form of the underlying graded $S$-modules $M \times M' \to M''$ such that
	\begin{align*}
		F''(\beta(x,y)) = \beta^\tau(F(x),F'(y)),
	\end{align*}
	where $\beta^\tau: M^\tau \times (M')^\tau \to (M'')^\tau$ is the induced bilinear map of $S_0$-modules.
\end{Def}

The tensor product of $\underline{M}$ and $\underline{M}'$ is characterized by the following universal property: it admits a bilinear form $\beta_0: \underline{M} \times \underline{M}' \to \underline{M} \otimes \underline{M}'$, and any other bilinear form $\underline{M} \times \underline{M}' \to \underline{M}''$ factors uniquely through $\beta_0$.

Let us review some other useful constructions for displays. 

\begin{Def}
	Let $\underline{M} = (M,F)$ be a predisplay over a frame $\underline{S} = (S,\sigma,\tau)$, and suppose $\underline{S} \to \underline{S}'$ is a morphism to another frame $\underline{S}'= (S',\sigma', \tau')$. Then the \textit{base change} of $\underline{M}$ to $\underline{S}'$ is the predisplay $\underline{M}_{\underline{S}'} = (M \otimes_S S', F \otimes \sigma')$. 
\end{Def}

Base change defines a functor \textbf{Predisp}$(\underline{S}) \to $ \textbf{Predisp}$(\underline{S}')$ which preserves the full subcategories of displays. 
%We can also define twists of displays. Let $\underline{U}(d) = (S(-d), \sigma)$ as a display. Then $\underline{U}(d)$ is the unit display twisted by $d$. 
%\begin{Def}\label{def-twist}
%Let $\underline{M} = (M, F)$ be a display over $\underline{S}$. If $d \in \zz$, the \textit{twist of $\underline{M}$ by $d$}, denoted $\underline{M}(d)$, is the tensor product $\underline{M} \otimes \underline{U}(d)$. 
%\end{Def}

The definition of type for a finite projective $S$-module extends in a natural way to displays. 

\begin{Def}
	Let $I$ be a collection of integers as in Definition \ref{def-type}. We say a display $\underline{M} = (M,F)$ is \textit{of type $I$} if the finite projective graded $S$-module $M$ is of type $I$.
\end{Def}

\begin{Def}\label{def-stddatum}
	 A \textit{standard datum} for a display is a pair $(L, \Phi)$ consisting of a finite projective graded $S_0$-module $L = \bigoplus_i L_i$ and a $\sigma_0$-linear automorphism $\Phi: L \to L$. 
\end{Def}

From a standard datum $(L, \Phi)$, we define a display $(M, F)$ by taking $M := L\otimes_{S_0} S$ and $F(x \otimes s) = \sigma(s) \Phi(x)$. On the other hand, if $\underline{M} = (M,F)$ is a display, then any normal decomposition $L$ of $M$ determines a standard datum by viewing $L$ as a submodule of $M = L\otimes_{S_0} S$ via $x \mapsto x \otimes 1$ and taking $\Phi = F \res_L$. It is clear that the display resulting from this $(L,\Phi)$ is indeed $(M,F)$. Hence if every finite projective graded $S$-module has a normal decomposition then every display over $\underline{S}$ is standard, i.e. is defined from a standard datum. Note also that if $(L, \Phi)$ is a standard datum for a display $\underline{M} = (M, F)$, then $M^\tau \cong L$ and $M^\sigma \cong L^{\sigma_0}$, cf. \cite[Remark 3.2.5]{Lau2018}.

%\begin{rmk} 
%	Every finite projective graded $W(R)^\oplus$-module has a normal decomposition, so every display over $\underline{W}(R)$ comes from a standard datum. It follows that the functor defined in Remark \ref{rmk-lz} restricts to an equivalence between displays over $\underline{W}(R)$ and Langer-Zink displays over $R$. 
%\end{rmk}
Let us remark that both the tensor product and base change can be defined in a natural way using standard data. Indeed, if $\underline{M} = (M, F)$ and $\underline{M}' = (M', F')$ are displays over $\underline{S}$ with standard data $(L, \Phi)$ and $(L', \Phi')$ respectively, then $(L \otimes_{S_0} L', \Phi \otimes \Phi')$ is a standard datum for $\underline{M} \otimes \underline{M}'$. Similarly $(L \otimes_{S_0} S_0', \Phi \otimes \sigma_0')$ is a standard datum for $\underline{M}_{\underline{S}'}$. The characterizations of the tensor product and the base change above prove that the resulting object is independent of the choice of standard data.

Now we return our focus to the Witt frame, which is the case of interest in the remainder of the paper. Note that if $R$ and $R'$ are two $p$-adic rings, then a ring homomorphism $R \to R'$ induces a morphism of frames $\underline{W}(R) \to \underline{W}(R')$. Let \textbf{Disp}$^W$ be the category fibered over \textbf{Nilp}$_{\zz_p}$ whose fiber over $R$ is \textbf{Disp}$(\underline{W}(R))$. As in the case of finite projective graded $W(R)^\oplus$-modules, this fibered category satisfies effective descent for morphisms and for objects.

\begin{lemma}\label{lem-displaydescent}
	The fibered category \textup{\textbf{Disp}}$^W$ is an fpqc stack over \textup{\textbf{Nilp}}$_{\zz_p}$. 
\end{lemma}
\begin{proof}
	This is \cite[Lemma 4.4.2]{Lau2018}.
\end{proof}

Suppose $\underline{M} = (M, F)$ and $\underline{M}' = (M', F')$ are displays over $W(R)^\oplus$. Let $\psi: M \to M'$ be a homomorphism of graded $W(R)^\oplus$-modules, and write $\psi_{R'}$ for the base change of $\psi$ to $W(R')^\oplus$. We have the following lemma, which says that the property of ``being a morphism of displays'' can be checked fpqc locally.

\begin{lemma}\label{lem-morphism_local}
	Let $R$ be a $p$-nilpotent $\zz_p$-algebra, and let $R \to R'$ be a faithfully flat extension of $\zz_p$-algebras. With the set-up as above, $\psi_{R'}$ is a morphism of displays over $\underline{W}(R')$ if and only if $\psi$ is a morphism of displays over $\underline{W}(R)$.
\end{lemma}
\begin{proof}
	Obviously $\psi_{R'}$ is a morphism of displays if $\psi$ is. For the converse, we want the following diagram to commute:
	\begin{center}
		\begin{tikzcd}
			M^\sigma
				\arrow[r, "\psi^\sigma"] \arrow[d, "F^\sharp"]
			& (M')^\sigma 
				\arrow[d, "(F')^\sharp"]
			\\ M^\tau
				\arrow[r, "\psi^\tau"]
			& (M')^\tau
		\end{tikzcd}
	\end{center}
	The diagram commutes after base change to $W(R')$, so the result follows from Witt vector descent for finitely generated projective modules, \cite[Corollary 34]{Zink2002}.
\end{proof}

%The next lemma shows that our definition of an effective predisplay agrees with that of \cite{Lau2018}.
%\begin{lemma}
%	Let $\underline{M} = (M,F)$ be a predisplay over $\underline{S}$. Then $\underline{M}$ is effective if and only if $t^n: M_0 \to M_{-n}$ is bijective for all $n \ge 0$.
%\end{lemma}
%\begin{proof}
%	Suppose $\underline{M}$ is effective, so $d(M) \ge 0$. Then $\theta_n: M_n \to M^\tau$ is bijective for all $n \le 0$. But $\theta_0 = \theta_{-n} \circ t^n$ for $n \ge 0$, so if $\theta_0$ and $\theta_{-n}$ are bijective it follows that $t^n$ is as well.
%	
%	Conversely, suppose $t^n: M_0 \to M_{-n}$ is bijective for all $n \ge 0$. Fix $n \ge 0$. We need to show $\theta_{-n}: M_{-n} \to M^\tau$ is bijective. There is some $m \ge n$ such that $\theta_{-m}: M_{-m} \to M^\tau$ is bijective. Since both the maps $t^n: M_0 \to M_{-n}$ and $t^m: M_0 \to M_{-m}$ are bijective, it follows that $t^{m-n}: M_{-n} \to M_{-m}$ is bijective as well. Hence $\theta_{-n} = \theta_{-m} \circ t^{m-n}$ is bijective for all $n \ge 0$.  
%\end{proof} 

%\input{Isodisplays.tex}
%\input{HodgeFilt.tex}
\subsection{1-displays}\label{sub-1displays}
Fix a $p$-adic ring $R$. Following the notation in Definition \ref{def-wittframe}, let us denote the Witt vector Frobenius on $W(R)$ by $f$. 
\begin{Def}\label{def-zink}
	A \textit{Zink display} over $R$ is a quadruple $\underline{P} = (P_0, P_1, F_0, F_1)$ consisting of a finitely generated projective $W(R)$-module $P_0$, $P_1 \subseteq P_0$ is a submodule and $F_i: P_i \to P_0$ is an $f$-linear map, for $i = 0$ and $1$, such that the following conditions are satisfied:
	\begin{enumerate}[(i)]
		\item  $I_RP_0 \subseteq P_1 \subseteq P_0$, and the filtration 
		\begin{align*}
			0 \subseteq P_1 / I_R P_0 \subseteq P_0 / I_R P_0
		\end{align*}
		has finitely generated projective $R$-modules as graded pieces.
		\item $F_1: P_1 \to P_0$ is an $f$-linear epimorphism.
		\item For $x \in P_0$ and $\xi \in W(R)$ we have 
		\begin{align*}
			F_1(v(\xi)\cdot x ) = \xi\cdot F_0(x).
		\end{align*}	
	\end{enumerate}
\end{Def}

We remark that Zink displays are frequently referred to only as ``displays'' in the literature, and in \cite{Zink2002} they are called ``not-necessarily-nilpotent'' (or 3n-) displays. The filtration in (i) of Definition \ref{def-zink} is called the Hodge filtration of $\underline{P}$. 

\begin{lemma}\label{lem-zink}
	The category of Zink displays over $R$ is equivalent to the category of 1-displays over $\underline{W}(R)$. 
\end{lemma}
\begin{proof}
	Let $\underline{M} = (M,F)$ be a 1-display over $\underline{W}(R)$. Since $d(M) \ge 0$,  we know $\theta_0: M_0 \to M^\tau$ is bijective, cf. Lemma \ref{lem-theta}. We claim additionally that $\theta_1: M_1 \to M^\tau$ is injective. Indeed, by \cite[Lemma 3.1.4]{Lau2018} and Lemma \ref{lem-type}, we can choose a normal decomposition $L$ for $M$ such that $L_i = 0$ if $i \notin {0,1}$. Then 
	\begin{align}\label{eq-m1}
		M_1 = (L_0 \otimes_{W(R)} I_R) \oplus (L_1 \otimes_{W(R)} W(R)),
	\end{align}
	and $\theta_1: M_1 \to M^\tau = L$ is given by the natural inclusion. 

%	It's enough to prove the claim in the case where $M$ is free, so suppose $M = \bigoplus_{i=1}^r W(R)^\oplus(m_i)$ is a finite free graded $W(R)^\oplus$-module. By Lemma \ref{lem-shape}, all of the $m_i$'s must be either $0$ or $-1$. Hence
%	\begin{align*}
%		M = (W(R)^\oplus)^{n_0} \oplus (W(R)^\oplus(-1))^{n_1}
%	\end{align*}
%	for some $n_0, n_1 \ge 0$. Then $M_1 = I_R^{n_0} \oplus W(R)^{n_1}$, and $\theta_1$ is given by 
%	the natural inclusion into $W(R)^{n_0+n_1}$. This proves the claim.	
	
	Now let $P_0 = M^\tau$ and let $P_1$ be the image of $M_1$ under $\theta_1$. Then $P_0$ is a finitely generated projective $W(R)$-module, and $P_1$ is a submodule. The restrictions of $F$ to $M_0$ and $M_1$ are $f$-linear homomorphisms of $W(R)$-modules $M_i \to M^\tau$. Then $F_i := F \res_{M_i} \circ \theta_i^{-1}: P_i \to P_0$ is also $f$-linear for $i = 0,1.$ We claim $(P_0, P_1, F_0, F_1)$ is a Zink display. 
	
	The first part of condition (i) in Definition \ref{def-zink} is immediate. To verify the remaining conditions, choose a normal decomposition $L = L_0 \oplus L_1$ for $M$ as above, and let $\Phi = F \res_L$, so $(L, \Phi)$ is a standard datum for $\underline{M}$. The $W(R)$-module $M_1$ can be written as in (\ref{eq-m1}), and we see
	\begin{align*}
		M_0 = (L_0 \otimes_{W(R)} W(R)) \oplus (L_1 \otimes_{W(R)} W(R)\cdot t).
	\end{align*}
	The second part of condition (i) follows because $P_0 / P_1 \cong L_0 \otimes_{W(R)} R$, and $P_1 / I_R P_0 \cong L_1 \otimes_{W(R)} R$. Condition (ii) is equivalent to the condition that $( F\res_{M_1})^\sharp: (M_1)^f \to M^\tau$ is surjective. Since $\Phi^\sharp: L^f \to L$ is a $W(R)$-module isomorphism and $L$ is naturally identified with $M^\tau$, it is enough to show for any $x\otimes \xi \in L^f$, there is some $y \in (M_1)^f$ with $F^\sharp(y) = \Phi^\sharp(x\otimes\xi)$. First suppose $x\otimes \xi \in(L_0)^f$. Then $x \otimes v(\xi) \otimes 1 \in (M_1)^f$, and
	\begin{align*}
		F^\sharp(x \otimes v(\xi) \otimes 1) = F(x \otimes v(\xi)) = \sigma(v(\xi)) F(x) = \xi F(x) = \Phi^\sharp(x \otimes \xi).
	\end{align*}
	Now if $x \otimes \xi \in (L_1)^f$, then $x \otimes 1 \otimes \xi \in (M_1)^f$, and 
	\begin{align*}
		F^\sharp(x \otimes 1 \otimes \xi) = \xi F(x) = \Phi^\sharp(x \otimes \xi).
	\end{align*}
	This completes the proof of condition (ii). For condition (iii), let $x = x_0 + x_1 \in P_0 = M^\tau = L$ with $x_0 \in L_0$ and $x_1 \in L_1$. If $\xi \in W(R)$, then 
	\begin{align*}
		\theta_1^{-1}(v(\xi)\cdot x) = x_0 \otimes v(\xi) + x_1 \otimes v(\xi)
	\end{align*}
	with the first $v(\xi)$ viewed as an element of $(W(R)^\oplus)_1 = I_R$ and the second as an element in $(W(R)^\oplus)_0 = W(R)$. Then
	\begin{align*}
		F_1(v(\xi)\cdot x) = \xi(\Phi(x_0) + p\Phi(x_1)).
	\end{align*} 
	But $\theta_0^{-1}(x) = x_0 \otimes 1 + x_1 \otimes t$, so this is the same as $\xi\cdot F_0(x)$.
	
	If $\psi: (M,F) \to (M',F')$ is a morphism of 1-displays, then the $W(R)$-module homomorphism $\psi^\tau$ defines a morphism of the corresponding Zink displays. Hence the association $(M, F) \mapsto (P_0, P_1, F_0, F_1)$ determines a functor from the category of $1$-displays over $\underline{W}(R)$ to the category of Zink displays over $R$, which we claim is an equivalence of categories. Choosing a normal decomposition $L$ for $M$ as above, we see
	\begin{align*}
		M = (L_0 \otimes_{W(R)} W(R)^\oplus) \oplus (L_1 \otimes_{W(R)} W(R)^\oplus(-1)).
	\end{align*}
	It follows that any morphism $\underline{M} \to \underline{M'}$ of 1-displays is uniquely determined by its restriction to $L_0$ and $L_1$, and therefore by its restriction to $M_0$ and $M_1$, so the functor is faithful. 
	
	Now let $\underline{P}$ and $\underline{P'}$ be the Zink displays associated to $1$-displays $(M,F)$ and $(M',F')$, and suppose $\varphi: \underline{P} \to \underline{P'}$ is a morphism of Zink displays. In particular, $\varphi$ is a $W(R)$-module homomorphism $P_0= M^\tau \to (M')^\tau=P_0'$ which sends $P_1=\theta_1(M_1)$ to $P_1'=\theta_1'(M_1')$, and which satisfies
	\begin{align}\label{eq-compat}
		F_0'\circ \varphi = \varphi \circ F_0 ,\ \text{ and } \ F_1'\circ \varphi = \varphi \circ F_1.
	\end{align}
	For $i=0,1$, let $\psi_i$ be the composition
	\begin{align*}
		L_i \hookrightarrow P_i \xrightarrow{\psi} P_i'\xrightarrow{(\theta_i')^{-1}} M_i'.
	\end{align*}
	Then $\psi_0 + \psi_1$ defines a $W(R)$-module homomorphism $L \to M'$ which sends $L_i$ to $M_i'$ for every $i$. Therefore it induces a graded $W(R)^\oplus$-module homomorphism $\psi: M \to M'$. By construction $\psi^\tau = \varphi$, and it follows from the identities (\ref{eq-compat}) that $F'\circ \psi = \psi^\tau \circ F$. Hence the functor is full.
%	so if $x_0 \in L_0$, $x_1 \in L_1$, we have 
%	\begin{align*}
%		F'\psi(x_0 \otimes 1 + x_1 \otimes 1) &= F'((\theta_0)^{-1}\alpha(x_0) + (\theta_1)^{-1}\alpha(x_1)) \\
%		&= F_0'\varphi(x_0) + F_1'\varphi(x_1) \\
%		&= \varphi( F_0(x_0) + F_1(x_1) )\\
%		&= \varphi(F((\theta_0)^{-1}(x_0)) + F((\theta_1)^{-1}(x_1))) \\
%		&= \psi^\tau F( x_0 \otimes 1 + x_1 \otimes 1).
%	\end{align*}
	
	Now suppose $(P_0, P_1, F_0,F_1)$ is a Zink display. By \cite[Lemma 2]{Zink2002}, condition (i) in the definition of Zink displays implies the existence of finitely generated projective $W(R)$-modules $L_0$ and $L_1$ such that $P_0 = L_0 \oplus L_1$ and $P_1 = I_R L_0 \oplus L_1$. If we define $\Phi = F_0 \res_{L_0} \oplus F_1\res_{L_1}$, then $(P_0, \Phi)$ constitutes a standard datum for a 1-display whose resulting Zink display is isomorphic to $(P_0,P_1,F_0,F_1)$. 
\end{proof}

%\begin{rmk}
%	It is clear that the Hodge filtrations of a Zink display and its corresponding 1-display match under the above equivalence.
%\end{rmk}

If $(P_0,P_1, F_0,F_1)$ is a Zink display, then by \cite[Lemma 10]{Zink2002} there exists a unique linear map $V^\sharp: F_0 \to F_0^{f}$ characterized by
\begin{align*}
V^\sharp(\xi \cdot F_0(x)) = p\xi \otimes x, \ \text{ and } \ V^\sharp(\xi \cdot F_1(y)) = \xi \otimes y 
\end{align*}
for all $\xi \in W(R)$, $x \in P_0, y \in P_1$. Denote by $V_i^\sharp$ the induced map $P_0^{f^i} \to P_0^{f^{i+1}}$.
\begin{Def}\label{def-nilpotence}
	A Zink display $(P_0, P_1, F_0,F_1)$ over $R$ is \textit{nilpotent} if there exists an $N$ such that the composition
	\begin{align*}
	V^\sharp_N \circ V^\sharp_{N-1} \circ \cdots \circ V^\sharp: P_0 \to P_0^{f^{N+1}}
	\end{align*}
	is zero modulo $I_R +pW(R)$.
\end{Def}
\begin{rmk}\label{rem-nilp}
	If $R = k$ is a perfect field of characteristic $p$, then displays over $k$ are equivalent to Dieudonn\'e modules over $k$, and the nilpotence condition on a display corresponds to topological nilpotence of the Verschiebung operator on the corresponding Dieudonn\'e module, cf. \cite[Proposition 15]{Zink2002}.
\end{rmk} 

For a general $p$-adic ring $R$, Zink defines a functor $BT_R$ from the category of Zink displays over $R$ to the category of formal groups over $R$, and he shows that the restriction of $BT_R$ to the full subcategory of nilpotent displays has essential image contained in the category of $p$-divisible formal groups. The following is the main theorem connecting displays and formal $p$-divisible groups. For many $R$ it was proved by Zink in \cite{Zink2002}, and for all $R$ which are $p$-adic this was proved by Lau in \cite{Lau2008}.

\begin{thm}[Zink, Lau]\label{thm-BT}
	The functor $BT_R$ induces an equivalence of categories between nilpotent Zink displays over $R$ and formal $p$-divisible groups over $R$. Further, $BT_R$ has the following properties: if $\underline{P} = (P_0, P_1, F_0, F_1)$ is a nilpotent Zink display, then
	\begin{enumerate}[\textup{(}i\textup{)}]
		\item \textup{Lie}$\left(BT_R(\underline{P})\right) = P_0 / P_1$
		\item The height of $BT_R(\underline{P})$ is equal to the rank of $P_0$ over $W(R)$.
%		\item If $\underline{P}^\vee(1)$ is a nilpotent display, then there is a canonical isomorphism
%		\begin{align*}
%			BT_R(\underline{P}^\vee(1)) \cong BT_R(\underline{P})^\vee,
%		\end{align*}
%		where the right-hand side is the Serre dual of $BT_R(\underline{P})$.
	\end{enumerate}
\end{thm}

Finally let us mention some aspects the of the connection between the theory of displays and the theory of crystals associated to $p$-divisible groups. Suppose $X$ is a formal $p$-divisible group over a $p$-nilpotent $\zz_p$-algebra $R$. Associated to $X$ is its covariant Dieudonn\'{e} crystal $\mathbb{D}(X)$. Evaluating $\mathbb{D}(X)$ on the trivial PD-thickening $\id_R: R\to R$, we obtain a finite projective $R$-module $\mathbb{D}(X)_R$ equipped with a functorial exact sequence 
\begin{align*}
0 \to \text{Lie}(X^\vee)^\ast \to \mathbb{D}(X)_R \to \text{Lie}(X) \to 0
\end{align*}
which is compatible with base change. Here $\text{Lie}(X)$ is the Lie algebra of $X$, which is a finite projective $R$-module, $X^\vee$ is the Serre dual of $X$, and $(-)^\ast$ denotes linear dual. The filtration
\begin{align*}
	\mathbb{D}(X)_R \supset \text{Lie}(X^\vee)^\ast \supset 0
\end{align*}
is the Hodge filtration of $X$. If $X = BT_R(\underline{P})$ for a nilpotent Zink display $\underline{P} = (P_0, P_1, F_0, F_1)$, then it follows from \cite[Theorem 94]{Zink2002} that there is a canonical isomorphism of finite projective $R$-modules
\begin{align*}
	P_0 \otimes_{W(R)} R \cong \mathbb{D}(X)_R
\end{align*}
which sends the Hodge filtration of $\underline{P}$ to the Hodge filtration of $X$.

\section{$G$-displays}

\subsection{$G$-displays of type $\mu$ and $(G,\mu)$-displays}\label{sub-Gdisplays}

Let $G$ be a flat affine group scheme of finite type over $\zz_p$, let $k_0$ be a finite extension of $\mathbb{F}_p$, and let $\mu: \mathbb{G}_{m,W(k_0)} \to G_{W(k_0)}$ be a cocharacter defined over $W(k_0)$. If $R$ is a $W(k_0)$-algebra, then $W(R)$ is endowed with the structure of a $W(k_0)$-algebra via composition
\begin{align*}
W(k_0) \xrightarrow{\Delta} W(W(k_0)) \to W(R),
\end{align*}
where the first map is the Cartier homomorphism (cf. \cite[Ch VII, Prop 4.2]{Lazard2006}) and the second is induced by functoriality from the $W(k_0)$-algebra structure homomorphism. Then $W(R)^\oplus$ is a graded $W(k_0)$-algebra, and $\sigma: W(R)^\oplus \to W(R)$ extends the Frobenius on $W(k_0)$. A frame whose graded ring and Frobenius satisfy these properties is called a $W(k_0)$-frame, cf. \cite[Definition 5.0.1]{Lau2018}.

Using the cocharacter $\mu$ we define a (right) action of $\mathbb{G}_{m,W(k_0)}$ on $G_{W(k_0)}$ as follows: if $\lambda \in \mathbb{G}_{m,W(k_0)}(R)$ and $g \in G_{W(k_0)}(R)$ for some $W(k_0)$-algebra $R$, define
\begin{align*}
	g \cdot \lambda := \mu(\lambda)^{-1}g \mu(\lambda).
\end{align*}
Write $G_{W(k_0)} = \text{Spec } A$ for a $W(k_0)$-algebra $A$. Then the action defined above also defines an action on $A$. Since $G_{W(k_0)}(R) = \textup{Hom}_{W(k_0)}(A, R)$ for any $W(k_0)$-algebra $R$, there is a canonical bijection between elements of $A$ and natural transformations $G_{W(k_0)} \to \mathbb{A}_{W(k_0)}^1$ given by sending $f \in A$ to evaluation on $f$. Hence we can make the action on $A$ explicit: if $f \in A$ and $\lambda \in \mathbb{G}_{m,W(k_0)}(R)$, define $\lambda \cdot f$ to be the function $G(R) \to R$ given by  
\begin{align*}
(\lambda \cdot f)(g) := f(\mu(\lambda)^{-1}g\mu(\lambda))
\end{align*}
for $g \in G_{W(k_0)}(R)$. 
Giving the collection of these actions as $R$ varies corresponds to a $\zz$-grading on $A$. In particular, $A_n$ is the set of $f \in A$ with $(\lambda \cdot f)(g) = \lambda^n f(g)$ for all $\lambda \in \mathbb{G}_{m,W(k_0)}(R)$, $g \in G(R)$. 

Following \cite[\textsection 5]{Lau2018}, for any $\zz$-graded $W(k_0)$-algebra $S$, let $G(S)_\mu$ be the set of $\mathbb{G}_m$-equivariant morphisms $\text{Spec }S \to G$ over $W(k_0)$. Equivalently, we have
\begin{align*}
	G(S)_\mu = \text{Hom}_{W(k_0)}^0(A,S),
\end{align*}
i.e. $G(S)_\mu$ is the subset of $G_{W(k_0)}(S)$ consisting of $W(k_0)$-algebra homomorphisms which preserve the respective gradings. The Hopf algebra structure for $A$ preserves the grading induced by $\mu$, so $G(S)_\mu$ forms a subgroup of $G_{W(k_0)}(S)$. 

Suppose now $\underline{S} = (S,\sigma,\tau)$ is a $W(k_0)$-frame. The $\zz_p$-algebra homomorphisms $\sigma, \tau: S \to S_0$ induce group homomorphisms 
\begin{align*}
	\sigma, \tau: G(S)_\mu \to G(S_0).
\end{align*}
Indeed, if $g \in G(S)_\mu$, then $\sigma(g)$ (resp. $\tau(g)$) is defined by post-composing $g \in \text{Hom}_{W(k_0)}(A, S)$ with $\sigma: S \to S_0$ (resp. $\tau: S \to S_0$). Using these homomorphisms we define a group action of $G(S)_\mu$ on $G(S_0)$ as follows:
\begin{equation}\label{eq-action}
	G(S_0) \times G(S)_\mu \to G(S_0), \ (x,g) \mapsto \tau(g)^{-1}x \sigma(g).
\end{equation}

Let us restrict our focus to the Witt frame, $\underline{W}(R)$ (cf. \ref{def-wittframe}). We define two group-valued functors on $W(k_0)$-algebras as follows: if $R$ is a $W(k_0)$-algebra, let
\begin{align*}
	L^+G(R) := G(W(R)), \text{ and } L^+_\mu G(R) := G(W(R)^\oplus)_\mu.
\end{align*}
By \cite[Lemma 5.4.1]{Lau2018} these are representable functors. We will refer to the $W(k_0)$-group scheme $L^+G$ as the \textit{positive Witt loop group scheme}, and to $L^+_\mu G$ as the \textit{display group} for the pair $(G,\mu)$. 

\begin{Def}[Lau] \label{def-gdisp} The stack of \textit{$G$-displays of type $\mu$} is the fpqc quotient stack 
\begin{align*}
	\Gdisp = [L^+G / L^+_\mu G].
\end{align*}
over \textbf{Nilp}$_{W(k_0)}$, where $L^+_\mu G$ acts on $L^+G$ via the action (\ref{eq-action}). 
\end{Def}

Explicitly, for a $p$-nilpotent $W(k_0)$-algebra $R$, $\Gdisp(R)$ is the groupoid of pairs $(Q,\alpha)$, where $Q$ is an fpqc-locally trivial $L^+_\mu G$-torsor over Spec $R$ and $\alpha: Q \to L^+G$ is an $L^+_\mu G$-equivariant map. 
%Equivalently, by \cite[Remark 5.3.3]{Lau2018} we may regard $G$-Disp$_\mu^W(R)$ as the groupoid of pairs $(Q,\beta)$, where $Q$ is an $L^+_\mu G$-torsor over Spec $R$ and $\beta: Q^\sigma \xrightarrow{\sim} Q^\tau$ is an isomorphism of $L^+G$-torsors. 
If $(G,\mu)$ and $(G',\mu')$ are two pairs as above, and $\varphi: G \to G'$ is a $\zz_p$-group scheme homomorphism such that $\varphi \circ \mu = \mu'$, then  $\varphi$ induces a morphism of stacks
\begin{align*}
	\Gdisp \to G'\text{-}\textbf{Disp}^W_{\mu'}.
\end{align*}
Indeed, $\mathbb{G}_m$-equivariance of $\varphi$ furnishes us with a group homomorphism $G(S)_\mu \to G'(S)_{\varphi\circ\mu}$, and since $\varphi$ is defined over $\zz_p$ it commutes with $\sigma$ and $\tau$. On the level of objects, the pair $(Q,\alpha)$ is sent to $(Q^\varphi, \alpha')$, where $Q^\varphi$ is the pushforward of $Q$ along $\varphi: L^+_\mu G \to L^+_{\mu'} G'$ and $\alpha'$ is the induced $L^+_{\mu'} G'$-equivariant morphism $Q^\varphi \to L^+G'$.

\begin{rmk}\label{rmk-gln}
Let $I = (i_1, i_2, \dots, i_n) \in \zz^n$ with $i_1 \le i_2 \le \cdots \le i_n$, and define a cocharacter 
\begin{align*}
	\mu_I : \mathbb{G}_m \to \text{GL}_n, \ \lambda \mapsto \text{diag}(\lambda^{i_1}, \lambda^{i_2}, \dots, \lambda^{i_n}).
\end{align*}
Then $\text{GL}_n$-\textbf{Disp}$^W_{\mu_I}$ is the stack of displays of type $I = (i_1, \dots, i_n)$, cf. \cite[Example 5.3.5]{Lau2018}.

As a particular example, consider the case where $I = (0^{(r)}, 1^{(n-r)})$ for some $r$. Then $\mu= \mu_{r,n}$ is the minuscule cocharacter $\lambda \mapsto \text{diag}(1^{(r)}, \lambda^{(n-r)})$, and $L^+_\mu \text{GL}_n(R)$ consists of block matrices
\begin{align*}
	\begin{pmatrix} A & B \\ C & D \end{pmatrix}
\end{align*}
where 
\begin{itemize}
	\item $A$ is an $r \times r$-matrix whose entries are in $W(R)^\oplus_0 = W(R)$,
	\item $B$ is an $r\times (n-r)$-matrix whose entries are in $W(R)^\oplus_1 = I_R$,
	\item $C$ is an $(n-r) \times r$-matrix whose entries are in $W(R)^\oplus_{-1} = W(R)\cdot t$,
	\item $D$ is an $(n-r) \times (n-r)$-matrix whose entries are in $W(R)^\oplus_0 = W(R)$.
\end{itemize}
In this case, $\textup{GL}_n$-\textbf{Disp}$^W_{\mu_{r,n}}$ is isomorphic to the stack of Zink displays $(P_0, P_1, F_0, F_1)$ with $\text{rk}_{W(R)} P_0 = n$ and $\text{rk}_{R}(P_0 / P_1) = d$, cf. Lemma \ref{lem-zink}.
\end{rmk}

B\"{u}ltel and Pappas define \cite{BP2017} an alternative category of $(G,\mu)$-displays over $R$ in the case where $G$ is reductive over $\zz_p$ and $\mu$ is a minuscule cocharacter defined over $W(k_0)$. Let us briefly explain. Let $P_\mu$ be the parabolic sub-group scheme of $G$ defined by $\mu$ (see \cite[Appendix 1]{BP2017}). Then B\"ultel and Pappas define a closed sub-group scheme $H^\mu$ of $L^+G$ whose points in a $W(k_0)$-algebra $R$ are those elements of $L^+G(R)$ which map to $P_\mu(R)$ under the canonical map $L^+G(R) \to G(R)$. By \cite[Proposition 3.1.2]{BP2017}, there is a group scheme homomorphism
\begin{align*}
	\Phi_{G,\mu}: H^\mu \to L^+G
\end{align*}
such that $\Phi_{G,\mu}(h) = F(\mu(p)h\mu(p)^{-1}) \in G(W(R)[1/p])$, where $F$ is induced from the Witt vector Frobenius. Then a $(G,\mu)$-display over a $W(k_0)$-scheme $S$ is a triple $(Q,P,u)$, where $Q$ is a torsor for $H^\mu$ over $S$, $P$ is the pushforward of $Q$ to $L^+G$, and $u:Q \to P$ is a morphism such that $u(q\cdot h) = u(q)\Phi_{G,\mu}(h)$ for all $h \in H^\mu$, $q \in Q$. In \cite[3.2.7]{BP2017} it is shown that the stack of $(G,\mu)$-displays over $\textbf{Nilp}_{W(k_0)}$ is isomorphic to the fpqc quotient stack $[L^+G/_{\Phi_{G,\mu}} H^\mu]$, where $H^\mu$ acts on $L^+G$ via 
\begin{align*}
	g\cdot h = h^{-1} g \Phi_{G,\mu}(h)
\end{align*}
for $R$ a $p$-nilpotent $\zz_p$-algebra, $h \in H^\mu(R)$, $g \in L^+G(R)$. In \cite[Remark 6.3.4]{Lau2018} Lau proves the following lemma by showing that $\tau$ induces an isomorphism $L^+_\mu G \xrightarrow{\sim} H^\mu$, which is compatible with the actions of $L^+_\mu G$ resp. $H^\mu$ on $L^+G$.
\begin{lemma}\label{lem-bp}
	The stack of $(G,\mu)$-displays as in \textup{\cite{BP2017}} is isomorphic to $\Gdisp$. 
\end{lemma}
\qed
\subsection{Graded fiber functors}
Let $G$ be a flat affine group scheme of finite type over $\zz_p$. Denote by \textbf{Rep}$_{\zz_p}(G)$ the category of representations of $G$ on finite free $\zz_p$-modules. Let $R$ be a $W(k_0)$-algebra.

%\begin{lemma}\label{lem-reps_colim}
%	Every representation $(V,\rho)$ in \textup{\textbf{Rep}}$_{\zz_p}'(G)$ is the directed union of its finite dimensional sub-representations.
%\end{lemma}
%\begin{proof}
%	This is the corollary to \cite[Proposition 1.2]{Serre1968}.
%\end{proof}

%A tensor functor is a functor between tensor categories which preserves the tensor product. Let $\mathcal{A}$ and $\mathcal{B}$ be tensor categories with unit objects $\mathbbm{1}_\mathcal{A}$ and $\mathbbm{1}_\mathcal{B}$ their respective unit objects. If $\Psi$ and $\Psi'$ are two tensor functors between $\mathcal{A}$ and $\mathcal{B}$, then a morphism of tensor functors $\alpha: \Psi \to \Psi'$ is a natural transformation of functors which is compatible with the tensor product, in the sense that $\alpha^{\mathbbm{1}_\mathcal{A}} = \id_{\mathbbm{1}_\mathcal{B}}$ and $\alpha^{X \otimes Y} = \alpha^{X} \otimes \alpha^{Y}$ for any two objects $X$, $Y$ in $\mathcal{A}$. 

\begin{Def}
	A \textit{graded fiber functor} over $W(R)^\oplus$ is an exact tensor functor
	\begin{align*}
		\mathscr{F}: \text{\textbf{Rep}}_{\zz_p}(G) \to \text{\textbf{PGrMod}}^W(R).
	\end{align*}
\end{Def}

Denote by \textbf{GFF}$^W(R)$ the category of graded fiber functors over $W(R)^\oplus$. Morphisms in this category are morphisms of tensor functors. If $R \to R'$ is a ring homomorphism and $\mathscr{F}$ is a graded fiber functor over $W(R)^\oplus$, then we define the base change of $\mathscr{F}$ to $W(R')^\oplus$, written $\mathscr{F}_{R'}$, as the composition \textbf{Rep}$_{\zz_p}(G) \to $ \textbf{PGrMod}$^W(R) \to $ \textbf{PGrMod}$^W(R')$. As $R$ varies in \textbf{Nilp}$_{W(k_0)}$ we obtain a fibered category \textbf{GFF}$^W$ whose fiber over $R$ is \textbf{GFF}$^W(R)$. 

\begin{lemma}\label{lem-functor_stack}
	The fibered category \textup{\textbf{GFF}}$^W$ is an fpqc stack in groupoids over \textup{\textbf{Nilp}}$_{W(k_0)}$.
\end{lemma}
\begin{proof}
The proof is essentially the same as that of \cite[Proposition 7.2]{PWZ2015}. Let us summarize the argument.

Both \textbf{Rep}$_{\zz_p}(G)$ and \textup{\textbf{PGrMod}}$^W(R)$ are rigid tensor categories (cf. Lemma \ref{lem-pmod_rigid}), so by \cite[Proposition 1.13]{Deligne1982}, if $\mathscr{F}_1$ and $\mathscr{F}_2$ are graded fiber functors over $W(R)^\oplus$, then every morphism of tensor functors $\mathscr{F}_1 \to \mathscr{F}_2$ is an isomorphism. Hence \textbf{GFF}$^W$ is fibered in groupoids. 

It remains to show \textbf{GFF}$^W$ satisfies effective descent for morphisms and for objects. Let $R$ be a $p$-nilpotent $W(k_0)$-algebra, let $R'$ be a faithfully flat extension, and let $R'' = R' \otimes_R R'$. Suppose $\lambda': (\mathscr{F}_1)_{R'} \to (\mathscr{F}_2)_{R'}$ is a morphism over $R'$ such that the two pullbacks to $R''$ agree. Then for each $(V,\rho) \in \text{Ob}(\text{\textbf{Rep}}_{\zz_p}(G))$, the same holds for ${\lambda'}_\rho$. By Lemma \ref{lem-pmod_descent}, morphisms of finite projective graded $W(R)^\oplus$-modules descend, so we obtain unique morphisms $\lambda_\rho$ for every $(V,\rho)$. Applying descent again, one checks that these morphisms piece together to form a natural transformation $\lambda: \mathscr{F}_1 \to \mathscr{F}_2$ which is compatible with the tensor product.

Finally we prove \textbf{GFF}$^W$ satisfies effective descent for objects. Let $\mathscr{F}'$ be a graded fiber functor over $W(R')^\oplus$ equipped with a descent datum, i.e. equipped with an isomorphism
\begin{align*}
p_1^*\mathscr{F}' \xrightarrow{\sim} p_2^*\mathscr{F}'
\end{align*}
of tensor functors over $W(R'')^\oplus$ satisfying the cocycle condition, where $p_1^*$ and $p_2^*$ denote base change along the maps induced by $r \mapsto r \otimes 1$ and $r \mapsto 1 \otimes r$ from $R' \to R'\otimes R'$, respectively. The given descent datum induces a descent datum on $\mathscr{F}'(V,\rho)$ for each $(V,\rho)$ in \textbf{Rep}$_{\zz_p}(G)$, so, by Lemma \ref{lem-pmod_descent}, for every $(V,\rho)$ we obtain a finite projective $W(R)^\oplus$-module $\mathscr{F}(V,\rho)$ over $R$ whose base change to $R'$ is $\mathscr{F}'(V,\rho)$. As above, one checks that the assignment $(V,\rho) \to \mathscr{F}(V,\rho)$ is functorial, and that the resulting functor is indeed a tensor functor. By Lemma \ref{lem-sequence} exactness is an fpqc local property for finite projective $W(R)^\oplus$-modules, so $\mathscr{F}$ is exact. 
\end{proof}

If $\mathscr{F}_1$ and $\mathscr{F}_2$ are two graded fiber functors over $W(R)^\oplus$, denote by \underline{Isom}$^\otimes(\mathscr{F}_1,\mathscr{F}_2)$ the functor which assigns to an $R$-algebra $R'$ the set Isom$^\otimes((\mathscr{F}_1)_{R'},(\mathscr{F}_2)_{R'})$ of isomorphisms of tensor functors $(\mathscr{F}_1)_{R'} \to (\mathscr{F}_2)_{R'}$. By the Lemma \ref{lem-functor_stack} this is an fpqc sheaf over \textbf{Nilp}$_{R}$. We write \underline{Aut}$^\otimes(\mathscr{F}):=\text{\underline{Isom}}^\otimes(\mathscr{F},\mathscr{F})$ and Aut$^\otimes(\mathscr{F}_R):=\text{Isom}^\otimes(\mathscr{F}_R,\mathscr{F}_R)$. There is a natural action of \underline{Aut}$^\otimes(\mathscr{F}_1)$ on \underline{Isom}$^\otimes(\mathscr{F}_1,\mathscr{F}_2)$ by pre-composition.

If $R$ is a $W(k_0)$-algebra, let us define a tensor functor 
\begin{align*}
\mathscr{C}_{\mu,R}:\text{\textbf{Rep}}_{\zz_p}(G) \to \textup{\textbf{PGrMod}}^W(R)
\end{align*}
attached to any cocharacter $\mu$ of $G$ defined over $W(k_0)$. Given a representation $(V,\rho)$ in $\textbf{Rep}_{\zz_p}(G)$, $\mu$ induces a canonical decomposition
\begin{align}\label{eq_decomp}
V_{W(k_0)} = \bigoplus_{i \in \zz} V_{W(k_0)}^i,
\end{align}
where $V_{W(k_0)} := V\otimes_{\zz_p} W(k_0)$, and 
\begin{align*}
	V_{W(k_0)}^i = \{ v \in V_{W(k_0)} \mid (\rho\circ\mu)(z)\cdot v = z^i v \text{ for all } z \in \mathbb{G}_m(W(k_0))\}.
\end{align*}
By base change along $W(k_0) \to W(R)^\oplus$ we obtain a finite projective graded $W(R)^\oplus$ module
\begin{align*}
	V \otimes_{\zz_p} W(R)^\oplus = V_{W(k_0)} \otimes_{W(k_0)} W(R)^\oplus.
\end{align*}
Any morphism $\varphi: (V,\rho) \to (U, \pi)$ in $\textbf{Rep}_{\zz_p}(G)$ preserves the grading induced by $\mu$, so we have defined a functor 
\begin{align*}
\mathscr{C}_{\mu, R}: \text{\textbf{Rep}}_{\zz_p}(G) \to \textup{\textbf{PGrMod}}^W(R), \ V\mapsto V_{W(k_0)} \otimes_{W(k_0)}W(R)^{\oplus}.
\end{align*}
The resulting functor obviously preserves the tensor product, and it is exact because the underlying modules of objects in \textbf{Rep}$_{\zz_p}(G)$ are free, so all short exact sequences remain exact after tensoring over $W(R)^\oplus$. Then $\mathscr{C}_{\mu,R}$ is a graded fiber functor over $W(R)^\oplus$.

\begin{rmk}
	For any $W(k_0)$-algebra $R$, the $W(k_0)$-algebra structure homomorphism for $W(R)$ factors through $\Delta: W(k_0) \to W(W(k_0))$ by definition. Then we see $\mathscr{C}_{\mu,R}$ is the base change of $\mathscr{C}_{\mu, W(k_0)}$ along \textbf{PGrMod}$^W(W(k_0)) \to$ \textbf{PGrMod}$^W(R)$. We will denote $\mathscr{C}_{\mu,W(k_0)}$ simply by $\mathscr{C}_\mu$.
\end{rmk}

\begin{Def}
	A graded fiber functor $\mathscr{F}$ over $W(R)^\oplus$ is \textit{of type $\mu$} if for some faithfully flat extension $R \to R'$ there is an isomorphism $\mathscr{F}_{R'} \cong \mathscr{C}_{\mu,R'}$.
\end{Def}

\begin{rmk}\label{rmk-typeistype}
	Let $(V,\rho)$ be a representation of $G$ on a finite free $\zz_p$-algebra of rank $n$, and suppose the weights of $\rho\circ\mu$ on $V$ are $\{w_1, \dots, w_r\}$ with $w_1 \le w_2 \le \dots \le w_r$. Let $r_i$ be the rank of $V_{W(k_0)}^{w_i}$. Then define $I:= I_\mu(V,\rho) = (i_1, \cdots, i_n)$ as follows: First let $i_1 = \cdots = i_{r_1} = w_1$. Then for $j\ge1$ and any $r$ satisfying 
	\begin{align*}
		\sum_{k=1}^j r_k < r \le \sum_{k=1}^{j+1} r_k,
	\end{align*}
	let $i_r = w_r$. We see that $\mathscr{C}_{\mu,R}(V,\rho)$ is of type $I$. If $\mathscr{F}$ is any graded fiber functor over $W(R)^\oplus$ of type $\mu$, then it follows from Lemma \ref{lem-localtype} that $\mathscr{F}(V,\rho)$ is of type $I$ as well.
\end{rmk}

Denote by \textbf{GFF}$_\mu(W(R)^\oplus)$ the category of graded fiber functors of type $\mu$ over $W(R)^\oplus$. Base change preserves graded fiber functors of type $\mu$, so we obtain a fibered category \textbf{GFF}$_\mu^W$ whose fiber over $R$ is \textbf{GFF}$_\mu(W(R)^\oplus)$. Since the property of ``being type $\mu$'' is an fpqc-local property, $\textbf{GFF}^W_\mu$ forms a substack of $\textup{GFF}^W$.

For any $\zz_p$-algebra $R$ let \textbf{PMod}$(R)$ be the category of finite projective $R$-modules. Associated to this category we have the canonical fiber functor 
\begin{align*}
	\omega_R: \text{\textbf{Rep}}_{\zz_p}(G) \to \textbf{PMod}(R), \ (V,\rho) \mapsto V\otimes_{\zz_p}R.
\end{align*}
Define \underline{Aut}$^\otimes(\omega)$ to be the fpqc sheaf in groups over \textbf{Nilp}$_{W(k_0)}$ which associates to an $W(k_0)$-algebra $R$ the group of automorphisms of $\omega_{R}$. By Tannakian duality, the assignment $g \mapsto \{\rho(g)\}_{(V,\rho)}$ defines an isomorphism of fpqc sheaves in groups
\begin{align*}
	G \xrightarrow{\sim} \underline{\text{Aut}}^\otimes(\omega),
\end{align*}
cf. \cite[Theorem 44]{Cornut2014} for the statement in this generality.

%Consider the functor \textbf{PGrMod}$^W(R) \to $ PMod$(W(R)^\oplus)$ given by forgetting the grading. By composing $\mathscr{C}_{\mu,R}$ with this forgetful functor we obtain $\omega_{W(R)^\oplus}$, and this induces an injective morphism of fpqc sheaves of groups
%\begin{align*}
%	\underline{\text{Aut}}^\otimes(\mathscr{C}_{\mu}) \hookrightarrow L^+\underline{\text{Aut}}(\omega) \cong L^+G,
%\end{align*}
%where, as in the previous section, if $H$ is any sheaf of groups, then $L^+H(A) := H(W(A))$. 
	
\begin{lemma}\label{lem-hom}
	Let $R$ be a $W(k_0)$-algebra. For all $g \in L^+_\mu G(R)$, the collection $\{\rho(g)\}_{(V,\rho)}$ comprises an element of \underline{\textup{Aut}}$^\otimes(\mathscr{C}_{\mu,R})$.
\end{lemma}
\begin{proof}
	For every $(V,\rho)$, we have
	\begin{align*}
		\rho(g) \in \text{GL}({V_{W(k_0)} \otimes_{W(k_0)}} W(R)^\oplus)_{\rho\circ\mu},
	\end{align*} 
	so it is enough to show that any $h \in \text{GL}(V_{W(k_0)} \otimes_{W(k_0)} W(R)^\oplus)_{\rho\circ\mu}$ preserves the grading on $V_{W(k_0)} \otimes_{W(k_0)} W(R)^\oplus$ for all $(V,\rho)$. 
	
	Suppose rk$_{\zz_p}(V) = r$, and choose an ordered basis $\{v_1, \dots, v_r\}$ for $V_{W(k_0)}$ over $W(k_0)$ such that each $v_i \in V_{W(k_0)}^{n_i}$ and $n_1 \le n_2 \le \dots \le n_r$. Relative to this basis we have 
	\begin{align*}
		\text{GL}(V_{W(k_0)} \otimes_{W(k_0)} W(R)^\oplus) \cong \text{GL}_r(W(R)^\oplus),
	\end{align*}
	and $(\rho\circ\mu)(z) = \text{diag}(z^{n_1}, \dots, z^{n_r})$. Let $A$ be the coordinate ring of $\text{GL}_{r,W(k_0)}$, so 
	\begin{align*}
		A = W(k_0)[X_{ij}, Y]_{i,j=1}^r/(\det(X_{ij})Y-1).
	\end{align*}
	If $\lambda \in \mathbb{G}_m(W(k_0))$, the action of $\lambda$ on $A$ is given by 
	\begin{align*}
		X_{ij} \mapsto \lambda^{n_j-n_i}X_{ij}.
	\end{align*}
	Then any $h \in \text{GL}_r(W(R)^\oplus)_{\rho\circ\mu}$ is represented by a matrix $(h_{ij})_{ij}$ with $h_{ij} \in W(R)^\oplus_{n_j-n_i}$. Let $v \in (V_{W(k_0)} \otimes_{W(k_0)} W(R)^\oplus)_\ell$.  We need to show $h(v) \in (V_{W(k_0)} \otimes_{W(k_0)} W(R)^\oplus)_\ell$. Write
	\begin{align*}
		v = \sum_{j=1}^r v_j \otimes \xi_j,
	\end{align*}
	where $\xi_j \in W(R)^\oplus_{\ell-n_j}$ for all $j$. Then
	\begin{align*}
		h(v) = \sum_{j=1}^r\sum_{i=1}^r v_i \otimes h_{ij}\xi_j,
	\end{align*}
	and $v_i \otimes h_{ij}\xi_j$ is of degree $n_i + (n_j - n_i) + (\ell-n_j) = \ell$ as desired.
%	 for every $i$ and $j$, $h$ induces a map between free $W(R)^\oplus$-modules 
%	\begin{align*}
%		V^{n_j}_{W(k_0)} \otimes_{W(k_0)} W(R)^\oplus 
%		%\hookrightarrow V\otimes_{\zz_p} W(R)^\oplus \to V\otimes_{\zz_p} W(R)^\oplus 
%		\to V^{n_i}_{W(k_0)} \otimes_{W(k_0)} W(R)^\oplus,
%	\end{align*}
%	and for $h$ to preserve the grading means 
%	\begin{align*}
%		h(V^{n_j}_{W(k_0)}\otimes_{W(k_0)} W(R)^\oplus_m) \subseteq V^{n_i}_{W(k_0)} \otimes_{W(k_0)} W(R)^\oplus_{m+n_j-n_i}
%	\end{align*}
%	for all $m \in \zz$, i.e. $h_{ij} \in W(R)^\oplus_{n_j-n_i}$. 
\end{proof}

It follows from the lemma that the assignment $g \mapsto \{\rho(g)\}_{(V,\rho)}$ induces a homomorphism of fpqc sheaves of groups on \textup{\textbf{Nilp}}$_{W(k_0)}$
\begin{align}\label{eq-hom}
\Psi: L^+_\mu G \to \underline{\textup{Aut}}^\otimes(\mathscr{C}_{\mu}).
\end{align}

\begin{thm}\label{thm-isom}
	The homomorphism \textup{(}\ref{eq-hom}\textup{)} is an isomorphism. \iffalse In particular, if $G$ is smooth over $\zz_p$, $\textup{\underline{Aut}}^\otimes(\mathscr{C}_{\mu})$ is flat and formally smooth over $W(k_0)$. \fi
\end{thm}
\begin{proof}
	For every $W(k_0)$-algebra $R$, $\Psi_R$ is the restriction of the map $G(W(R)^\oplus) \to \text{Aut}^\otimes(\omega_{W(R)^\oplus})$ given by $g \mapsto \{\rho(g)\}_{(V,\rho)}$. An inverse to this map is constructed in \cite[Theorem 44]{Cornut2014} (cf. also \cite[Theorem 9.2]{Milne2017}). We need only verify the restriction of this inverse to Aut$^\otimes(\mathscr{C}_{\mu,R})$ respects the grading. Let us review the construction of this map. 
	
	Denote by \textbf{Rep}$_{\zz_p}'(G)$ the category whose objects are the representations $(V,\rho)$ of $G$ on $\zz_p$-modules such that 
	\begin{align*}
		(V,\rho) = \varinjlim \ (W,\pi),
	\end{align*}
	where $(W,\pi)$ runs through the partially ordered set of all $G$-sub-representations of $(V,\rho)$ belonging to $\textbf{Rep}_{\zz_p}(G)$. The functor $\mathscr{C}_{\mu,R}$ extends to a functor  
	\begin{align*}
		\mathscr{C}_{\mu,R}': \textbf{Rep}'_{\zz_p}(G) \to \textbf{GrMod}(W(R)^\oplus), \ V \mapsto V_{W(k_0)} \otimes_{W(k_0)} W(R)^\oplus.
	\end{align*}
	If we denote by $\textup{\underline{Aut}}^\otimes(\mathscr{C}_\mu')$ the group-valued functor on $W(k_0)$-algebras given by $R \mapsto \textup{Aut}^\otimes(\mathscr{C}_{\mu,R}')$, then there is a canonical isomorphism
	\begin{align*}
		\textup{\underline{Aut}}^\otimes(\mathscr{C}_\mu) \xrightarrow{\sim} \textup{\underline{Aut}}^\otimes(\mathscr{C}_\mu').
	\end{align*}
	We will abuse notation and also denote the composition $L^+_\mu G \to \underline{\text{Aut}}^\otimes(\mathscr{C}_\mu')$ by $\Psi$. It is enough to define an inverse to this composition. 

	Write $G = $ Spec $A$. Recall the regular representation $\rho_{\text{reg}}$ is the representation of $G$ on $A$, viewed as a $\zz_p$-module, whose comodule morphism is the comultiplication for $A$. Explicitly, if $R$ is a $\zz_p$-algebra, $g \in G(R)$, and $a \in A\otimes_{\zz_p}R$, then $\rho_{\text{reg}}(g)\cdot a$ is defined by
	\begin{align*}
		(\rho_{\text{reg}}(g) \cdot a)(h) = a(hg).
	\end{align*}
	Since $G$ is a flat affine group scheme over a Noetherian ring, \cite[Cor. to Proposition 2]{Serre1968} implies that $\rho_{\text{reg}}$ is an object of \textbf{Rep}$_{\zz_p}'(G)$. One checks that the morphisms  
	\begin{align}\label{eq-G_equiv}
		(\zz_p, \mathbbm{1}) \to (A, \rho_{\text{reg}}) \ \text{ and } \ (A,\rho_{\text{reg}}) \otimes (A,\rho_{\text{reg}}) \to (A, \rho_{\text{reg}}),
	\end{align}
	given by the unit and multiplication respectively, are $G$-equivariant.
	
	Now let $R$ be a $W(k_0)$-algebra and let $\lambda \in \underline{\text{Aut}}^\otimes(\mathscr{C}'_{\mu, R})$, so for every $(V,\rho)$ in \textbf{Rep}$_{\zz_p}'(G)$, 
	\begin{align*}
		\lambda_\rho: V_{W(k_0)} \otimes_{W(k_0)} W(R)^\oplus \to V_{W(k_0)} \otimes_{W(k_0)} W(R)^\oplus
	\end{align*}
	is a graded $W(R)^\oplus$-module automorphism. In particular, $\lambda$ determines a graded $W(R)^\oplus$-module automorphism $\lambda_{\rho_{\text{reg}}}$ of $A \otimes_{\zz_p} W(R)^\oplus$. Moreover, since $\lambda$ is a morphism of tensor functors, functoriality applied to (\ref{eq-G_equiv}) implies $\lambda_{\rho_{\text{reg}}}$ is a graded $W(R)^\oplus$-algebra homomorphism. 
	
	Define $\Phi_R: \textup{\underline{Aut}}^\otimes(\mathscr{C}_{\mu,R}) \to G(W(R))^\oplus$ by assigning to $\lambda \in \textup{\underline{Aut}}^\otimes(\mathscr{C}_{\mu,R})$ the composition
	\begin{align*}
		A_{W(k_0)} \otimes_{W(k_0)} W(R)^\oplus \xrightarrow{\lambda_{\rho_{\text{reg}}}} A_{W(k_0)} \otimes_{W(k_0)} W(R)^\oplus \xrightarrow{\varepsilon \otimes \id_{W(R)^\oplus}} W(R)^\oplus,
	\end{align*}
	where the $\varepsilon$ is the counit for $A_{W(k_0)}$. We claim this composition is a graded $W(R)^\oplus$-module homomorphism, so $\Phi_R(\lambda) \in G(W(R)^\oplus)_\mu = L^+_\mu G(R)$. By assumption $\lambda_{\rho_\text{reg}}$ respects the grading, so it remains only to show $\varepsilon \otimes \id_{W(R)^\oplus}$ is a graded $W(R)^\oplus$-homomorphism. Because the zero element of $W(R)^\oplus$ is homogeneous of degree $n$ for all $n$, it is enough to show $\varepsilon(a) = 0$ if $a \in (A_{W(k_0)})_n$ for $n \ne 0$. Let $a \in (A_{W(k_0)})_n$ and $\lambda \in \mathbb{G}_m(W(k_0))$. Then
	\begin{align*}
		(\lambda \cdot a)(e) = a(\mu(\lambda)^{-1}e\mu(\lambda)) = a(e) = \varepsilon(a),
	\end{align*}
	where $e$ is the identity element of $G(W(k_0))$. But since $a \in (A_{W(k_0)})_n$, 
	\begin{align*}
		(\lambda \cdot a)(e) = \lambda^n a(e) = \lambda^n \varepsilon(a).
	\end{align*}
	Hence $(\lambda^n-1)\varepsilon(a) = 0$ in $W(k_0)$ for all $\lambda \in W(k_0)^\times$, so $\varepsilon(a) = 0$ if $n \ne 0$.
		
	The construction of $\Phi_R$ is functorial in $R$, so as $R$ varies we obtain a natural transformation
	\begin{align*}
		\Phi: \textup{\underline{Aut}}^\otimes(\mathscr{C}_{\mu}') \to L^+_\mu G,
	\end{align*}
	and the verifications in \cite[Theorem 44]{Cornut2014} show $\Phi$ and $\Psi$ compose to the identity in both directions.
	\iffalse
	Finally, if $G$ is smooth over $\zz_p$, then $L^+_\mu G$ is flat and formally smooth by Lemma \ref{lem-prosmooth}, so the same is true for $\textup{\underline{Aut}}^\otimes(\mathscr{C}_\mu)$. \fi
\end{proof}

It follows from the theorem that if $\mathscr{F}$ is a graded fiber functor of type $\mu$ over $W(R)^\oplus$, then \underline{Isom}$^\otimes(\mathscr{C}_{\mu,R}, \mathscr{F})$ is an $L^+_\mu G$-torsor, with $L^+_\mu G$ acting by pre-composition. This defines a functor
\begin{align} \label{eq-functor}
	\text{\textbf{GFF}}_\mu^W(R) \to \text{\textbf{Tors}}_{L^+_\mu G} (R), \ \mathscr{F} \mapsto \text{\underline{Isom}}^\otimes(\mathscr{C}_{\mu,R}, \mathscr{F}),
\end{align}
where \textbf{Tors}$_{L^+_\mu G}$ is the fpqc stack of $L^+_\mu G$-torsors over \textbf{Nilp}$_{W(k_0)}$.

\begin{cor}\label{cor-torsors}
	The functor \textup{(}\ref{eq-functor}\textup{)} induces an isomorphism of stacks
	\begin{align*}
		\textup{\textbf{GFF}}_\mu^W \xrightarrow{\sim} \textup{\textbf{Tors}}_{L^+_\mu G}.
	\end{align*}	
\end{cor}
\begin{proof}
	The substack $\textbf{GFF}^W_\mu$ of $\textbf{GFF}^W$ is the maximal sub-gerbe generated by $\mathscr{C}_{\mu}$, so $\textbf{GFF}^W_\mu \cong \textbf{Tors}_{\text{\underline{Aut}}^\otimes(\mathscr{C}_{\mu})}$ by \cite[Th\'eor\`eme 2.5.1]{Giraud1971}. Hence the result follows from Theorem \ref{thm-isom}.
\end{proof}

\subsection{Tannakian $(G,\mu)$-displays}
Let $G$ and $\mu$ be as in the previous section. In this section we give a Tannakian definition of $G$-displays of type $\mu$ and show that the resulting stack coincides with those defined by Lau and B\"ultel-Pappas. Let $R$ be a $p$-adic $W(k_0)$-algebra.

\begin{Def}
	A \textit{Tannakian $G$-display} over $R$ is an exact tensor functor
	\begin{align*}
		\mathscr{D}: \textup{\textbf{Rep}}_{\zz_p}(G) \to \textup{\textbf{Disp}}^W(R).
	\end{align*}
\end{Def}

As in the previous section such functors form a fibered category over \textbf{Nilp}$_{W(k_0)}$, which we will denote by $G$-\textbf{Disp}$^{W,\otimes}$.

\begin{lemma}\label{lem-displaystack}
	The fibered category $G$\textup{-\textbf{Disp}}$^{W,\otimes}$ is an fpqc stack in groupoids.
\end{lemma}
\begin{proof}
	The proof is essentially the same as that of Lemma \ref{lem-functor_stack}, after replacing \textbf{PGrMod}$^W$ by \textbf{Disp}$^W$ everywhere.
\end{proof}
Denote by $\upsilon_R$ the natural forgetful functor which sends a display $\underline{M} = (M, F)$ to its underlying finite projective graded $W(R)^\oplus$-module $M$. If $\mathscr{D}$ is a Tannakian $G$-display over $R$, then by composing with the forgetful functor $\upsilon_R$ we obtain a graded fiber functor $\upsilon_R \circ \mathscr{D}:$ \textbf{Rep}$_{\zz_p}(G) \to $ \textbf{PGrMod}$^W(R)$.

\begin{Def}
	A \textit{Tannakian $(G,\mu)$-display} over $R$ is a Tannakian $G$-display $\mathscr{D}$ over $R$ such that $\upsilon_R \circ \mathscr{D}$ is a graded fiber functor of type $\mu$.
\end{Def}

Denote by $G$-\textbf{Disp}$^{W,\otimes}_\mu$ the fibered category over \textbf{Nilp}$_{W(k_0)}$ whose fiber over $R$ is the category of Tannakian $(G,\mu)$-displays over $R$. Evidently  $G$\textup{-\textbf{Disp}}$^{W,\otimes}_\mu$ is a substack of $G\text{-}\textbf{Disp}^{W,\otimes}$.

\begin{cons}\label{cons_GDisp}
	Suppose $\mathscr{D}$ is a Tannakian $(G,\mu)$-display over $R$. We will associate to $\mathscr{D}$ a $G$-display of type $\mu$. By Corollary \ref{cor-torsors},
	\begin{align*}
		Q_{\mathscr{D}} := \text{\underline{Isom}}^\otimes(\mathscr{C}_{\mu,R}, \upsilon_R \circ \mathscr{D})
	\end{align*}
	is an $L^+_\mu G$-torsor over $R$. Let $R'$ be an $R$-algebra and suppose $\lambda: \mathscr{C}_{\mu,R'} \xrightarrow{\sim} \upsilon_{R'}\circ \mathscr{D}_{R'}$ is an isomorphism of tensor functors. If $(V,\rho)$ is in \textbf{Rep}$_{\zz_p}(G)$, write $\mathscr{D}_{R'}(V,\rho) = (M(\rho)',F(\rho)')$. Define $\alpha_{\mathscr{D}}(\lambda)_\rho$ as the composition
	\begin{align*}
		V \otimes_{\zz_p} W(R') \xrightarrow{(\lambda_\rho)^\sigma} (M(\rho)')^\sigma \xrightarrow{(F(\rho)')^\sharp} (M(\rho)')^\tau \xleftarrow{(\lambda_\rho)^\tau} V \otimes_{\zz_p} W(R').
	\end{align*}
	On the left we are implicitly identifying 
	\begin{align*}
		(V \otimes_{\zz_p} W(R')^\oplus)^\sigma \cong V\otimes_{\zz_p} W(R')
	\end{align*}
	using the isomorphism induced by the natural isomorphism of rings $W(R')^\oplus \otimes_{W(R')^\oplus, \sigma} W(R') \xrightarrow{\sim} W(R')$. We have a similar identification on the right when we replace $\sigma$ by $\tau$.

	Because $\lambda: \mathscr{C}_{\mu,R'} \to \upsilon_{R'} \circ \mathscr{D}_{R'}$ is a tensor morphism, it follows that $\{\alpha_{\mathscr{D}}(\lambda)_\rho\}_{(V,\rho)}$ is an element of $\textup{Aut}^\otimes(\omega_{W(R')})$, which is isomorphic to $G(W(R')) = L^+G(R')$ by Tannakian duality. Hence there is some $\alpha_{\mathscr{D}}(\lambda) \in L^+G(R)$ such that $\rho(\alpha_{\mathscr{D}}(\lambda)) = \alpha_{\mathscr{D}}(\lambda)_\rho$ for every $(V,\rho)$. Altogether we have a morphism of fpqc sheaves
	\begin{align*}
		\alpha_{\mathscr{D}}: Q_{\mathscr{D}} \to L^+G.
	\end{align*}
	It remains to show $\alpha_{\mathscr{D}}$ is $L^+_\mu G$-equivariant. For this let $h \in L^+_\mu G(R')$. Then $(\lambda \cdot h)_\rho$ is the composition
	\begin{align*}
		V\otimes_{\zz_p} W(R')^\oplus \xrightarrow{\rho(h)} V\otimes_{\zz_p} W(R')^\oplus \xrightarrow{\lambda_\rho} M(\rho)'.
	\end{align*}
	Hence we see $\alpha_{\mathscr{D}}(\lambda \cdot h)_\rho$ is given by 
	\begin{align*}
		(\rho(h)^\tau)^{-1} \circ ((\lambda_\rho)^\tau)^{-1} \circ F(\rho)^\sharp \circ (\lambda_\rho)^\sigma \circ (\rho(h))^\sigma %&= \rho(\tau(h^{-1})) \circ \alpha_{\mathscr{D}}(\lambda)_\rho \circ \rho(\sigma(h)) \\
	&= \rho(\tau(h^{-1})\cdot \alpha_{\mathscr{D}}(\lambda)\cdot \sigma(h)).
	\end{align*}
	By Tannakian duality again we obtain $\alpha_{\mathscr{D}}(\lambda \cdot h) = \tau(h^{-1})\cdot \alpha_{\mathscr{D}}(\lambda) \cdot \sigma(h)$, so $\alpha_{\mathscr{D}}$ is $L^+_\mu G$-equivariant. 
\end{cons}
The pair $\left(Q_{\mathscr{D}}, \alpha_{\mathscr{D}}\right)$ comprises a $G$-display of type $\mu$ in the sense of Definition \ref{def-gdisp}. Suppose now $\mathscr{D}_1$ and $\mathscr{D}_2$ are Tannakian $(G,\mu)$-displays over $R$, and write $\mathscr{D}_1(V,\rho) = (M_1(\rho), F_1(\rho))$ and $\mathscr{D}_2(V,\rho) = (M_2(\rho), F_2(\rho))$. Given a morphism of Tannakian $(G,\mu)$-displays $\psi: \mathscr{D}_1 \to \mathscr{D}_2$, we get a morphism 
\begin{align*}
	Q_{\mathscr{D}_1} = \underline{\text{Isom}}^\otimes(\mathscr{C}_{\mu,R}, \upsilon_R \circ \mathscr{D}_1) \to \underline{\text{Isom}}^\otimes(\mathscr{C}_{\mu,R}, \upsilon_R \circ \mathscr{D}_2) = Q_{\mathscr{D}_2}
\end{align*}
by post-composition with $\psi$. This is obviously a morphism of torsors, and if $(V,\rho)$ is a representation of $G$, $\lambda \in Q_{\mathscr{D}_1}(R')$, then $\alpha_{\mathscr{D}_2}(\upsilon_R(\psi) \circ \lambda)_\rho$ is given by
\begin{align*}
	(\lambda_\rho^\tau)^{-1} \circ ((\psi_{R'})_\rho^\tau)^{-1} \circ (F_2(\rho)')^\sharp \circ (\psi_{R'})_\rho^\sigma \circ \lambda_\rho^\sigma,
\end{align*}
where $(\mathscr{D}_i)_{R'}(V,\rho) = (M_i(\rho)', F_i(\rho)')$. But because $(\psi_{R'})\rho$ is a morphism of displays $M_1(\rho)_{\underline{W}(R')} \to M_2(\rho)_{\underline{W}(R')}$, this becomes
\begin{align*}
	(\lambda_\rho^\tau)^{-1} \circ (F_1(\rho)')^\sharp \circ \lambda_\rho^\sigma = \alpha_{\mathscr{D}_1}(\lambda)_\rho.
\end{align*}
We conclude that the morphism $Q_{\mathscr{D}_1} \to Q_{\mathscr{D}_2}$ is a morphism of $G$-displays of type $\mu$, so the construction $\mathscr{D} \mapsto (Q_\mathscr{D} ,\alpha_{\mathscr{D}})$ is functorial. Denote the resulting functor by $T_R$. 

This construction is evidently compatible with base change, so we obtain a morphism of stacks
\begin{align} \label{eq-morphism}
	T: G\textup{-\textbf{Disp}}^{W,\otimes}_\mu \to G\textup{-\textbf{Disp}}^W_\mu, \ \mathscr{D} \mapsto \left(Q_{\mathscr{D}}, \alpha_{\mathscr{D}}\right).
\end{align}

\begin{thm}\label{thm-equiv}
	The morphism \textup{(}\ref{eq-morphism}\textup{)} is an isomorphism of fpqc stacks over \textup{\textbf{Nilp}}$_{W(k_0)}$.
\end{thm}
\begin{proof}
	Fix a $p$-nilpotent $W(k_0)$-algebra $R$. It is immediate from Corollary \ref{cor-torsors} that $T_R$ is faithful. Let us prove it is full. Let $\mathscr{D}_1$ and $\mathscr{D}_2$ be Tannakian $(G,\mu)$-displays over $R$, and write
	\begin{align*}
		\mathscr{D}_1(V,\rho) = (M_1(\rho), F_1(\rho)) \ \text{ and } \ \mathscr{D}_2(V,\rho) = (M_2(\rho), F_2(\rho))
	\end{align*}
	for every representation $(V,\rho)$ of $G$. Suppose
	\begin{align*}
		\eta: (Q_{\mathscr{D}_1}, \alpha_{\mathscr{D}_1}) \to (Q_{\mathscr{D}_2}, \alpha_{\mathscr{D}_2}) 
	\end{align*}
	is a morphism of $G$-displays of type $\mu$. By Corollary \ref{cor-torsors} there exists some $\psi: \upsilon_R \circ \mathscr{D}_1 \to \upsilon_R \circ \mathscr{D}_2$ which induces $\eta$. For every representation $(V,\rho)$ in \textbf{Rep}$_{\zz_p}(G)$, we obtain a morphism of graded $W(R)^\oplus$-modules 
	\begin{align*}
		\psi_\rho: M_1(\rho) \to M_2(\rho).
	\end{align*}
	The collection of these morphisms is functorial and compatible with tensor product, so it remains only to show $\psi_\rho$ is compatible with $F_1(\rho)$ and $F_2(\rho)$. By Lemma \ref{lem-morphism_local} it is enough to check this condition after a faithfully flat extension of rings $R \to R'$. Choose such an extension with the property that $Q_{\mathscr{D}_1}(R')$ is nonempty, and suppose $\lambda: \mathscr{C}_{\mu,R'} \to \upsilon_{R'} \circ (\mathscr{D}_1)_{R'}$ is an isomorphism of graded fiber functors. 
	
	Let $(V,\rho)$ be a representation of $G$. For brevity, let us write $M_i(\rho)' := M_i(\rho)_{W(R')^\oplus}$ and $F_i(\rho)'$ for the base change of $F_i(\rho)$ to $W(R')^\oplus$. Consider the following diagram:
	\begin{center}
		\begin{tikzcd}
			(M_1(\rho)')^\sigma 
				\arrow[r, "(\lambda_\rho^\sigma)^{-1}"] \arrow[d, "(F_1(\rho)')^\sharp"']
			& V \otimes_{\zz_p} W(R') 
				\arrow[r, "="] \arrow[d, "\rho(\alpha_{\mathscr{D}_1}(\lambda))"']
			& V \otimes_{\zz_p} W(R') 
				\arrow[r, "(\psi_{R'} \circ \lambda)^\sigma_\rho"] \arrow[d, "\rho(\alpha_{\mathscr{D}_2}(\psi_{R'} \circ \lambda))"']
			& (M_2(\rho)')^\sigma
				\arrow[d, "(F_2(\rho)')^\sharp"']
			\\ (M_1(\rho)')^\tau
				\arrow[r, "(\lambda_\rho^\tau)^{-1}"]
			& V \otimes_{\zz_p} W(R') 
				\arrow[r, "="]
			& V \otimes_{\zz_p} W(R') 
				\arrow[r, "(\psi_{R'} \circ \lambda)_\rho^\tau"]
			& (M_2(\rho)')^\tau
		\end{tikzcd}
	\end{center}
	The left- and right-most squares commute by definition of $\alpha_{\mathscr{D}_i}$. Because $\eta$ is a morphism of $G$-displays of type $\mu$, we have
	\begin{align*}\
		\alpha_{\mathscr{D}_1}(\lambda) = \alpha_{\mathscr{D}_2}(\eta (\lambda)) = \alpha_{\mathscr{D}_2}(\psi_{R'} \circ \lambda).
	\end{align*}
	Therefore the middle square and hence the whole diagram commutes. But composition across the top is $(\psi_{R'})_\rho^\sigma$ and across the bottom is $(\psi_{R'})_\rho^\tau$, so commutativity of this diagram means that $(\psi_{R'})_\rho$ is a morphism of displays for every $(V,\rho)$, i.e. that $\psi$ is a morphism of Tannakian $(G,\mu)$-displays which induces $\eta$. We conclude $T_R$ is full.
	
	It remains to show $T_R$ is essentially surjective. Let $\underline{Q} = (Q,\alpha)$ be a $G$-display of type $\mu$ over $R$. By Corollary \ref{cor-torsors}, there is some graded fiber functor $\mathscr{F}$ of type $\mu$ such that $Q \cong \underline{\text{Isom}}^\otimes(\mathscr{C}_{\mu,R}, \mathscr{F})$. Write $\mathscr{F}(V,\rho) = M(\rho)$. By \cite[\href{https://stacks.math.columbia.edu/tag/046N}{Lemma 046N}]{stacks-project} it is enough to show the base change $\underline{Q}_{R'}$ is in the essential image of $T_{R'}$ for some faithfully flat extension $R \to R'$.
	
	Suppose $R \to R'$ is a faithfully flat extension such that $\text{Isom}^\otimes(\mathscr{C}_{\mu,R'},\mathscr{F}_{R'})$ is nonempty. Let $\lambda: \mathscr{C}_{\mu,R'} \to \mathscr{F}_{R'}$ be an isomorphism of graded fiber functors of type $\mu$. Then $\alpha(\lambda) \in L^+G(R')$, so $\rho(\alpha(\lambda))$ is an automorphism of $V \otimes_{\zz_p} W(R')$ for every $(V,\rho)$. Define $F(\rho)'$ to be the $\sigma$-linear homomorphism $M(\rho)' \to (M(\rho)')^\tau$ such that
	\begin{align*}
		(F(\rho)')^\sharp := \lambda_\rho^\tau \circ \rho(\alpha(\lambda)) \circ (\lambda_\rho^\sigma)^{-1}.
	\end{align*}
	Then $\underline{M(\rho)'} = (M(\rho)', F(\rho)')$ is a display over $W(R')^\oplus$. We claim the association $\mathscr{D}_{\underline{Q}}': (V,\rho) \mapsto \underline{M(\rho)'}$ is a Tannakian $(G,\mu)$-display over $R'$. 
	
	First let us show $\mathscr{D}_{\underline{Q}}'$ is functorial. Suppose $\varphi: (V,\rho) \to (U,\pi)$ is a morphism in \textbf{Rep}$_{\zz_p}(G)$. Then $\mathscr{F}(\varphi)$ is a homomorphism of graded $W(R')^\oplus$-modules $M(\rho) \to M(\pi)$, and we need to show that $\mathscr{F}(\varphi)$ is compatible with $F(\rho)'$ and $F(\pi)'$. Consider the following diagram:
	\begin{center}
		\begin{tikzcd}
			(M(\rho)')^\sigma 
				\arrow[r, "(\lambda_\rho^\sigma)^{-1}"] \arrow[d, "(F(\rho)')^\sharp"]
			& V \otimes_{\zz_p} W(R') 
				\arrow[r, "\varphi_{W(R')}"] \arrow[d, "\rho(\alpha(\lambda))"]
			& U \otimes_{\zz_p} W(R')
				\arrow[r, "\lambda_\pi^\sigma"] \arrow[d, "\pi(\alpha(\lambda))"]
			& (M(\pi)')^\sigma
				\arrow[d, "(F(\pi)')^\sharp"]
			\\ (M(\rho)')^\tau
				\arrow[r, "(\lambda_\rho^\tau)^{-1}"] 
			& V \otimes_{\zz_p} W(R')
				\arrow[r, "\varphi_{W(R')}"]
			& U \otimes_{\zz_p} W(R) 
				\arrow[r, "\lambda_\pi^\tau"]
			& (M(\pi)')^\tau
		\end{tikzcd}
	\end{center}
	Again, the outside squares commute by definition of $F(\rho)'$ and $F(\pi)'$. The middle square commutes because $\varphi$ is a morphism in \textbf{Rep}$_{\zz_p}(G)$. Since $\lambda$ is a natural transformation, composition across the top is $\mathscr{F}(\varphi)^\sigma$, and composition across the bottom is $\mathscr{F}(\varphi)^\tau$. Hence $\mathscr{D}_{\underline{Q}}'$ is a functor. A completely analogous argument proves that it is compatible with the tensor product, so $\mathscr{D}_{\underline{Q}}'$ is a Tannakian $(G,\mu)$-display over $R'$.

Now consider $\underline{Q}_{\mathscr{D}'_{\underline{Q}}}$, the $G$-display of type $\mu$ associated to $\mathscr{D}
_{\underline{Q}}$. By definition of $Q_{\mathscr{D}'_{\underline{Q}}}$ and construction of $\mathscr{D}'_{\underline{Q}}$, we have
	\begin{align*}
		Q_{\mathscr{D}'_{\underline{Q}}} = \underline{\text{Isom}}^\otimes(\mathscr{C}_{\mu,R'},\upsilon_{R'} \circ \mathscr{D}'_{\underline{Q}}) = \underline{\text{Isom}}^\otimes(\mathscr{C}_{\mu,R'}, \mathscr{F}_{R'}) \cong Q_{R'}.
	\end{align*}
	By Tannakian duality, under this identification we have $\alpha_{R'} = \alpha_{\mathscr{D}'_{\underline{Q}}}$. Hence $\underline{Q}_{R'} \cong T_{R'}(\mathscr{D}'_{\underline{Q}})$, and $T$ is essentially surjective.
\end{proof}

Combining the theorem with Lemma \ref{lem-bp} we obtain the following corollary:
\begin{cor}
	If $G$ is a reductive group scheme over $\zz_p$ and $\mu$ is a minuscule cocharacter defined over $W(k_0)$ then the stack of $(G,\mu)$-displays \textup(as in \textup{\cite{BP2017}}\textup) is isomorphic to the stack of Tannakian $(G,\mu)$-displays.
\end{cor}

In \cite{BP2017}, a $(G,\mu)$-display is called banal if the underlying torsor is trivial. We close this section by defining the analogous notion for Tannakian $(G,\mu)$-displays, and by giving a local description of the stack of Tannakian $(G,\mu)$-displays which is formally very similar to that of \cite[3.2.7]{BP2017}. Fix a $p$-nilpotent $W(k_0)$-algebra $R$.

\begin{Def}
	A Tannakian $(G,\mu)$-display $\mathscr{D}$ over $R$ is \textit{banal} if there is an isomorphism $\upsilon_R \circ \mathscr{D} \cong \mathscr{C}_{\mu,R}$.
\end{Def}
Banal Tannakian $(G,\mu)$-displays over $R$ form a full subcategory of Tannakian $(G,\mu)$-displays over $R$, and by definition any Tannakian $(G,\mu)$-display is fpqc-locally banal. 
\begin{cons}\label{cons-DU}
	To any $U \in L^+G(R)$ we can associate a banal $(G,\mu)$-display $\mathscr{D}_U$ as follows. Let $(V,\rho)$ be a representation of $G$ on a finite free $\zz_p$-module. To $(V,\rho)$ we associate the following standard datum:
	\begin{itemize}
		\item $L = V\otimes_{\zz_p} W(R)$ viewed as a graded $W(R)$-module with $L_i = V_{W(k_0)}^i \otimes_{W(k_0)} W(R)$, where $V_{W(k_0)}^i$ is the decomposition (\ref{eq_decomp}) of $V\otimes_{\zz_p} W(k_0)$ induced by $\mu$;
		\item $\Phi_U: L \to L$ is the composition $\rho(U) \circ (\id_V \otimes f)$. 
	\end{itemize}
	Write $F_U: L \otimes_{W(R)} W(R)^\oplus \to L \otimes_{W(R)} W(R)^\oplus$ for the resulting $\sigma$-linear map, defined explicitly as 
	\begin{align*}
	F_U(x \otimes \xi \otimes s) = \sigma(s)\rho(U)(x \otimes f(\xi))
	\end{align*}
	for $x \in V,$ $\xi \in W(R)$ and $s \in W(R)^\oplus$. Then $\mathscr{D}_U$ is the Tannakian $(G,\mu)$-display
	\begin{align*}
		\textbf{Rep}_{\zz_p}(G) \to \textbf{Disp}^W(R), \ (V,\rho) \mapsto (L \otimes_{W(R)} W(R)^\oplus, F_U).
	\end{align*}
	By construction it is clear that $\upsilon_R \circ \mathscr{D}_U = \mathscr{C}_{\mu,R}$, so $\mathscr{D}_U$ is indeed banal.
\end{cons}

In fact, the following lemma shows any banal Tannakian $(G,\mu)$-display is isomorphic to $\mathscr{D}_U$ for some $U \in L^+G(R)$.
\begin{lemma}\label{lemma-local}
	Let $\mathscr{D}$ be a banal Tannakian $(G,\mu)$-display over $R$, and let and let $\lambda: \mathscr{C}_{\mu, R} \xrightarrow{\sim} \upsilon_R \circ \mathscr{D}$ be an isomorphism of graded fiber functors. Define $U = \alpha_{\mathscr{D}}(\lambda)$ as in Construction \ref{cons_GDisp}. Then $\mathscr{D} \cong \mathscr{D}_U$.
\end{lemma}
\begin{proof}
	It is clear that $\lambda$ provides an isomorphism between the underlying graded fiber functors. Let $(V,\rho)$ be a representation, and write $\mathscr{D}(V,\rho) = (M(\rho),F(\rho))$. In order to see that $\lambda$ is compatible with the display structures, we need to check $\lambda_\rho^\tau \circ \rho(U) = F(\rho)^\sharp \circ \lambda_\rho^\sigma$. But $U = \alpha_{\mathscr{D}}(\lambda)$, so this follows from Construction \ref{cons_GDisp}.
\end{proof}

We can now give an explicit description of the category of banal Tannakian $(G,\mu)$-displays. This corresponds to the description of banal $(G,\mu)$-displays given in \cite[3.2.7]{BP2017} when $G$ is reductive and $\mu$ is minuscule. Define a category $[L^+G/L^+_\mu G]^\text{pre}(R)$ as follows:
\begin{itemize}
	\item The objects in $[L^+G/L^+_\mu G]^\text{pre}(R)$ are elements $U \in L^+G(R)$;
	\item given $U$, $U'$ in $L^+G(R)$, the set of morphisms $U$ to $U'$ is given by
	\begin{align*}
	\text{Hom}(U,U') = \{h \in L^+G_\mu(R) \mid \tau(h)^{-1}U'\sigma(h) = U\}.
	\end{align*}
\end{itemize}

\begin{prop}\label{prop-local}
	The category of banal Tannakian $(G,\mu)$-displays over $R$ is equivalent to the category $[L^+G/L^+_\mu G]^\text{pre}(R)$. 
\end{prop}
\begin{proof}
	We claim the assignment $U \mapsto \mathscr{D}_U$ determines a functor from $[L^+G/L^+_\mu G]^\text{pre}(R)$ to the category of banal Tannakian $(G,\mu)$-displays over $R$. Let $U$, $U'$ in $L^+G(R)$ and $h \in \text{Hom}(U,U')$. Applying the homomorphism $\Psi$ (cf. (\ref{eq-hom})) to $h$ we obtain a morphism $\Psi(h)$ of the underlying graded fiber functors of $\mathscr{D}_U$ and $\mathscr{D}_{U'}$. These are both equal to $\mathscr{C}_{\mu,R}$, so $\Psi(h) \in \text{Aut}^\otimes(\mathscr{C}_{\mu,R}) = L^+_\mu G(R)$. The condition $\tau(h)^{-1}U'\sigma(h) = U$ exactly corresponds to the condition that $\Psi(h)$ determines a morphism of Tannakian $(G,\mu)$-displays $\mathscr{D}_U \to \mathscr{D}_{U'}$, so the above functor is well-defined. That the functor is fully faithful is an immediate consequence of Theorem \ref{thm-isom}, and that it is essentially surjective follows from Lemma \ref{lemma-local}.
\end{proof}
\subsection{G-quasi-isogenies}
Suppose $R$ is a $p$-adic $\zz_p$-algebra. Then $W(R)$ is endowed with a natural structure of a $\zz_p$-algebra via 
\begin{align*}
\zz_p \xrightarrow{\Delta} W(\zz_p) \to W(R).
\end{align*}
%This homomorphism is injective, so the localization $W(R)[1/p]$ is necessarily nonzero. 
The Frobenius and Verschiebung for $W(R)$ extend in a natural way to $W(R)[1/p]$.

\begin{Def}\label{def-isodisp}
	An \textit{isodisplay} over $R$ is a pair $\underline{N} = (N, \varphi)$ where $N$ is a finitely generated projective $W(R)[1/p]$-module and $\varphi: N \to N$ is an $f$-linear isomorphism. 
\end{Def}

The category of isodisplays over $R$ has a natural structure of an exact tensor category, with tensor product defined by $(N_1, \varphi_1) \otimes (N_2, \varphi_2) := (N_1 \otimes N_2, \varphi_1 \otimes \varphi_2)$, and with exactness inherited from the analogous category defined by omitting the finitely generated projective condition for the $W(R)[1/p]$-modules $N$.

Let $\underline{M} = (M,F)$ be a display over $\underline{W}(R)$, and suppose the depth of $M$ is $d$ (cf.\,Definition \ref{def-depth}). Then we can associate to $\underline{M}$ an isodisplay as follows. Because $d(M) = d$, we have an isomorphism of $W(R)$-modules $M_d \xrightarrow{\theta_d} M^\tau$, cf. Lemma \ref{lem-theta}. Define $\varphi$ as follows: first consider the composition 
\begin{align*}
\varphi': M^\tau \xrightarrow{\theta_d^{-1}} M_d \xrightarrow{F_d} M^\tau,
\end{align*}
where $F_d$ is the restriction of $F$ to $M_d$. This is an $f$-linear endomorphism of $M^\tau$. We claim it induces an $f$-linear automorphism of $M^\tau[1/p]$. Indeed, we can choose a standard datum $(L, \Phi)$, with $L = \bigoplus_{i=d}^a L_i$, so $M = L\otimes_{W(R)} W(R)^\oplus$ and $F(x\otimes s) = \sigma(s) \Phi(x)$ for $x \in L$, $s \in W(R)^\oplus$. Then 
\begin{align*}
	M_d = \bigoplus_{n=0}^{a-d} \left(L_{d+n} \otimes_{W(R)} W(R)^\oplus_{-n}\right),
\end{align*}
and $F_d$ becomes the composition
\begin{align*}
	\bigoplus_{n=0}^{a-d} L_{d+n} \otimes_{W(R)} W(R)^\oplus_{-n} \xrightarrow{\bigoplus \id \otimes p^n} \bigoplus_{n=0}^{a-d} L_{d+n}\otimes_{W(R)} W(R)^\oplus_{-n} \xrightarrow{\bigoplus \id\otimes\tau_{-n}} \bigoplus_{n=0}^{a-d} L_{d+n} \xrightarrow{\Phi} \bigoplus_{n=0}^{a-d} L_{d+n}
\end{align*}
The last map is an $f$-linear bijection by assumption, and the second map is a bijection by definition of $\tau$. The first becomes a bijection after we invert $p$, so this proves the claim. 

Now define $\varphi := p^d \varphi'[1/p]$. By this procedure we obtain an isodisplay $\underline{M}[1/p]=(M^\tau[1/p],\varphi)$. This construction is evidently functorial, so if we denote the category of isodisplays over $R$ by \textbf{Isodisp}$(R)$, we obtain an exact tensor functor
\begin{align}\label{eq-isodisp}
\text{\textbf{Disp}}^{W}(R) \to \text{\textbf{Isodisp}}(R).
\end{align}
This procedure generalizes the one given for assigning an isodisplay to a $1$-display in \cite[Example 63]{Zink2002}. 
\begin{rmk} 
	If $\underline{S}$ is a frame over $R$, we can define an analogous category of  isodisplays over $\underline{S}$. If every finite projective graded $S$-module $M$ admits a normal decomposition, then Lemma \ref{lem-theta} gives an isomorphism of $S_0$-modules $M_d \xrightarrow{\sim} M^\tau$, where $d = d(M)$, and we can mimic the construction above to define a functor analogous to (\ref{eq-isodisp}). This holds in particular when $S_0$ is $p$-adic by \cite[Lemma 3.1.4]{Lau2018}. However, without the guaranteed existence of normal decompositions, it is unclear to the author whether $\theta_d$ is an isomorphism in general.
\end{rmk}
\begin{Def}
	Let $\underline{M} = (M, F)$ and $\underline{M}' = (M',F')$ be displays over $\underline{W}(R)$.
	\begin{enumerate}[(i)]
		\item A \textit{quasi-isogeny} $\gamma: \underline{M} \dashrightarrow \underline{M'}$ is an isomorphism of  isodisplays $\underline{M}[1/p] \xrightarrow{\sim} \underline{M'}[1/p]$.
		\item  A quasi-isogeny is an \textit{isogeny} if it is induced by a morphism of displays. 
	\end{enumerate}
	We say $\underline{M}$ is \textit{isogenous} to $\underline{M'}$ if there exists an isogeny $\underline{M} \to \underline{M'}$.
\end{Def}

Now suppose $G$ is a flat affine group scheme over $\zz_p$, and that $\mu$ is a cocharacter for $G$ defined over $W(k_0)$. Let $R$ be a $W(k_0)$-algebra. 

\begin{Def}
	A \textit{$G$-isodisplay} over $R$ is an exact tensor functor \textbf{Rep}$_{\zz_p}(G) \to \text{\textbf{Isodisp}}(R)$.
\end{Def}

If $\mathscr{D}$ is a Tannakian $(G,\mu)$-display, then we obtain a $G$-isodisplay by composition with the natural functor (\ref{eq-isodisp}). Denote the resulting $G$-isodisplay by $\mathscr{D}[1/p]$.

\begin{Def}
	Let $\mathscr{D}_1$ and $\mathscr{D}_2$ be Tannakian $(G,\mu)$-displays. A \textit{$G$-quasi-isogeny} $\mathscr{D}_1 \dashrightarrow \mathscr{D}_2$ is an isomorphism of $G$-isodisplays $\mathscr{D}_1[1/p] \xrightarrow{\sim} \mathscr{D}_2[1/p]$.
\end{Def}

Suppose $\mathscr{D} \cong \mathscr{D}_U$ is a banal Tannakian $(G,\mu)$-display over $R$, given by $U \in L^+G(R)$ as in the previous section. Then we can explicitly compute the resulting $G$-isodisplay. In this case, if $(V,\rho)$ is a representation of $G$, we have $\mathscr{D}[1/p](V,\rho) = (N(\rho), \varphi(\rho))$, where 
\begin{align*}
	N(\rho) = V \otimes_{\zz_p} W(R)[1/p].
\end{align*}

\begin{lemma}\label{lem-froblocal}
	For every $(V,\rho)$ in \textup{\textbf{Rep}}$_{\zz_p}(G)$, the Frobenius on $\mathscr{D}[1/p](V,\rho)$ is given by 
	\begin{align*}
		\varphi = \rho\left(U\mu^\sigma(p)\right)\circ(\textup{id}_V \otimes f).
	\end{align*}
\end{lemma}
\begin{proof}
	Fix $(V,\rho)$ in \textbf{Rep}$_{\zz_p}(G)$, and let $\mathscr{D}(V,\rho) = (M(\rho),F(\rho))$, where $M(\rho) = V\otimes_{\zz_p} W(R)^\oplus$. Because $\mathscr{D}$ is banal and defined from $U \in L^+G(R)$, $F$ is defined from the $f$-linear automorphism 
	\begin{align*}
		\Phi_U = \rho(U) \circ (\id_V \otimes f)
	\end{align*}
	of $V\otimes_{\zz_p} W(R)$. 
	Suppose the weights of $\rho \circ \mu$ are $\{w_1, \dots, w_r\}$ with $w_1 \le w_2 \le \dots \le w_r$. Then $d(M(\rho)) = w_1$. 
	
	Let $L_i = V_{W(k_0)}^{i} \otimes_{W(k_0)} W(R)$. Then we have 
	\begin{align*}
		M(\rho)^\tau \cong V \otimes_{\zz_p} W(R) = \bigoplus_{i = 1}^r L_{w_i}.
	\end{align*}
	The $f$-linear automorphism $\varphi$ of $N(\rho) = \left(M(\rho)^\tau\right)[1/p]$ is constructed in two steps. First we consider the composition 
	\begin{align*}
		M(\rho)^\tau \xrightarrow{\theta_{w_1}^{-1}} M(\rho)_{w_1} \xrightarrow{F_{w_1}} M(\rho)^\tau.
	\end{align*}
	In our case, if $x \in L_{w_i}$, then
	\begin{align*}
		\theta_{w_1}^{-1}(x) = x \otimes t^{w_i-w_1} \in L_{w_i} \otimes_{W(R)} W(R)^\oplus_{w_1-w_i},
	\end{align*}
	where $t \in W(R)^\oplus_1$ is the indeterminate from Definition \ref{def-wittframe}. Applying $F_{w_1}$, we obtain
	\begin{align*}
		F_{w_1}(x \otimes t^{w_i-w_1}) = p^{w_i-w_1} \rho(U)(\id_V \otimes f)(x). 
	\end{align*}
	Multiplying by $p^{d(M)}$, we have 
	\begin{align*}
		\varphi(x) = p^{w_i}\rho(U)(\id_V\otimes f)(x).
	\end{align*}
	for $x \in V^{w_i}_{W(k_0)}\otimes_{W(k_0)} W(R)[1/p]$.	This is evidently the same as $\rho(U)(\id_V \otimes f) \rho(\mu(p))(x)$, so the result follows from the identity
	\begin{align*}
		(\id_V \otimes f) \circ \rho(\mu(p)) = \rho(\mu^\sigma(p)) \circ (\id_V \otimes f).
	\end{align*}

\end{proof}

\section{RZ Spaces}

\subsection{Local Shimura data and the RZ functor}\label{sub-localdata}
We recall the formalism for local Shimura data developed in \cite{RV2014}, and we give a purely group theoretic definition of an RZ functor, following \cite{BP2017}. Our definition relies on the framework developed in the previous section, and as such, allows for a more general development than that in \cite{BP2017}. In particular, we do not need to assume $G$ is reductive or $\mu$ minuscule in order to formulate the definitions in this section.

Let $k$ be an algebraic closure of $\mathbb{F}_p$, and let $W(k)$ be the ring of Witt vectors over $k$. Write $K = W(k)[1/p]$, and let $\bar{K}$ be an algebraic closure of $K$. In this section we write $\sigma$ for the automorphism of $K$ coming from a lift of the absolute Frobenius $x \mapsto x^p$ of $k$. Let $G$ be a smooth affine group scheme over $\zz_p$ whose generic fiber $G_{\qq_p}$ is reductive. 
Consider pairs $(\{\mu\},[b])$ such that
\begin{itemize}
	\item $\{\mu\}$ is a $G(\bar{K})$-conjugacy class of cocharacters ${\mathbb{G}_m}_{\bar{K}} \to G_{\bar{K}}$;
	\item $[b]$ is a $\sigma$-conjugacy class elements $b \in G(K)$.
\end{itemize}
Let $E \subseteq \bar{K}$ be the field of definition of the conjugacy class $\{\mu\}$. Denote by $\mathcal{O}_E$ its valuation ring and $k_E$ its residue field. We make the following assumption:
\begin{ass}\label{ass_intmu}
	The field $E \subseteq \bar{K}$ is contained in $K$, and there exists a cocharacter $\mu: {\mathbb{G}_m}_{E} \to G_E$	in the conjugacy class $\{\mu\}$ which is defined over $E$ and which extends to an integral cocharacter 
	\begin{align*}
		\mu: {\mathbb{G}_m}_{\mathcal{O}_E} \to G_{\mathcal{O}_E}.
	\end{align*}
\end{ass} 
When the assumption is satisfied we may identify $\mathcal{O}_E\cong W(k_E)$ and $E \cong W(k_E)[1/p]$.

\begin{Def}\label{def-Shimdatum}
	A \textit{local integral Shimura datum} is a triple $(G,\{\mu\},[b])$ as above such that 
	\begin{enumerate}[(i)]
		\item $\{\mu\}$ is minuscule and satisfies Assumption \ref{ass_intmu}, and
		\item for any integral representative $\mu$ of $\{\mu\}$ as in Assumption \ref{ass_intmu}, the $\sigma$-conjugacy class $[b]$ has a representative
		\begin{align*}
			b \in G(W(k))\mu^\sigma(p)G(W(k)).
		\end{align*}
	\end{enumerate}
\end{Def} 

\begin{Def}\label{def-framingpair}
	Let $(G,\{\mu\},[b])$ be a local integral Shimura datum. A \textit{framing pair} for $(G,\{\mu\},[b])$ is a pair $(\mu, b)$ where 
	\begin{itemize}
		\item $\mu:\mathbb{G}_{m,W(k_E)} \to G_{W(k_E)}$ is a representative of the conjugacy class $\{\mu\}$ as in Assumption \ref{ass_intmu},
		\item $b$ is a representative of the $\sigma$-conjugacy class $[b]$ such that, for some $u \in L^+G(k)$,
		\begin{align}\label{eq-b}
			b = u\mu^\sigma(p).
		\end{align} 
	\end{itemize}
\end{Def}
It follows from Definition \ref{def-Shimdatum} that a framing pair always exists for a local integral Shimura datum $(G, \{\mu\}, [b])$. If $(\mu,b)$ is a framing pair, then the element $u\in L^+G(k)$ such that $b = u \mu^\sigma(p)$ is uniquely determined. 
\begin{Def}\label{def-framingobject}
	Let $(\mu,b)$ be a framing pair for $(G,\{\mu\},[b])$, and let $u \in L^+G(k)$ be the unique element such that $b = u \mu^\sigma(p)$. The \textit{framing object} associated to $(\mu,b)$ is the banal Tannakian $(G,\mu)$-display $\mathscr{D}_u$ associated to $u$ by Construction \ref{cons-DU}. 
\end{Def}

\begin{Def}\label{def-RZfunctor}
	Fix a framing pair $(\mu,b)$ for $(G,\{\mu\},[b])$, and let $\mathscr{D}_0$ be the associated framing object. The \textit{RZ-functor} associated to the triple $(G,\mu,b)$ is the functor on \textbf{Nilp}$_{W(k)}$ which assigns to a $p$-nilpotent $W(k)$-algebra $R$ the set of isomorphism classes of pairs $(\mathscr{D},\iota)$, where
	\begin{itemize}
		\item $\mathscr{D}$ is a Tannakian $(G,\mu)$-display over $R$,
		\item $\iota: \mathscr{D}_{R/pR} \dashrightarrow (\mathscr{D}_0)_{R/pR}$ is a $G$-quasi-isogeny. 
	\end{itemize}
	Two pairs $(\mathscr{D}_1, \iota_1)$ and $(\mathscr{D}_2,\iota_2)$ are isomorphic if there is an isomorphism $\mathscr{D}_1 \xrightarrow{\sim} \mathscr{D}_2$ lifting $\iota_2^{-1} \circ \iota_1$. Denote the RZ functor associated to $(G,\mu,b)$ by $\textup{RZ}_{G,\mu,b}$. 
\end{Def}
Associated to $\textup{RZ}_{G,\mu,b}$ we have a category $\textbf{RZ}_{G,\mu,b}$ fibered over \textbf{Nilp}$_{W(k)}$, such that if $R$ is a $p$-nilpotent $W(k)$-algebra, then $\textbf{RZ}_{G,\mu,b}(R)$ is the groupoid of pairs $(\mathscr{D},\iota)$ over $R$ as in Definition \ref{def-RZfunctor}. It follows from Lemma \ref{lem-displaystack} that $\textbf{RZ}_{G,\mu,b}$ is an fpqc stack in groupoids.

\subsection{Realization as a quotient stack}
In this section we will reinterpret $\textbf{RZ}_{G,\mu,b}$ as a quotient stack. From this we obtain an equivalence between our RZ functor and the one defined in \cite{BP2017}, in the case where both are defined.

First we recall the definition of the Witt loop scheme as in \cite[Section 2.2]{BP2017}. Let $R$ be a ring and let $X$ be an affine scheme of finite type over $W(R)$. Then the functor on $R$-algebras
\begin{align*}
	R' \mapsto X(W(R')[1/p])
\end{align*}
is representable by an ind-scheme over $R$ by \cite[Proposition 32]{Kreidl2014}. If $X$ is a $W(k_0)$-scheme, we can apply this to the base change of $X$ along the Cartier homomorphism $W(k_0) \xrightarrow{\Delta} W(W(k_0))$ to obtain an ind-scheme over $W(k_0)$. We will denote this ind-scheme by $LX$.

For such a scheme $X$, denote by $^\sigma X$ the base change of $X$ via the automorphism $\sigma$ of $W(k_0)$:
\begin{align*}
	^\sigma X = X \times_{\text{Spec }W(k_0),\sigma} \text{Spec }W(k_0).
\end{align*}
There is a natural isomorphism $^\sigma(LX) \xrightarrow{\sim} L( ^\sigma X)$. If $R$ is a $W(k_0)$-algebra, then the Witt vector frobenius $f$ on $W(R)$ induces a map on $R$ points $f: LX(R) \to {^\sigma(}LX)(R)$ as follows: if $x \in LX(R)$, then $f(x)$ is the composition
\begin{align*}
	\text{Spec }W(R)[1/p] \xrightarrow{f} \text{Spec }W(R)[1/p] \xrightarrow{x} X.
\end{align*}
This map is functorial in $R$ and hence defines a morphism of ind-schemes $f: LX \to {^\sigma L}X$ over $W(k_0)$. If $X$ is defined over $\zz_p$, there is a natural isomorphism $^\sigma LX \xrightarrow{\sim} LX$, so in this case $f$ defines an endomorphism $f: LX \to LX$. If $G$ is a group scheme over $W(k_0)$, then $LG$ is a group ind-scheme over $W(k_0)$, and in this case $f$ is a group ind-scheme homomorphism. 

Let $(G,\{\mu\},[b])$ be a local integral Shimura datum, and choose a framing pair $(\mu,b)$ for $(G,\{\mu\},[b])$, so $b = u \mu^\sigma(p)$ for some $u \in L^+G(k)$. To $b$ and $\mu$ we associate two morphisms:
\begin{align*}
	c_b: LG \to LG, \ g \mapsto g^{-1}\cdot b\cdot f(g), \\
	m_\mu: L^+G \to LG, \ U \mapsto U\cdot\mu^\sigma(p).
\end{align*}
Using these morphisms we form the fiber product $L^+G \times_{m_\mu, c_b} LG$, defined by the following Cartesian diagram:
\begin{center}
	\begin{tikzcd}
		L^+G \times_{m_\mu, c_b} LG 
			\arrow[r] \arrow[d]
		& LG
			\arrow[d, "c_b"]
		\\ L^+G 	
			\arrow[r, "m_\mu"]
		& LG
	\end{tikzcd}
\end{center}
\begin{lemma}\label{lem-RZaction}
	Suppose $h \in L^+_\mu G(R)$, and $(U,g) \in (L^+G \times_{m_\mu, c_b} LG)(R)$. Then 
	\begin{align}\label{eq-RZaction}
		(U,g)\cdot h := (\tau(h)^{-1}\cdot U \cdot \sigma(h), g \cdot \tau(h))
	\end{align}
	is an element of $(L^+G \times_{m_\mu, c_b} LG)(R)$.
\end{lemma}
\begin{proof}
	We need to show
	\begin{align*}
		c_b(g\tau(h)) = m_\mu(\tau(h)^{-1}U\sigma(h)),
	\end{align*}
	which reduces to showing
	\begin{align*}
		\sigma(h) = \mu^\sigma(p) f(\tau(h)) \mu^\sigma(p)^{-1}
	\end{align*}
	when viewed as an element of $LG(R)$. This identity can be obtained in a straightforward manner by reformulating the last line of the proof of \cite[Lemma 5.2.1]{Lau2018}. 
\end{proof}

It follows from the lemma (\ref{eq-RZaction}) determines an action of $L^+_\mu G$ on $L^+G\times_{m_\mu, c_b} LG$. Using this action we form the quotient stack
\begin{align*}
	[(L^+G\times_{m_\mu,c_b} LG)/L^+_\mu G]
\end{align*}
over \textbf{Nilp}$_{W(k)}$. Explicitly, if $R$ is a $W(k)$-algebra, an object in $[(L^+G\times_{m_\mu, c_b} LG) /L^+_\mu G]$ is a pair $(Q,\beta)$ consisting of an $L^+_\mu G$-torsor $Q$ and a morphism $\beta: Q \to L^+G \times_{m_\mu, c_b} LG$ which is equivariant with respect to the action (\ref{eq-RZaction}).

\begin{thm}\label{prop-rzquotient}
	Let $(\mu,b)$ be a framing pair for $(G,\{\mu\},[b])$. Then there is an isomorphism of stacks 
	\begin{align*}
		\textbf{\textup{RZ}}_{G,\mu,b} \xrightarrow{\sim} [(L^+G \times_{m_\mu, c_b} LG) / L^+_\mu G].
	\end{align*}
\end{thm}
\begin{proof}
	Let $R$ be a $p$-nilpotent $W(k)$-algebra. Say $(\mathscr{D},\iota) \in RZ_{G,\mu,b}$ is banal if $\mathscr{D}$ is banal. Since any $\mathscr{D}$ is fpqc-locally banal, it is enough to show that banal $(\mathscr{D},\iota)$ in $RZ_{G,\mu,b}(R)$ are given by pairs $(U,g) \in (L^+G\times_{m_\mu,c_b} LG)(R)$, and that, if $(\mathscr{D}_1,\iota_1)$, $(\mathscr{D}_2,\iota_2)$ are banal, corresponding to $(U_1,g_1)$ and $(U_2,g_2)$ respectively, then morphisms $(\mathscr{D}_1, \iota_1) \to (\mathscr{D}_2,\iota_2)$ are given by $h \in L^+_\mu G(R)$ such that $(U_2,g_2)\cdot h = (U_1, g_1)$. 
	
	If $(\mathscr{D},\iota) \in RZ_{G,\mu,b}(R)$ is banal, then by Lemma \ref{lemma-local}, $\mathscr{D} \cong \mathscr{D}_U$ for some $U \in L^+G(R)$. To give the quasi-isogeny $\iota$ is to give, for every $(V,\rho)$ in $\textbf{Rep}_{\zz_p}(G)$, an automorphism of $W(R/pR)[1/p]$-modules
	\begin{align*}
		\iota_\rho: V \otimes_{\zz_p} W(R/pR)[1/p] \xrightarrow{\sim} V \otimes_{\zz_p} W(R/pR)[1/p]
	\end{align*}
	which commutes with the respective Frobenius morphisms.	Since $p$ is nilpotent in $R$, we have an identification $W(R)[1/p] \cong W(R/pR)[1/p]$ (cf. \cite[pg. 29]{BP2017}), so by Tannakian duality and Lemma \ref{lem-froblocal}, it is equivalent to give $g \in LG(R)$ such that $U \mu^\sigma(p) = g^{-1}u \mu^\sigma(p) f(g)$, i.e. such that $m_\mu(U) = c_b(g)$. 
	
	Now suppose $(\mathscr{D}_1,\iota_1)$ and $(\mathscr{D}_2,\iota_2)$ are banal, corresponding to $(U_1, g_1)$, $(U_2, g_2)$. Then by Proposition \ref{prop-local}, a morphism $\psi:(\mathscr{D}_1,\iota_1)\to (\mathscr{D}_2,\iota_2)$ is given by an $h \in L^+_\mu G(R)$ such that $\tau(h)^{-1}\cdot U_2\cdot \sigma(h) = U_1$. The condition that $\psi$ lifts $\iota_2^{-1}\circ \iota_1$ means that for every $(V,\rho)$ in $\textbf{Rep}_{\zz_p}(G)$, $(\psi_{R/pR})_\rho^\tau = (\iota_2)_\rho^{-1} \circ (\iota_1)_\rho$ as $W(R)[1/p]$-module automorphisms of $V \otimes_{\zz_p} W(R)[1/p]$. But this says $\rho(\tau(h)) = \rho(g_2)^{-1}\rho(g_1)$ for every $(V,\rho)$, which by Tannakian duality is equivalent to $\tau(h) = g_2^{-1}g_1$. Hence $(U_2,g_2)\cdot h = (U_1,g_1)$, which completes the proof.
\end{proof}

\section{Representability in some cases}
\subsection{The representability conjecture}

Suppose $(G,\{\mu\},[b])$ is a local integral Shimura datum, and let $(\mu,b)$ be a framing pair. In this section, suppose additionally that $G$ is reductive over $\zz_p$. Then the framing object $\mathscr{D}_0$ corresponds to a $(G,\mu)$-display $\mathcal{D}_0$ over $k$ in the sense of \cite{BP2017}. In \textit{loc. cit.}, a functor is associated to the data $(G,\mu, b)$ which classifies isomorphism classes of deformations of $\mathcal{D}_0$ up to $G$-quasi-isogeny. Let us call this functor $\text{RZ}^{\text{BP}}_{G,\mu, b}$, and denote by $\textbf{RZ}^{\text{BP}}_{G,\mu,b}$ the corresponding fpqc stack in groupoids.

\begin{prop} \label{cor-BPcompare}
	If $G$ is a reductive group scheme over $\zz_p$, then the stacks $\textbf{\textup{RZ}}^\textup{BP}_{G,\mu, b}$ and $\textbf{\textup{RZ}}_{G,\mu,b}$ are isomorphic.
\end{prop}
\begin{proof}
	Combine Theorem \ref{prop-rzquotient} with \cite[\textsection 4.2.3]{BP2017} and \cite[Remark 6.3.4]{Lau2018}.
\end{proof}

B\"ultel and Pappas make the following representability conjecture, see \cite[Conjecture 4.2.1]{BP2017}. 

\begin{conj}[B\"ultel-Pappas] \label{conj}
	Assume that $G$ is reductive and that $-1$ is not a slope of $\textup{Ad}^G(b)$. Then the functor $\textup{RZ}_{G,\mu,b}^\text{BP}$ is representable by a formal scheme which is formally smooth and formally locally of finite type over $W(k)$.
\end{conj}

See \cite[3.2.16]{BP2017} for details on the slope condition. For $G = \textup{GL}_n$, and for $b$ with no slopes equal to zero, Conjecture \ref{conj} follows essentially from the equivalence between Zink displays and formal $p$-divisible groups due to Zink and Lau and the representability of the classical Rapoport-Zink functor from \cite{RZ1996}. This is explained in the next section. In \cite{BP2017}, in the case where $G$ is reductive and the data $(G,\mu,b)$ is of Hodge type (see Definition \ref{def-hodge}), Conjecture \ref{conj} is proven for the restriction of the functor $\text{RZ}_{G,\mu,b}$ to Noetherian rings in $\textbf{Nilp}_{W(k)}$. Further, B\"ultel and Pappas show \cite[Remark 5.2.7]{BP2017} that the resulting formal scheme is isomorphic to the formal schemes constructed in \cite{HP2017} and \cite{Kim2018}. In particular, this shows that, when $R$ is Noetherian, $\text{RZ}_{G,\mu,b}(R)$ agrees with the points in the classical Rapoport-Zink space defined in \cite{RZ1996} in the unramified EL- and PEL-type cases. In Section \ref{sub-EL} we will take this one step further and prove that $\text{RZ}_{G,\mu,b}$ is naturally isomorphic to the classical Rapoport-Zink functor in the unramified EL-type case. This combined with the known representability of the classical functor will prove Conjecture \ref{conj} in that case.

\subsection{Representability for $\text{GL}_n$}\label{sub-gln}
In this section we summarize the proof of Conjecture \ref{conj} in the case $G = \text{GL}_n$ by comparing the moduli spaces of deformations of $p$-divisible groups up to quasi-isogeny from \cite{RZ1996} to an analogous moduli space defined using Zink displays (equivalently, 1-displays). The results of this section are well known and alluded to in \cite{BP2017}, but we found it prudent to work out some details for use in Section \ref{sub-EL}. Suppose $k$ is an algebraically closed field of characteristic $p$, and fix a $p$-divisible group $X_0$ over $k$. 
\begin{Def}\label{def-RZ}
	Define $\textup{RZ}(X_0)$ to be the set-valued functor on the category of $W(k)$-schemes in which $p$ is locally nilpotent which associates to any such $S$ the set of isomorphism classes of pairs $(X, \rho)$, where 
	\begin{itemize}
		\item $X$ is a $p$-divisible group over $S$, and
		\item $\rho: X \times_S \overline{S} \dashrightarrow X_0\times_{\text{Spec } k} \overline{S}$ is a quasi-isogeny.
	\end{itemize}
	Here $\overline{S}$ is the closed subscheme of $S$ defined by $p\mathcal{O}_S$, and we say two pairs $(X_1, \rho_1)$, $(X_2, \rho_2)$ are isomorphic if $\rho_2^{-1} \circ \rho_1$ lifts to an isomorphism $X_1 \to X_2$. 
\end{Def}
By \cite[Theorem 2.16]{RZ1996}, the functor $\textup{RZ}(X_0)$ is representable by a formal scheme which is formally smooth and locally formally of finite type over Spf $W(k)$. We can define an exactly analogous functor using displays.

\begin{Def}\label{def-RZM}
	Define $\textup{RZ}(\underline{M_0})$ to be the set-valued functor on \textbf{Nilp}$_{W(k)}$ which associates to a ring $R$ in \textbf{Nilp}$_{W(k)}$ the set of isomorphism classes of pairs $(\underline{M}, \gamma)$, where
	\begin{itemize}
		\item $\underline{M}$ is a 1-display over $\underline{W}(R)$, and
		\item $\gamma: \underline{M}_{\underline{W}(R/pR)} \dashrightarrow \underline{M_0}_{\underline{W}(R/pR)}$ is a quasi-isogeny.
	\end{itemize}
	We say two pairs $(\underline{M_1}, \gamma_1)$ and $(\underline{M_2},\gamma_2)$ are isomorphic if $\gamma_2^{-1} \circ \gamma_1$ lifts to an isomorphism $\underline{M_1} \to \underline{M_2}$.
\end{Def}

\begin{prop}\label{prop-gln}
	Let $\underline{M}_0$ be a nilpotent 1-display, and let $X_0 = BT_k(\underline{M_0})$. Then the functor $\textup{RZ}(\underline{M_0})$ and the restriction of $\textup{RZ}(X_0)$ to \textup{\textbf{Nilp}}$_{W(k)}$ are naturally isomorphic. In particular, $\textup{RZ}(\underline{M_0})$ is representable by a formal scheme which is formally smooth and locally formally of finite over $\textup{Spf }W(k)$.
\end{prop}
\begin{proof}
	After one checks that the nilpotence condition for displays is preserved by isogenies, and that the functor $BT_R$ of Theorem \ref{thm-BT} sends isogenies of nilpotent displays to isogenies of $p$-divisible groups, the theorem is an immediate corollary of Theorem \ref{thm-BT}.
\end{proof}

Let $G = \text{GL}_n$, let $1 \le d \le n$, and let $\mu_{d,n}$ be the cocharacter 
\begin{align*}
\mu_{d,n}: \mathbb{G}_m \to \text{GL}_n \ a \mapsto \text{diag}(1^{(d)}, a^{(n-d)}).
\end{align*}
Let $[b]$ be a $\sigma$-conjugacy class of elements in $\text{GL}_n(K)$, and choose a representative $b$ as in Section \ref{sub-localdata}. This determines a $(\text{GL}_n, \mu_{d,n})$-display $\mathscr{D}_0$. By Remark \ref{rmk-gln}, the stack of $(\text{GL}_n, \mu_{d,n})$-displays is isomorphic to the stack of displays of type $(0^{(d)}, 1^{(n-d)})$. Such a display has depth 0 and altitude 1, and is therefore a 1-display. Then the functor $\textup{RZ}_{\text{GL}_n,\mu_{d,n},b}$ is naturally isomorphic to the functor $\textup{RZ}(\underline{M_0})$, where $\underline{M_0}$ is the 1-display associated to $\mathscr{D}_0$. If the slopes of $b$ are all different from 0, then the resulting 1-display is nilpotent, so the functor is representable and isomorphic to $\textup{RZ}(BT_k(\underline{M_0}))$ by Proposition \ref{prop-gln}.	This proves the following corollary.
\begin{cor}\label{cor-gln}
	If the slopes of $b$ are all different from 0, the functor $\textup{RZ}_{\text{GL}_n,\mu_{d,n},b}$ is naturally isomorphic to $\textup{RZ}(BT_k(\underline{M_0}))$. In particular, it is representable by a formal scheme which is formally smooth and formally locally of finite type over \textup{Spf} $W(k)$.
\end{cor}

%\begin{proof}
%	Fix a $W(k)$-algebra $R$ in which $p$ is nilpotent, and let $(\underline{M}, \gamma) \in \textup{RZ}(\underline{M_0})(R)$. By Lemma \ref{lem-nilpotent} $\underline{M}_{R/pR}$ is a nilpotent 1-display over $R/pR$, so by \cite[Lemma 21]{Zink2002} $\underline{M}$ is nilpotent over $R$. Hence $BT_R(\underline{M})$ is a formal $p$-divisible group. If $(BT_R(\underline{M_1},BT_R(\gamma_1))$ and $(BT_R(\underline{M_2}, BT_R(\gamma_2))$ are isomorphic by some $\varepsilon$ which lifts $BT_R(\gamma_1)\circ BT_R(\gamma_2)^{-1}$, then full-faithfulness of $BT_R$ implies there is some $\delta: \underline{M_2} \xrightarrow{\sim} \underline{M_1}$ lifting $\gamma_1 \circ \gamma_2^{-1}$. Finally given an $(X,\rho) \in \textup{RZ}(X_0)(R)$, essential surjectivity of $BT_R$ implies there is some $\underline{M}$ with an isomorphism $\psi: BT_R(\underline{M}) \to X$. Choose $m$ such that $[p^m]\rho$ is an isogeny. Then there is some $\gamma: \underline{M_0}_{R/pR} \to \underline{M}_{R/pR}$ such that $BT_R(\gamma)$ is the reduction of $\psi^{-1}$ composed with $[p^m]\rho$. Then $(BT_R(\underline{M}), p^{-m} BT_R(\gamma))$ is isomorphic to $(X,\rho)$. 
%\end{proof}

\subsection{The determinant condition}
In the next two sections our goal is to give a generalization of Corollary \ref{cor-gln} when the local Shimura datum $(G,\{\mu\}, [b])$ is of EL-type. We begin by recalling the definition of RZ data in the EL case, cf. \cite[1.38]{RZ1996}, \cite[4.1]{RV2014}. As in \ref{sub-localdata}, let $k$ be an algebraic closure of $\mathbb{F}_p$, and let $W(k)$ be the ring of Witt vectors over $k$. Write $K = W(k)[1/p]$, and let $\bar{K}$ be an algebraic closure of $K$.

\begin{Def}
	An \textit{unramified integral EL-datum} is a tuple
	\begin{align*}
		\textbf{D} = (B, \mathcal{O}_B, \Lambda)
	\end{align*}
	such that
	\begin{itemize}
		\item $B$ is a semisimple $\qq_p$-algebra whose simple factors are matrix algebras over unramified extensions of $\qq_p$;
		\item $\mathcal{O}_B$ is a maximal order in $B$;
		%\item $C$ is the center of $B$;
		\item $\Lambda$ is a finite free $\zz_p$-module with an  $\mathcal{O}_B$-action.
	\end{itemize}
\end{Def}
To an unramified integral EL-datum $\textbf{D} = (B, \mathcal{O}_B, \Lambda)$ we associate a reductive algebraic group $G = G_{\textbf{D}}$ over $\zz_p$ whose points in a $\zz_p$-algebra $R$ are given by
\begin{align*}
	G(R) = \text{GL}_{\mathcal{O}_B}(\Lambda \otimes_{\zz_p} R).
\end{align*}

%\begin{Def}
%	An \textit{unramified integral PEL-datum} is a tuple
%	\begin{align*}
%		\textbf{D} = (B, \mathcal{O}_B, F, V, (,), \ast)
%	\end{align*}
%	such that $(B, \mathcal{O}_B, F, V)$ is an unramified integral EL-datum, and
%	\begin{itemize}
%		\item $(,)$ is a perfect alternating $\qq_p$-bilinear form on $V$ under which $\Lambda$ is self-dual;
%		\item $\ast: a \mapsto a^\ast$ is a $\qq_p$-linear involution of $B$ under which $\mathcal{O}_B$ is stable, such that
%		\begin{align*}
%			(av, w) = (v, a^\ast w), \ v,w \in V.
%		\end{align*}
%	\end{itemize}
%\end{Def}
%To such a datum $\textbf{D} = (B, \mathcal{O}_B, F, V, \Lambda, (,), \ast)$ we associate a reductive group $G = G_\textbf{D}$ over $\zz_p$ whose points in a $\zz_p$-algebra $R$ are given by
%\begin{align*}
%	G(R) = \{g \in \text{GL}_{\mathcal{O}_B}(\Lambda \otimes_{\zz_p} R) \mid (gv, gw) = c(g)(v,w), \ c(g) \in R\}.
%\end{align*}

\begin{Def}\label{def-unramifiedRZEL}
	An \textit{unramified integral RZ-datum of EL-type} is a tuple
	\begin{align*}
		(\textbf{D}, \{\mu\}, [b])
	\end{align*} 
	such that $\textbf{D}$ is an EL-datum with associated group $G$ and $(\{\mu\}, [b])$ is a pair as in \textsection \ref{sub-localdata} such that $\{\mu\}$ satisfies Assumption \ref{ass_intmu} and $(G,\{\mu\},[b])$ is a local integral Shimura datum. We require the following condition:
	\begin{itemize}
		\item For any $\mu$ in $\{\mu\}$, the only weights occurring in the weight decomposition for $\Lambda \otimes_{\zz_p} W(k)$ are $0$ and $1$, i.e.
		\begin{align*}
			\Lambda \otimes_{\zz_p} W(k) = \Lambda^0 \oplus \Lambda^1.
		\end{align*}
%		\item In the PEL case, for any $\mu \in \{\mu\}$, the composition
%		\begin{align*}
%			{\mathbb{G}_m}_{\bar{K}} \xrightarrow{\mu} G_{\bar{K}} \xrightarrow{c} {\mathbb{G}_m}_{\bar{K}}
%		\end{align*}
%		is the identity, where $c$ is the similitude morphism associated to $G$. 
	\end{itemize}
\end{Def}

For the remainder of this section let us fix an unramified integral RZ-datum of EL-type $(\textbf{D},\{\mu\}, [b])$. Before we can formulate the corresponding moduli problem of $p$-divisible groups, we must first recall the determinant condition following \cite[3.23]{RZ1996}. Let $\mathbb{V}$ be the scheme over $\zz_p$ whose points in a $\zz_p$-algebra $R$ are given by
\begin{align*}
	\mathbb{V}(R) = \mathcal{O}_B\otimes_{\zz_p} R.
\end{align*}
For any $W(k)$-algebra $R$ define
\begin{align*}
	\delta_{\textbf{D}}: \mathbb{V}(R) \to \mathbb{A}^1_{W(k)}(R), \ a \mapsto \det\left(a \mid \Lambda^0 \otimes_{W(k)} R\right).
\end{align*}
This determines a morphism of $W(k)$-schemes $\mathbb{V}_{W(k)} \to \mathbb{A}^1_{W(k)}$. Now, let $R$ be a $W(k)$-algebra and $L$ be a finite projective $R$-module endowed with an $\mathcal{O}_B$-action. Then we can define similarly, for any $R$-algebra $R'$,
\begin{align*}
	\delta_L: \mathbb{V}_R(R') \to \mathbb{A}^1_{R}(R'), \ a \mapsto \det\left(a \mid L\otimes_R R'\right),
\end{align*}
which determines a morphism of $R$-schemes $\mathbb{V}_R \to \mathbb{A}^1_R$. 
\begin{Def}
	We say that $L$ \textit{satisfies the determinant condition with respect to \textbf{D}} if the morphisms of $R$-schemes $\delta_{\textbf{D}} \otimes \id_{\text{Spec}(R)}$ and $\delta_L$ are equal.
\end{Def}
Let $X$ be a $p$-divisible group over a $p$-nilpotent $W(k)$-algebra $R$ which is equipped with an action of $\mathcal{O}_B$, i.e. a homomorphism $\mathcal{O}_B \to \text{End}(X).$ Then the Lie algebra $\text{Lie}(X)$ is endowed with the structure of an $\mathcal{O}_B \otimes_{\zz_p} R$-module, so one can ask whether Lie$(X)$ satisfies the determinant condition with respect to \textbf{D}.

Let $R$ be a $p$-nilpotent $W(k)$-algebra, and suppose $X$ is a formal $p$-divisible group with an action by $\mathcal{O}_B$. Let $\underline{M} = (M,F)$ be the nilpotent 1-display with $\mathcal{O}_B$-action corresponding to $X$, so $X = BT_R(\underline{M})$. We would like to reinterpret the determinant condition as a condition on the projective graded $W(R)^\oplus$-module $M$. 

Suppose the height of $X$ is equal to rk$_{\zz_p} \Lambda$, so by Theorem \ref{thm-BT}, 
\begin{align*}
\text{rk}_{W(R)} M^\tau = \text{rk}_{\zz_p} \Lambda.
\end{align*}
By the recollections from Section \ref{sub-1displays}, we have an identification 
\begin{align*}
M^\tau \otimes_{W(R)} R \cong \mathbb{D}(X)_R
\end{align*} 
which identifies the Hodge filtrations, i.e.
\begin{align*}
(M^\tau \otimes_{W(R)} R \supset E_1 \supset 0) \xrightarrow{\sim} (\mathbb{D}(X)_R \supset \text{Lie}(X^\vee)^\ast \supset 0)
\end{align*}
is an isomorphism of filtered $\mathcal{O}_B \otimes_{\zz_p} R$-modules. Here $E_1$ is the image of $M_1$ under the composition 
\begin{align*}
M_1 \xrightarrow{\theta_1} M^\tau \to M^\tau \otimes_{W(R)} R.
\end{align*} 
In particular, we have an identification 
\begin{align*}
\text{Lie}(X) \cong M^\tau / \theta_1(M_1).
\end{align*}
Viewing $\mathcal{O}_B$ as a graded $\zz_p$-algebra concentrated in degree zero, we can view $\mathcal{O}_B \otimes_{\zz_p} W(R)^\oplus$ as a graded ring. Then $\Lambda \otimes_{\zz_p} W(R)^\oplus$ inherits the structure of a graded $\mathcal{O}_B \otimes_{\zz_p} W(R)^\oplus$-module.

\begin{lemma}\label{lem-det}
	The following are equivalent:
	\begin{enumerate}[\textup{(}i\textup{)}]
		\item For some faithfully flat extension $R \to R'$ there is an isomorphism of graded $\mathcal{O}_B \otimes_{\zz_p} W(R')^\oplus$-modules 
		\begin{align*}
		M_{W(R')^\oplus} \xrightarrow{\sim} \Lambda \otimes_{\zz_p} W(R')^\oplus
		\end{align*}
		\item For some faithfully flat extension $R \to R'$ there is an isomorphism of filtered $\mathcal{O}_B \otimes_{\zz_p} R'$-modules
		\begin{align}\label{eq-filt}
		(M^\tau \otimes_{W(R)} R' \supset E_1 \otimes_R R' \supset 0) \xrightarrow{\sim} (\Lambda \otimes_{\zz_p} R' \supset \Lambda^1 \otimes_{W(k)} R' \supset 0)
		\end{align}
		\item \textup{Lie}$(X)$ satisfies the determinant condition with respect to \textbf{D}.
	\end{enumerate}
\end{lemma}
\begin{proof}
	We start by proving (i) and (ii) are equivalent. Suppose (i) holds, so for some faithfully flat $R \to R'$ there is a $\varphi: M_{W(R')^\oplus} \xrightarrow{\sim} \Lambda \otimes_{\zz_p} W(R')^\oplus$ which is compatible with the $\mathcal{O}_B$-action and the grading. Then in particular, $M_1'$, the first graded piece of $M_{W(R')^\oplus}$, is carried by $\varphi$ into the first graded piece of $\Lambda \otimes_{\zz_p} W(R')^\oplus$. That is,
	\begin{align*}
	\varphi(M_1') \cong (\Lambda^0 \otimes_{W(k)} I_{R'}) \oplus (\Lambda^1 \otimes_{W(k)} W(R')).
	\end{align*}
	By reducing modulo $I_{R'}(M_{W(R')^\oplus})^\tau$ we see that (ii) holds.
	
	Now let us prove (ii) implies (i). Let
	\begin{align*}
	\bar{\varphi}: M^\tau \otimes_{W(R)} R' \xrightarrow{\sim} \Lambda \otimes_{\zz_p} R'
	\end{align*}
	be an isomorphism preserving the filtration. First we lift $\bar{\varphi}$ to an $\mathcal{O}_B\otimes_{\zz_p} W(R)$-module isomorphism
	\begin{align*}
		\varphi: (M_{W(R')^\oplus})^\tau \xrightarrow{\sim} \Lambda \otimes_{\zz_p} W(R').
	\end{align*}
	%It's enough to show $\bar{\varphi}$ lifts to an isomorphism $M_{W(R')^\oplus} \xrightarrow{\sim} \Lambda \otimes_{\zz_p} W(R')^\oplus$, since any such lift will necessarily preserve the grading.
	%	It's enough to show $\bar{\varphi}$ lifts to an isomorphism $\varphi: M_{W(R')^\oplus} \xrightarrow{\sim} \Lambda \otimes_{\zz_p} W(R')^\oplus$. Indeed, the following diagram commutes,
	%	\begin{center}
	%		\begin{tikzcd}
	%			M_{W(R')^\oplus} 
	%				\arrow[r, "\varphi"] \arrow[d]
	%			& \Lambda \otimes_{\zz_p} W(R')^\oplus 
	%				\arrow[d]
	%			\\ M^\tau \otimes_{W(R)} R' 
	%				\arrow[r, "\bar{\varphi}"]
	%			& \Lambda \otimes_{\zz_p} R'
	%		\end{tikzcd}
	%	\end{center}
	%	so if $\bar{\varphi}$ preserves the filtration below, $\varphi$ must preserve the grading above.
	%	So if $m \in M_1'$, where $M_1'$ is the first graded piece of $M_{W(R')^\oplus}$, then the reduction of $\varphi(m)$ will land in $\Lambda^1 \otimes_{W(k)} R'$, hence $\varphi(m)$ must land in the first graded piece of $\Lambda \otimes_{\zz_p} W(R')^\oplus$. 
	\noindent By restricting to simple factors and applying Morita equivalence we may assume $\mathcal{O}_B = \mathcal{O}_L$ is the ring of integers in an unramified extension $L$ of degree $n$ over $\qq_p$. In that case we have an isomorphism
	\begin{align*}
	\mathcal{O}_L \otimes_{\zz_p} W(k) \cong \prod_{j \in \zz/n\zz} W(k),
	\end{align*}
	which gives decompositions 
	\begin{align*}
	(M_{W(R')^\oplus})^\tau = \bigoplus_{j \in \zz/n\zz} M'(j), \ \text{ and }\ \Lambda \otimes_{\zz_p} W(R') = \bigoplus_{j \in \zz/n\zz} \Lambda'(j),
	\end{align*}
	where $M'(j) = \{ m \in (M_{W(R')^\oplus})^\tau \mid a\cdot m = \sigma^j(a)m, \ a \in \mathcal{O}_L\}$, and similarly for $\Lambda'(j)$. Since $\bar{\varphi}$ is an isomorphism of $\mathcal{O}_L \otimes_{\zz_p} R'$-modules, it induces isomorphisms
	\begin{align*}
	M'(j) \otimes_{W(R')^\oplus} R' \xrightarrow{\sim} \Lambda'(j) \otimes_{W(k)} R'
	\end{align*}
	for every $j$. Each $M'(j)$ is projective as a $W(R')$-module, so these isomorphisms can be lifted to $W(R')$-module homomorphisms
	\begin{align*}
	M'(j) \to \Lambda'(j),
	\end{align*}
	which are surjective by Nakayama's lemma. Summing these together we obtain a map
	\begin{align*}
	(M_{W(R')^\oplus})^\tau \to \Lambda \otimes_{\zz_p} W(R'),
	\end{align*}
	which is an $\mathcal{O}_L \otimes_{\zz_p} W(R')$-module homomorphism since it preserves the decompositions above. At the same time, it is a surjective homomorphism between projective $W(R')$-modules of the same rank, so it must be an isomorphism.
	
	If $M'$ is a finite projective  graded $W(R')^\oplus$-module of non-negative depth and altitude one, then the assignment $M' \mapsto ((M')^\tau, \theta_1'(M_1'))$ determines a functor $F$ to the category $\textbf{Pairs}(R)$ of pairs $(P,Q)$, where $P$ is a finite projective $W(R)$-module and $Q \subset P$ is a submodule. By the proof of Lemma \ref{lem-zink} this functor is fully faithful. In the case at hand, because $\bar{\varphi}$ preserves the filtration (\ref{eq-filt}), it follows that $\varphi$ sends $\theta_1'(M_1')$ into $(\Lambda^0 \otimes_{W(k)} I_{R'}) \oplus (\Lambda^1 \otimes_{W(k)} W(R'))$, so $\varphi$ determines an isomorphism
	\begin{align*}
		((M_{W(R')^\oplus})^\tau, \theta_1'(M_1')) \xrightarrow{\sim} (\Lambda \otimes_{\zz_p} W(R'), (\Lambda^0 \otimes_{W(k)} I_{R'}) \oplus (\Lambda^1 \otimes_{W(k)} W(R'))) 
	\end{align*}
	in $\textbf{Pairs}(R)$. Since the left-hand side is $F(M_{W(R')^\oplus})$ and the right-hand side is $F(\Lambda \otimes_{\zz_p} W(R')^\oplus)$, full faithfulness of $F$ implies the existence of an isomorphism of graded $W(R')^\oplus$-modules 
	\begin{align*}
		\tilde{\varphi}: M_{W(R')^\oplus} \xrightarrow{\sim} \Lambda \otimes_{\zz_p} W(R')^\oplus
	\end{align*}
	lifting $\varphi$. Since the $\mathcal{O}_B$-action on $(M_{W(R')^\oplus})^\tau$ is induced by the given action on $M_{W(R')^\oplus}$, we can  once again apply full faithfulness of $F$ to see $\tilde{\varphi}$ is compatible with the respective $\mathcal{O}_B$-actions.
	
	Now let us prove the equivalence of (ii) and (iii). Suppose (ii) holds. The determinant condition can be checked fpqc-locally because morphisms of schemes can be glued locally for the fpqc topology. But by (ii) and the above identifications there is an isomorphism of $\mathcal{O}_B \otimes_{\zz_p} R'$-modules
	\begin{align*}
	\text{Lie}(X_{R'}) \xrightarrow{\sim} \Lambda^0 \otimes_{W(k)} R',
	\end{align*}
	where $R \to R'$ is a faithfully flat extension. This implies the determinant condition holds over $R'$, and therefore over $R$ as well.
	
	Conversely, suppose (iii) holds, i.e. Lie$(X)$ satisfies the determinant condition with respect to \textbf{D}. Again by restricting to simple factors and applying Morita equivalence we can assume $\mathcal{O}_B = \mathcal{O}_L$ is the ring of integers in an unramified extension $L$ of degree $n$ over $\qq_p$. As in the proof of (ii) implies (i), we have decompositions
	\begin{align*}
	\Lambda^0 = \bigoplus_{j \in \zz/n\zz} \Lambda^0(j), \ \text{ and } \ \text{Lie}(X) = \bigoplus_{j \in \zz/n\zz} L(j).
	\end{align*}
	In this case the determinant condition is equivalent to
	\begin{align*}
	\text{rk}_{W(k)} \Lambda^0(j) = \text{rk}_R L(j)
	\end{align*}
	for every $j$, cf. \cite[3.23(b)]{RZ1996}. Hence $L(j)$ and $\Lambda^0(j)\otimes_{W(k)} R$ are projective $R$-modules of the same rank, so after some localization we can find an isomorphism $\alpha_j: L(j) \xrightarrow{\sim} \Lambda^0(j) \otimes_{W(k)} R$ for every $j$. Similarly we can decompose
	\begin{align*}
	\Lambda^1 = \bigoplus_{j \in \zz/n\zz} \Lambda^1(j) \ \text{ and } \ \text{Lie}(X^\vee)^\ast = \bigoplus_{j \in \zz/n\zz} L'(j),
	\end{align*}
	and, perhaps after another localization, we can find isomorphisms $\beta_j: L'(j) \xrightarrow{\sim} \Lambda^1(j) \otimes_{W(k)} R$ for every $j$. Now, write
	\begin{align*}
	\mathbb{D}(X)_R = \bigoplus_{j \in \zz/n\zz} D(j) \ \text{ and } \ \Lambda = \bigoplus_{j \in \zz/n\zz} \Lambda(j).
	\end{align*}
	Since $L(j)$ and $\Lambda^0(j)$ are projective, the short exact sequences
	\begin{align*}
	0 \to L'(j) \to D(j) \to L(j) \to 0 
	\end{align*}
	and 
	\begin{align*}
	0 \to \Lambda^1(j)\otimes_{W(k)} R \to \Lambda(j)\otimes_{W(k)} R \to \Lambda^0(j) \otimes_{W(k)} R \to 0
	\end{align*}
	split over $R$. Hence we can piece together the isomorphisms $\alpha_j$ and $\beta_j$ to obtain $\phi_j: D(j) \xrightarrow{\sim} \Lambda(j) \otimes_{W(k)} R$, preserving the filtration by direct summands. Combining these for all $j$ gives an isomorphism 
	\begin{align*}
	\varphi: \mathbb{D}(X)_R \xrightarrow{\sim} \Lambda \otimes_{\zz_p} R.
	\end{align*}
	By construction, $\varphi$ is an isomorphism of $\mathcal{O}_L \otimes_{\zz_p} R$-modules which identifies the Hodge filtrations of the two sides. By the remarks before the statement of the lemma we have a natural identification
	\begin{align*}
	(M^\tau \otimes_{W(R)} R \supset E_1 \supset 0) \xrightarrow{\sim} (\mathbb{D}(X)_R \supset \text{Lie}(X^\vee)^\ast \supset 0).
	\end{align*}
	Combining this with $\varphi$ gives the desired isomorphism.
\end{proof}

\subsection{Representability in the EL-type case}\label{sub-EL}
In this section we prove that $\textup{RZ}_{G,\mu,b}$ is representable by a formal scheme when $(\mu,b)$ is a framing pair for a local integral Shimura datum $(G,\{\mu\},[b])$ associated to an unramified integral RZ-datum of EL-type. Throughout this section we continue to assume $k$ is an algebraic closure of $\mathbb{F}_p$, that $\textbf{D} = (B,\mathcal{O}_B,\Lambda)$ is an unramified integral EL-datum, and that $(\textbf{D},\{\mu\},[b])$ is an unramified integral RZ-datum of EL-type. Let $X_0$ be a $p$-divisible group over $k$ equipped with an action by $\mathcal{O}_B$. Let us begin by recalling the definition of the unramified EL-type Rapoport-Zink space associated to $X_0$ and $\textbf{D}$. 

\begin{Def}\label{def-RZEL}
	Define $\textup{RZ}_{\textbf{D}}(X_0)$ to be the set-valued functor on the category of $W(k)$-schemes in which $p$ is locally nilpotent which associates to any such $S$ the set of isomorphism classes of pairs $(X,\rho)$, where 
	\begin{itemize}
		\item $X$ is a $p$-divisible group over $S$ equipped with an action of $\mathcal{O}_B$, such that Lie$(X)$ satisfies the determinant condition with respect to \textbf{D}, and
		\item $\rho: X \times_S \overline{S} \dashrightarrow X_0\times_{\text{Spec}(k)} \overline{S}$ is a quasi-isogeny which commutes with the action of $\mathcal{O}_B$.
	\end{itemize}
	Here again $\overline{S}$ denotes the closed subscheme of $S$ defined by $p\mathcal{O}_S$, and two pairs $(X_1, \rho_1)$ and $(X_2, \rho_2)$ are isomorphic if $\rho_2^{-1} \circ \rho_1$ lifts to an isomorphism $X_1 \to X_2$ respecting the $\mathcal{O}_B$-actions.
\end{Def}

By \cite[Theorem 3.25]{RZ1996}, the inclusion $\textup{RZ}_{\textbf{D}}(X_0) \hookrightarrow \textup{RZ}(X_0)$ is a closed immersion, so the functor $\textup{RZ}_{\textbf{D}}(X_0)$ is representable by a formal scheme which is formally smooth and formally locally of finite type over Spf $W(k)$. If $\underline{M_0}$ is a 1-display over $k$ equipped with an $\mathcal{O}_B$-action, we obtain an analogous subfunctor of $\textup{RZ}(\underline{M_0})$ by replacing all $p$-divisible groups that arise in Definition \ref{def-RZEL} with displays.

\begin{Def}
	Let $\textup{RZ}_{\textbf{D}}(\underline{M_0})$ be the set-valued functor on \textbf{Nilp}$_{W(k)}$ which associates to a ring $R$ in \textbf{Nilp}$_{W(k)}$ the set of isomorphism classes of pairs $(\underline{M},\gamma)$, where
	\begin{itemize}
		\item $\underline{M} = (M,F)$ is a 1-display over $\underline{W}(R)$ equipped with an action by $\mathcal{O}_B$ such that the underlying graded $\mathcal{O}_B \otimes_{\zz_p} W(R)^\oplus$-module $M$ is fpqc-locally isomorphic to $\Lambda \otimes_{\zz_p} W(R)^\oplus$, and
		\item $\gamma: \underline{M}_{\underline{W}(R/pR)} \dashrightarrow \underline{M_0}_{\underline{W}(R/pR)}$ is a quasi-isogeny of displays which commutes with the $\mathcal{O}_B$-actions.
	\end{itemize}
	As in Definition \ref{def-RZM}, two pairs $(\underline{M_1}, \gamma_1)$ and $(\underline{M_2}, \gamma_2)$ are isomorphic if $\gamma_2^{-1} \circ \gamma_1$ lifts to an isomorphism $\underline{M_1} \to \underline{M_2}$.
\end{Def}

\begin{lemma}\label{lem-intermediate}
	Let $\underline{M}_0$ be a nilpotent 1-display endowed with an $\mathcal{O}_B$-action, and let $X_0 = BT_k(\underline{M_0})$. Then there is a natural isomorphism
	\begin{align*}
		\textup{RZ}_{\textbf{D}}(\underline{M_0}) \xrightarrow{\sim} \textup{RZ}_{\textbf{D}}(X_0)
	\end{align*} 
	of functors on $\textup{\textbf{Nilp}}_{W(k)}$.
\end{lemma}
\begin{proof}
	The natural isomorphism $\textup{RZ}(\underline{M_0}) \xrightarrow{\sim} \textup{RZ}(X_0)$ from Proposition \ref{prop-gln} restricts to a natural isomorphism $\textup{RZ}_{\textbf{D}}(\underline{M_0}) \xrightarrow{\sim} \textup{RZ}_{\textbf{D}}(X_0)$ by Lemma \ref{lem-det}. 
\end{proof}

Now, suppose $G$ is the group associated to $(\textbf{D},\{\mu\},[b])$, and denote by $\eta$ the natural closed embedding 
\begin{align*}
	\eta: G \hookrightarrow GL(\Lambda).
\end{align*} 
Let $(\mu,b)$ be a framing pair for $(G,\{\mu\},[b])$, and denote by $\mathscr{D}_0$ the associated framing object, cf. Definition \ref{def-framingobject}. Evaluating $\mathscr{D}_0$ on $(\Lambda,\eta)$, we obtain a display over $\underline{W}(k)$, which we will denote by $\underline{M_0}$. The condition in Definition \ref{def-unramifiedRZEL} implies $\mathscr{C}_{\mu,k}(\Lambda,\eta)$ is of type $(0^{(d)}, 1^{(n-d)})$, where $d = \text{rk}_{W(k)} \Lambda^0$ and $n = \text{rk}_{\zz_p} \Lambda$. Hence $\underline{M_0}$ is also of type $(0^{(d)},1^{(n-d)})$, and therefore $\underline{M_0}$ is a $1$-display. 

Now, notice that $\textup{RZ}_{\textbf{D}}(\underline{M_0})$ is the functor of isomorphism classes associated to a category fibered in groupoids $\textbf{RZ}_{\textbf{D}}(\underline{M_0})$. By Lemma \ref{lem-displaydescent} (or \cite[Theorem 37]{Zink2002}), $\textbf{RZ}_{\textbf{D}}(\underline{M_0})$ is an fpqc stack. Evaluation on the embedding $\eta: G \hookrightarrow GL(\Lambda)$ induces a natural morphism of stacks 
\begin{align*}
	\textbf{RZ}_{G,\mu,b} \to \textbf{RZ}_{\textup{D}}(\underline{M_0}).
\end{align*}
Let us explain. If $R$ is a $p$-nilpotent $W(k)$-algebra and $(\mathscr{D},\iota)$ is an object of $\textbf{RZ}_{G,\mu,b}(R)$, then by evaluating $\mathscr{D}$ on $(\Lambda, \eta)$ we obtain a display $\underline{M} = (M,F)$ over $\underline{W}(R)$. Since there is an isomorphism $\lambda:\upsilon_{R'}\mathscr{D}_{R'}\xrightarrow{\sim}  \mathscr{C}_{\mu,R'}$ after some faithfully flat $R \to R'$, we have an isomorphism of graded $W(R')^\oplus$-modules $\lambda_\eta: M_{W(R')^\oplus} \cong \Lambda \otimes_{\zz_p} W(R')^\oplus$. By Lemma \ref{lem-localtype}, then, $\underline{M}$ is a 1-display. Because the endomorphism on $\Lambda$ induced by any $a \in \mathcal{O}_B$ is $G$-equivariant, by functoriality we obtain an action of $\mathcal{O}_B$ on $\underline{M}$. Furthermore, by functoriality of $\lambda$, for each $a \in \mathcal{O}_B$, we have $\mathscr{C}_{\mu,R'}(a)\circ \lambda_\eta = \lambda_\eta \circ (\upsilon_{R'}\circ \mathscr{D}_{R'})(a)$, i.e. the isomorphism $M_{W(R')^\oplus} \cong \Lambda \otimes_{\zz_p} W(R')^\oplus$ is compatible with the $\mathcal{O}_B$-action. Now the $G$-quasi-isogeny $\iota$ induces a quasi-isogeny
\begin{align*}
\gamma := \iota_\eta: \underline{M}_{\underline{W}(R/pR)} \dashrightarrow \underline{M_0}_{\underline{W}(R/pR)}
\end{align*}
by evaluation on $(\Lambda, \eta)$. Again, because the action by $\mathcal{O}_B$ is $G$-equivariant, we see that $\gamma$ is compatible with the $\mathcal{O}_B$-actions. Hence $(\underline{M},\gamma)$ is an element of $\textbf{RZ}_{\textbf{D}}(\underline{M_0})(R)$. This construction is compatible with base change, so we have defined a morphism of stacks
\begin{align}\label{eq_RZmorphism}
\textbf{RZ}_{G,\mu,b} \to \textbf{RZ}_{\textbf{D}}(\underline{M_0}).
\end{align}

We will show that (\ref{eq_RZmorphism}) is an isomorphism, but first we prove a useful lemma. If $(\underline{M},\gamma)$ is any object in $\textbf{RZ}_{\textbf{D}}(\underline{M_0})(R)$, define
\begin{align}\label{eq-QM}
Q_{M} := \underline{\text{Isom}}^0(\Lambda\otimes_{\zz_p} W(R)^\oplus,M)
\end{align}
as the functor on $R$-algebras taking $R'$ to the group of graded $\mathcal{O}_B \otimes_{\zz_p} W(R')^\oplus$-isomorphisms between $M_{W(R')^\oplus}$ and $\Lambda \otimes_{\zz_p} W(R')^\oplus$. Then $Q_{M}$ is an $L_{\mu}^+G$-torsor over $R$ because $M$ is locally isomorphic to $\Lambda \otimes_{\zz_p} W(R)^\oplus$, and the group of graded $\mathcal{O}_B\otimes_{\zz_p} W(R')^\oplus$ automorphisms of $\Lambda \otimes_{\zz_p} W(R')^\oplus$ is isomorphic to $L^+_\mu G(R')$. By mimicking Construction \ref{cons_GDisp}, we obtain a morphism
\begin{align}\label{eq-aF}
\alpha_F: Q_M \to L^+G
\end{align}
such that the pair $(Q_{M}, \alpha_F)$ determines a $G$-display of type $\mu$ over $R$.

\begin{lemma}\label{lem-help}
	Let $R$ be a $p$-nilpotent $W(k)$-algebra, let $(\mathscr{D},\iota) \in RZ_{G,\mu,b}(R)$, and let $(\underline{M},\gamma)$ be the image of $(\mathscr{D},\iota)$ under (\ref{eq_RZmorphism}). If $(Q_{\mathscr{D}},\alpha_{\mathscr{D}})$ is the $G$-display of type $\mu$ associated to $\mathscr{D}$ by Construction \ref{cons_GDisp}, then evaluation on $(\Lambda, \eta)$ induces a natural isomorphism of $G$-displays of type $\mu$
		\begin{align*}
			(Q_\mathscr{D},\alpha_{\mathscr{D}}) \xrightarrow{\sim} (Q_M, \alpha_F).
		\end{align*}
\end{lemma}
\begin{proof}
	By definition, $Q_{\mathscr{D}} = \underline{\text{Isom}}^\otimes(\mathscr{C}_{\mu,R}, \upsilon_R \circ \mathscr{D})$, so evaluation on $(\Lambda, \eta)$ defines an isomorphism of $L^+_\mu G$-torsors $\delta:Q_\mathscr{D} \to Q_M$. We need to show that $\delta$ is an isomorphism of $G$-displays of type $\mu$, i.e. that $\alpha_F \circ \delta = \alpha_\mathscr{D}$. Let $\lambda \in Q_{\mathscr{D}}(R')$ for some $R$-algebra $R'$. Then $\eta((\alpha_F \circ \delta)(\lambda)) = \eta(\alpha_F(\lambda_{\eta})) \in L^+\text{GL}(\Lambda)(R')$ is the composition $(\lambda_\eta^\tau)^{-1} \circ F^\sharp \circ \lambda_\eta^\sigma$. On the other hand, by definition of $\alpha_{\mathscr{D}}$, this is exactly $\alpha_{\mathscr{D}}(\lambda)_\eta$. But
	\begin{align*}
	\alpha_{\mathscr{D}}(\lambda)_\eta = \eta(\alpha_{\mathscr{D}}(\lambda)),
	\end{align*}
	so we have $\eta(\alpha_{\mathscr{D}}(\lambda)) = \eta(\alpha_F(\lambda_\eta))$. Hence $\alpha_{\mathscr{D}}(\lambda) =  \alpha_F(\delta(\lambda))$.
\end{proof}
The following is the main theorem of this section.

\begin{thm}\label{thm-ELcompare}
	The morphism (\ref{eq_RZmorphism}) is an isomorphism of fpqc stacks.
\end{thm}
\begin{proof}
To see that it is faithful, let $R$ be a $p$-nilpotent $W(k)$-algebra, and suppose $\psi_1$ and $\psi_2$ are two morphisms $(\mathscr{D},\iota) \to (\mathscr{D}_2, \iota_2)$ in $\textbf{RZ}_{G,\mu,b}$ which agree after applying (\ref{eq_RZmorphism}). By descent, it is enough to check that $\psi_1$ and $\psi_2$ are equal fpqc-locally. But after some faithfully flat $R \to R'$, $\psi_1$ and $\psi_2$ correspond to elements $h_1, h_2 \in \text{Aut}^\otimes(\mathscr{C}_{\mu,R'}) = L^+_\mu G(R')$, and to say that $\psi_1$ and $\psi_2$ agree after applying (\ref{eq_RZmorphism}) means $\eta(h_1) = \eta(h_2)$. But $\eta$ is a faithful representation, so this implies $h_1 = h_2$, hence $\psi_1 = \psi_2$. 
%we note that any two morphisms of exact tensor functors which agree on a faithful representation must agree in general, because a faithful representation is a tensor generator for Rep$_{\zz_p}(G)$. 

To show the morphism is full, suppose $(\mathscr{D}_1, \iota_1)$, $(\mathscr{D}_2,\iota_2)$ are objects in $\textbf{RZ}_{G,\mu,b}(R)$ corresponding to objects $(\underline{M_1}, \gamma_1)$, $(\underline{M_2},\gamma_2)$ in $\textbf{RZ}_{\textbf{D}}(\underline{M_0})(R)$, and suppose $\varphi: (\underline{M_1},\gamma_1) \xrightarrow{\sim} (\underline{M_2},\gamma_2)$ is a morphism in $\textbf{RZ}_{\textbf{D}}(\underline{M_0})(R)$. By Lemma \ref{lem-help}, we have isomorphisms of $G$-displays of type $\mu$
\begin{align*}
	\delta_1: (Q_{\mathscr{D}_1},\alpha_{\mathscr{D}_1}) \xrightarrow{\sim} (Q_{M_1},\alpha_{F_1}) \ \text{ and } \ \delta_2: (Q_{\mathscr{D}_2},\alpha_{\mathscr{D_2}}) \xrightarrow{\sim} (Q_{M_2},\alpha_{F_2}),
\end{align*}
where the $Q_{M_i}$ and $\alpha_{F_i}$ are defined as in (\ref{eq-QM}) and (\ref{eq-aF}), respectively. The map of underlying $W(R)^\oplus$-modules $\varphi: M_1 \to M_2$ induces a morphism $\psi: Q_{M_1}\to Q_{M_2}$. As in the proof of Lemma \ref{lem-help}, one checks that $\alpha_{F_2}\circ\psi=\alpha_{F_1}$, so the composition 
\begin{align*}
	(Q_{\mathscr{D}_1},\alpha_{\mathscr{D}_1}) \xrightarrow{\delta_1} (Q_{M_1},\alpha_{F_1}) \xrightarrow{\psi} (Q_{M_2},\alpha_{F_2}) \xrightarrow{(\delta_2)^{-1}} (Q_{\mathscr{D}_2}, \alpha_{\mathscr{D}_2})
\end{align*}
is a morphism of $G$-displays of type $\mu$, which we call $\zeta$. By Theorem \ref{thm-equiv}, $\zeta$ is induced by a unique morphism of Tannakian $(G,\mu)$-displays $\xi: \mathscr{D}_1 \to \mathscr{D}_2$. We claim $\xi$ induces $\varphi$ via (\ref{eq_RZmorphism}), i.e. that $\xi_{\eta} = \varphi$. Once we know this, it will follow that $\xi$ lifts $(\iota_2)^{-1}\circ \iota_1$. Indeed, if $\xi$ induces $\varphi$, then $\xi_\eta$ lifts $(\iota_2)_{\eta}^{-1} \circ (\iota_1)_\eta$, and then we can argue as in the proof of faithfulness.

By descent it is enough to check $\xi_{\eta}$ and $\varphi$ agree fpqc-locally. Suppose $R \to R'$ is a faithfully flat extension, and let $\lambda \in \text{\underline{Isom}}^\otimes(\mathscr{C}_{\mu,R'}, \upsilon_{R'} \circ (\mathscr{D}_1)_{R'})$. Then because $\xi$ induces $\zeta$, we have $\zeta(\lambda) = \xi_{R'} \circ \lambda$. On the other hand, by definition of $\zeta$, $\zeta(\lambda)$ is the unique morphism $\mathscr{C}_{\mu,R'} \xrightarrow{\sim} \upsilon_{R'} \circ (\mathscr{D}_2)_{R'}$ such that $\zeta(\lambda)_{\eta} = \psi(\delta(\lambda)) = \varphi_{R'} \circ\lambda_{\eta}$. Hence
\begin{align*}
	\varphi_{R'} \circ \lambda_{\eta} = \zeta(\lambda)_\eta = (\xi_{R'})_\eta \circ \lambda_\eta.
\end{align*}
But $\lambda_\eta$ is an isomorphism, so $\varphi_{R'} = (\xi_{R'})_\eta$, i.e. $\varphi = \xi_\eta$. This shows (\ref{eq_RZmorphism}) is full.

Last we prove (\ref{eq_RZmorphism}) is essentially surjective. As usual, by \cite[\href{https://stacks.math.columbia.edu/tag/046N}{Lemma 046N}]{stacks-project}, it is enough to show any $(\underline{M},\gamma)$ over $R$ is in the essential image of (\ref{eq_RZmorphism}) fpqc-locally. But after some faithfully flat $R \to R'$ we have an isomorphism $\lambda: \Lambda \otimes_{\zz_p} W(R')^\oplus \xrightarrow{\sim} M_{W(R')^\oplus}$. Then the composition $(\lambda^\tau)^{-1} \circ F_{R'}^\sharp \circ \lambda^\sigma$ is obtained by applying $\eta$ to some $U \in L^+G(R')$. Similarly, we obtain $g \in LG(R')$ such that $\eta(g)$ is given by the composition
\begin{align*}
	\Lambda \otimes_{\zz_p} W(R')[1/p] \xrightarrow{\lambda^\tau} (M_{W(R'/pR')^\oplus})^\tau[1/p] \xrightarrow{\gamma} ((M_0)_{W(R'/pR')^\oplus})^\tau[1/p] \xrightarrow{\sim} \Lambda \otimes_{\zz_p} W(R')[1/p],
\end{align*}
after identifying 
\begin{align*}
	\Lambda \otimes_{\zz_p} W(R')[1/p] \cong \Lambda \otimes_{\zz_p} W(R'/pR')[1/p]
\end{align*}
as in the proof of Theorem \ref{prop-rzquotient}. Let $\mathscr{D}_U$ be the banal Tannakian $(G,\mu)$-display obtained from $U$. Then $g$ induces a $G$-quasi-isogeny $\iota_g: (\mathscr{D}_U)_{R/pR} \dashrightarrow (\mathscr{D}_0)_{R/pR}$, and $(\mathscr{D}_U, \iota_g)$ is an object in $\textbf{RZ}_{G,\mu,b}(R')$ whose image under (\ref{eq_RZmorphism}) is isomorphic to $(\underline{M}, \gamma)$.

\end{proof}

Suppose now that all slopes of $\eta(b)$ are nonzero. Then the $1$-display $\underline{M_0}$ obtained by evaluating $\mathscr{D}_0$ on $(\Lambda, \eta)$ is nilpotent. Hence $X_0 = BT_k(\underline{M_0})$ is a formal $p$-divisible group, and both $\underline{M_0}$ and $X_0$ inherit $\mathcal{O}_B$-actions from the action of $\mathcal{O}_B$ on $\Lambda$. Combining Lemma \ref{lem-intermediate} and Theorem \ref{thm-ELcompare}, we obtain the following corollary.
\begin{cor}
	Suppose all slopes of $\eta(b)$ are different from $0$. Then there is a natural isomorphism of functors
	\begin{align*}
		\textup{RZ}_{G,\mu,b} \xrightarrow{\sim} \textup{RZ}_{\textbf{D}}(X_0).
	\end{align*} 
	In particular, $\textup{RZ}_{G,\mu,b}$ is representable by a formal scheme which is formally smooth and formally locally of finite type over \textup{Spf} $W(k)$. 
\end{cor}
This proves Conjecture \ref{conj} in the unramified EL-type case.

\subsection{Remarks on the Hodge-type case}
In this section we specialize the study of our RZ functor to the Hodge-type case. In this case, our Tannakian approach further allows us to prove that the $\textbf{RZ}_{G,\mu,b}$ is a stack in setoids, i.e. that objects in $\textbf{RZ}_{G,\mu,b}(R)$ have no nontrivial automorphisms for any $R$. This is known when $G$ is reductive and $R$ is Noetherian by \cite[Proposition 3.6.1]{BP2017}. %By comparison with the results of \cite{BP2017}, we obtain representability of the restriction of our functor to Noetherian $W(k)$-algebras in the case where $G$ is a reductive group scheme over $\zz_p$. 

\begin{Def}\label{def-hodge}
	A local Shimura datum $(G,\{\mu\},[b])$ is \textit{of Hodge-type} if there exists a faithful representation $\Lambda$ of $G$ on a finite free $\zz_p$-module such that the corresponding closed embedding of group schemes $\eta: G \hookrightarrow \text{GL}(\Lambda)$ satisfies the following property: after a choice of basis $\Lambda_{\mathcal{O}_E} \xrightarrow{\sim} \mathcal{O}_E^n$, the composite
	\begin{align*}
		\eta \circ \mu : \mathbb{G}_{m, \mathcal{O}_E} \to \text{GL}_{n, \mathcal{O}_E}
	\end{align*}
	is the cocharacter $a \mapsto \text{diag}(1^{(r)}, a^{(n-r)})$ for some $1 \le r < n$. 
\end{Def}
\begin{Def}
	A \textit{local Hodge embedding datum} for a local Shimura datum $(G,\{\mu\},[b])$ of  Hodge-type consists of 
	\begin{itemize}
		\item a group scheme embedding $\eta: G \hookrightarrow \text{GL}(\Lambda)$ as above,
		\item a framing pair $(\mu,b)$ for $(G,\{\mu\},[b])$.
	\end{itemize}
\end{Def}

Let $(\Lambda, \eta, b)$ be a local Hodge embedding datum for a local Shimura datum $(G,\{\mu\},[b])$ of Hodge-type. Let $R$ be a $p$-nilpotent $W(k)$-algebra and $(\mathscr{D},\iota)$ be an object in $\textbf{RZ}_{G,\mu,b}(R)$. By evaluating $\mathscr{D}$ on $(\Lambda, \eta)$, we obtain a display $\underline{M}(\eta)$ over $\underline{W}(R)$. Because $\upsilon_R \circ \mathscr{D}$ is fpqc-locally of type $\mu$, it follows from Lemma \ref{lem-localtype} that $\underline{M}(\eta)$ is of type $(0^{(r)}, 1^{(n-r)})$. In particular, it is a 1-display over $\underline{W}(R)$. 

Let $\mathscr{D}_0$ be the framing object associated to $(\mu,b)$, and denote by $\underline{M_0}$ the evaluation of $\mathscr{D}_0$ on $(\Lambda, \eta)$. Then $\iota$ induces a quasi-isogeny of $1$-displays 
\begin{align*}
	\iota_{\eta}: \underline{M}(\eta)_{\underline{W}(R/pR)} \dashrightarrow \underline{M_0}_{\underline{W}(R/pR)}.
\end{align*}
Hence the pair $(\underline{M}(\eta), \iota_\eta)$ determines an object in $\textbf{RZ}(\underline{M_0})(R)$, where $\textbf{\textup{RZ}}(\underline{M_0})$ is the fpqc-stack whose functor of isomorphism classes is $\textup{RZ}(\underline{M_0})$. This construction defines a morphism of stacks
\begin{align*}
	\textbf{RZ}_{G,\mu,b} \to \textup{\textbf{RZ}}(\underline{M_0}).
\end{align*}
If $\eta(b)$ has no slopes equal to zero, then $\underline{M_0}$ is a nilpotent 1-display over $\underline{W}(k)$, so by Proposition \ref{prop-gln}, $\textup{RZ}(\underline{M_0})$ is represented by the classical RZ-space $\textup{RZ}(X_0)$, where $X_0 = BT_k(\underline{M_0})$.

\begin{prop}
	Assume that $(G,\{\mu\},[b])$ is a Hodge-type local Shimura datum with a local Hodge embedding datum such that $\eta(b)$ has no slopes equal to 0. Then $\textbf{\textup{RZ}}_{G,\mu,b}$ is a stack in setoids. 
\end{prop}
\begin{proof}
	Let $R$ be in \textbf{Nilp}$_{W(k)}$ and let $(\mathscr{D},\iota)$ be an object in $\textbf{RZ}_{G,\mu,b}(R)$. Suppose $\psi$ is an automorphism of the pair $(\mathscr{D},\iota)$. Then $\psi$ is a natural transformation of tensor functors $\mathscr{D} \to \mathscr{D}$ which lifts the identity $\mathscr{D}_{R/pR} \to \mathscr{D}_{R/pR}$. Evaluating $\mathscr{D}$ on $(\Lambda, \eta)$, we obtain an automorphism $\psi_\eta$ of  $(\underline{M}(\eta), \iota_\eta)$ in $\textup{\textbf{RZ}}(\underline{M_0})$. But by Proposition \ref{prop-gln}, $\textup{RZ}(\underline{M_0})$ is representable, hence its objects have no nontrivial automorphisms. Then $\psi_\eta = \id_{\underline{M}(\eta)}$. This implies that $\psi = \id_{\mathscr{D}}$, since any endomorphism of an exact tensor functor from Rep$_{\zz_p}(G)$ to an exact rigid tensor category which is the identity on a faithful representation is itself the identity. 
\end{proof}

%(comparison with Kim's functor to be added)

%\appendix
%\input{CommutativeAlgebra.tex}
%\input{FLinearAlgebra.tex}
%
\bibliographystyle{amsalpha}
\bibliography{Refs}

\end{document}